\documentclass[a4paper,11pt]{amsart}
\usepackage{amssymb}
\usepackage{amscd}
\usepackage{comment}
\usepackage{graphics}
\usepackage{amsmath,amsthm}
\usepackage[colorlinks=true]{hyperref}
\usepackage{enumerate}
\usepackage{booktabs,multirow}
\usepackage{tikz}
\usepackage{rotating}
\usetikzlibrary{patterns}
\usetikzlibrary{decorations.pathreplacing}
\usetikzlibrary{calc,through}
\usetikzlibrary{cd}

\allowdisplaybreaks[1]
\numberwithin{figure}{section}
\numberwithin{equation}{section}

\title{Symmetric Dyck tilings, ballot tableaux and tree-like tableaux of shifted shapes}
\author[K.~Shigechi]{Keiichi~Shigechi}
\email{k1.shigechi AT gmail.com}
\date{\today}

\newcommand\tikzpic[2]{
\raisebox{#1\totalheight}{
\begin{tikzpicture}
#2
\end{tikzpicture}
}}
\newcommand\circnum[1]{
\hspace{-3mm}
\tikzpic{-0.22}{[scale=0.7]
\node[draw,circle,inner sep=1pt]at(0,0){$#1$};
}
\hspace*{-3mm}
}
\newtheorem{theorem}[figure]{Theorem}
\newtheorem{example}[figure]{Example}
\newtheorem{lemma}[figure]{Lemma}
\newtheorem{claim}[figure]{Claim}
\newtheorem{defn}[figure]{Definition}
\newtheorem{prop}[figure]{Proposition}
\newtheorem{cor}[figure]{Corollary}

\newtheorem{remark}[figure]{Remark}
\begin{document}

\begin{abstract}
Symmetric Dyck tilings and ballot tilings are certain tilings in the region
surrounded by two ballot paths.
We study the relations of combinatorial objects which are bijective to 
symmetric Dyck tilings such as labeled trees, Hermite histories, and perfect 
matchings.
We also introduce two operations on labeled trees for symmetric Dyck tilings: 
symmetric Dyck tiling strip (symDTS) and symmetric Dyck tiling ribbon (symDTR).
We give two definitions of Hermite histories for symmetric Dyck tilings, and 
show that they are equivalent by use of the correspondence between symDTS operation
and an Hermite history.
Since ballot tilings form a subset in the set of symmetric Dyck tilings,
we construct an inclusive map from labeled trees for ballot tilings to 
labeled trees for symmetric Dyck tilings.
By this inclusive map, the results for symmetric Dyck tilings can be applied 
to those of ballot tilings.
We introduce and study the notions of ballot tableaux and tree-like tableaux of shifted 
shapes, which are generalizations of Dyck tableaux and tree-like tableaux, respectively.
The correspondence between ballot tableaux and tree-like tableaux of shifted shapes 
is given by using the symDTR operation and the structure of labeled trees for 
symmetric Dyck tilings.
\end{abstract}

\maketitle

\tableofcontents

\section{Introduction}
\label{sec:Intro}
Cover-inclusive Dyck tilings were introduced by Zinn-Justin and the author 
in the study of Kazhdan--Lusztig polynomials for Grassmannian permutations in \cite{SZJ12}.
Independently, Kenyon and Wilson introduced them in the study of the double-dimer 
model and spanning trees \cite{KW11,KW15}. 
Since then, Dyck tilings has appeared in connection with different contexts such 
as fully packed loop models \cite{FN12}, multiple Schramm--Loewner evolutions \cite{PelWu19,Pon18}, 
and the intersection cohomology of Grassmannian Schubert varieties \cite{Pa19}. 

Dyck tilings are tilings by Dyck tiles in the region surrounded by two Dyck paths.
The partition function of Dyck tilings with fixed boundary paths is defined 
as the sum of the weights given to Dyck tilings.
The partition function with a fixed lower Dyck path is written as a product 
of $q$-integers (conjectured for $q=1$ case in \cite{KW11} and proved in \cite{K12}).
There, linear extensions of tree posets play a central role to study the partition
function \cite{KMPW12}.

There are several generalizations of Dyck tilings.
The first one is ballot tilings, which are type $B$ analogue of Dyck tilings \cite{S17}.
Ballot tilings are tilings by ballot tiles, which include Dyck tiles, in the
region surrounded by two ballot paths.
Since Dyck paths are also ballot paths, ballot tilings can be regraded as a 
generalization of Dyck tilings. 
The second is symmetric Dyck tilings, which are Dyck tilings with a symmetry along the 
vertical line in the middle \cite{JVK16}.
A symmetric Dyck tiling can also be regarded as a tiling in the region surrounded by 
two ballot paths.
The difference between ballot tilings and symmetric Dyck tilings is that 
the former has conditions on ballot tiles forming a ballot tiling.
The third is $k$-Dyck tilings studied in \cite{JVK16}.
$k$-Dyck tilings are defined by replacing Dyck paths and Dyck tiles by
$k$-Dyck paths and $k$-Dyck tiles.
In this paper, we study the first two generalizations of Dyck tilings: ballot tilings
and symmetric Dyck tilings.
The partition function of ballot tilings was obtained in \cite{S17}, and that of symmetric 
Dyck tilings was obtained in \cite{JVK16}.
As for symmetric Dyck tilings, we introduce a different partition function from \cite{JVK16} 
by choosing simple weights for tilings. 
As a result, the partition function can be written in terms of $q$-integers in a simple form.

The set of symmetric Dyck tilings of size $2n$ is obviously the subset in the set of 
Dyck tilings of size $2n$.
By dividing a symmetric Dyck tiling of size $2n$ into two pieces in the middle,
we have a tiling which we call a fundamental symmetric Dyck tiling of size $n$.
Symmetric Dyck tilings of size $2n$ have top and lower Dyck paths of size $2n$, 
and fundamental symmetric Dyck tilings have top and lower ballot paths of size $n$. 
The ballot tilings are fundamental symmetric Dyck tilings with some conditions on 
ballot tiles.
In ballot tilings of size $n$, we obtain Dyck tilings of size $n$ by 
restricting ourselves to Dyck tiles.
Therefore, we have natural inclusions of these three types of tilings:  
\begin{align*}
\mathrm{DT}_{n}\subset \text{ballot tilings} \subset \text{symmetric Dyck tilings} \subset \mathrm{DT}_{2n}, 
\end{align*}
where $\mathrm{DT}_{n}$ denote the Dyck tilings of size $n$.
Dyck tilings are used to compute maximal parabolic Kazhdan--Lusztig polynomials for Hecke algebras 
of type $A$ in \cite{SZJ12}, and ballot tilings are for type $B$ in \cite{S141}.
The weight of Dyck or ballot tilings to compute Kazhdan--Lusztig polynomials is 
different from the weight of tilings to obtain the partition functions studied in \cite{KW11,K12}.
From the above-mentioned inclusions, it is natural to regard symmetric Dyck tilings as tilings in-between 
type $A$ and type $B$.
Typically, the partition function for symmetric Dyck tilings (see Theorem \ref{thrm:pfsymDT})
looks similar to both the one for Dyck tilings (see Theorem \ref{thrm:DT})
and the one for ballot tilings (see Theorem \ref{thrm:pfbt}).
Actually, (fundamental) symmetric Dyck tilings have properties similar to both type $A$ and 
type $B$ as explained below.

There are several combinatorial objects which are bijective to Dyck tilings.
Main objects appeared in the analysis of Dyck tilings are the following
five combinatorial objects: labeled trees, 
Hermite histories, perfect matchings, Dyck tableaux and tree-like tableaux. 
First three notions appeared in the study of Dyck tilings \cite{KMPW12}, 
and the last two notions appeared in \cite{ABDH11} and \cite{ABN11}.
In this paper, we utilize these five combinatorial objects for the study of 
symmetric Dyck tilings and ballot tilings.

As mentioned above, ballot tilings are included in symmetric Dyck tilings.
However, labeled trees for ballot tilings and symmetric Dyck tilings have 
different structures due to the fact that labeled trees encode information
about tilings.
Especially, when we regard a ballot tiling as a symmetric Dyck tiling,
the labeled trees are different even though we have the same tiling.
To solve this discrepancy, we introduce an inclusive map from a labeled tree 
for a ballot tiling to a labeled tree for a symmetric Dyck tiling such that 
the underlying two tilings are the same (see Section \ref{sec:symDTSbt}).
The image of the inclusive map is a subset in  
the set of fundamental symmetric Dyck tilings.
Then, one can apply the results for symmetric Dyck tilings to the ones 
for ballot tilings by this inclusive map.
Since this inclusive map preserves the properties as type $B$, 
the subset of fundamental symmetric Dyck tilings has also the properties
as type $B$.

In addition, we have two operations on labeled trees to obtain symmetric 
Dyck tilings: symmetric Dyck tilings strip (symDTS) and 
symmetric Dyck tiling ribbon (symDTR).
Similarly, we have two operations on labeled trees to obtain ballot tilings,
which are a composition of the inclusive map from ballot tilings to symmetric
Dyck tilings and symDTS/symDTR.
These operations on labeled trees are a natural generalizations of 
Dyck tiling strip and Dyck tiling ribbon bijections studied in \cite{KMPW12}.
We study symmetric Dyck tilings and ballot tilings by combining the five 
combinatorial objects mentioned above and these two bijections.

We generalize the notions of labeled trees, Hermite histories and 
perfect matchings in \cite{KMPW12} to the cases of symmetric Dyck tilings and 
ballot tilings, 
and study the relations among them.
We need to implement some extra structures in 
Hermite histories and perfect matchings since labeled trees for symmetric 
Dyck tilings and ballot tilings have also some extra information 
compared to Dyck tilings.
Acturally, although the size of a perfect matching is equal to the one of 
a (non-symmetric) Dyck tiling, the size of a perfect matching is possibly 
larger than the one of a symmetric Dyck tiling.
This discrepancy of the sizes reflects the vertical symmetry of the tiling.

We give two definitions of Hermite histories for symmetric Dyck tiligs.
One is the definition by using the nesting numbers of a perfect matching, 
and the other is the one by using the number of tiles in a symmetric 
Dyck tiling.
One of the main results is that the two definitions of Hermite histories 
coincide with each other (see Theorem \ref{thrm:Hh}).
Given the Hermite history of a symmetric Dyck tiling $\mathcal{D}$,
one can define a word $w_1$ which encodes the numbers of tiles associated 
to up steps in the lower boundary path.
Reversely, given a word $w_1$ and a lower boundary path, one can 
reconstruct $\mathcal{D}$ by using the Hermite history.
Similarly, given $\mathcal{D}$, one can define a word $w_2$ by the 
symmetric Dyck tiling strip. 
Then, a pair of $w_2$ and a lower boundary path is bijective to 
$\mathcal{D}$.
Thus, we have two words $w_1$ and $w_2$ for a symmetric Dyck tiling $\mathcal{D}$.
These two words are different in general. 
The difference comes from the existence of non-trivial Dyck tiles in $\mathcal{D}$. 
We give two explicit maps from the word $w_2$ to the word $w_1$.  
The first map is obtained by use of the tree structure associated with $\mathcal{D}$.
The second map is by use of the set encoding the information about non-trivial Dyck 
tiles in $\mathcal{D}$. 
These maps play a central role to prove Theorem \ref{thrm:Hh}.
The correspondence among labeled trees, Hermite histories and perfect matchings 
is a property reminiscent of the correspondence in type $A$.

Originally, the notions of Dyck tableaux and tree-like tableaux 
are introduced for Dyck tilings with the zig-zag lower boundary path, or 
equivalently for a permutation \cite{ABDH11,ABN11}.
The notions of Dyck tableaux and tree-like tableaux are generalized to 
the generic lower boundary path, or equivalently for labeled trees of generic shape 
in \cite{S19}.
As for symmetric Dyck tilings, in this paper, we introduce the notions of ballot tableaux 
and tree-like tableaux of shifted shapes.
One of the main objects in this paper is the introduction of these ballot tableaux
and tree-like tableaux of shifted shapes.
They are generalizations of Dyck tableaux in \cite{ABDH11,S19} and tree-like tableaux
in \cite{ABN11,S19}.
In \cite{ABN11}, tree-like tableaux for symmetric Dyck tilings are introduced and 
they can also be regarded as tableaux of shifted shapes by cutting them along the main 
diagonal.
The tableaux are associated with symmetric Dyck tilings whose lower boundary path 
is a zig-zag path.
The advantage of our construction of tree-like tableaux of shifted shapes
is that we have tableaux for labeled trees of generic shapes, in other words,
for general lower boundary ballot paths.
This is realized by relaxing a condition to construct tree-like tableaux, namely,
by introducing diagonal insertions.
We no longer have a simple correspondence between tree-like tableaux and 
binary trees as observed in \cite{ABN11} due to diagonal insertions.
However, tree-like tableaux of shifted shapes still possess the properties
similar to ballot tableaux.

Ballot tableaux and tree-like tableaux of shifted shapes have similar 
properties to the symmetric DTR bijection rather than the symmetric DTS 
bijection.
Since the symmetric DTR bijection has an insertion algorithm, 
we introduce insertion procedures for ballot tableaux and tree-like 
tableaux of shifted shapes.
Then, we show that there exists a bijection between ballot tableaux and tree-like 
tableaux by use of the symmetric DTR bijection.
We also study the enumeration of tree-like tableaux of shifted shapes and show 
that the enumeration is given by an integer sequence with a simple recurrence 
relation.

The paper is organized as follows.
In section 2, we introduce symmetric Dyck tilings and ballot tilings, 
and define the partition functions for these tilings.
In section 3, we give definitions of labeled trees for both symmetric 
Dyck tilings and ballot tilings.
We give a description of the partition functions in terms of the trees.
In section 4, we introduce the operation on symmetric Dyck tilings
and ballot tilings, which is called symmetric Dyck tiling strip.
This operation acting on symmetric Dyck tilings and ballot tilings 
gives the correspondence between labeled trees and the tilings.
We also give an explicit inclusion from ballot tilings to symmetric Dyck 
tilings. 
In section 5, we construct a bijection among symmetric Dyck tilings, 
Hermite histories, and perfect matchings. 
To study these combinatorial objects, we make use of the symDTS bijection and 
insertion histories.
These are a generalization of the results for Dyck tiling in \cite{S19}. 
In section 6, we introduce another operation called symmetric Dyck 
tiling ribbon acting on labeled trees.
In sections 7 and 8, we introduce and study the notions of ballot tableaux
and tree-like tableaux of shifted shapes.
The operation symDTR introduced in section 6 is a key tool to study ballot tableaux and tree-like 
tableaux for shifted shapes.

\paragraph{\bf Notations}
We introduce the quantum integer $[n]:=\sum_{i=0}^{n-1}q^{i}$,  
quantum factorial $[n]!:=\prod_{i=1}^{n}[i]$, $[2m]!!:=\prod_{i=1}^{m}[2i]$, 
and the $q$-analogue of the binomial coefficients
\begin{align*}
\genfrac{[}{]}{0pt}{}{n}{m}:=\frac{[n]!}{[n-m]![m]!}, 
\qquad
\genfrac{[}{]}{0pt}{}{n}{m}_{q^{2}}:=\frac{[2n]!!}{[2(n-m)]!!\cdot[2m]!!}.
\end{align*}

\section{Symmetric Dyck tilings and ballot tilings}
\subsection{Dyck tilings}
Dyck paths of size $n$ are defined as lattice paths starting from $(0,0)$ 
to $(2n,0)$, which are not below the horizontal line $y=0$ in the Cartesian coordinate.  
Each step of a Dyck path is either $(+1,+1)$ (an up or ``$U$" step) or 
$(+1,-1)$ (a down or ``$D$" step).
We call the Dyck path $(UD)^{n}$ a zig-zag Dyck path or simply a zig-zag path.

A Dyck tile (also called ``Dyck strip") is a ribbon (connected skew shapes which do not 
contain a $2\times2$ rectangle) such that the centers of boxes forming the ribbon form 
a Dyck path.
The size of a Dyck tile is the size of underlining Dyck path plus one.
Therefore, the size of a single box is one.

We associate a Dyck path $\lambda$ with a Young diagram formed by boxes 
in the region specified by $\lambda$, the line $y=x$ and the line $y=-x+2n$. 
Let $\lambda$ and $\mu$ be Dyck paths. 
When the skew Young diagram $\lambda/\mu$ exists, we say that $\mu$ is above $\lambda$
or $\lambda$ is below $\mu$, and denote it by $\lambda\le \mu$.

A Dyck tiling is a tiling of a skew Young diagram $\lambda/\mu$ by Dyck tiles. 
A Dyck tiling $D$ is called cover-inclusive if we translate a Dyck tile of $D$ 
downward by $(0,-2)$, then it is strictly below $\lambda$ or contained in 
another Dyck tile.
In this paper, since we consider only cover-inclusive Dyck tiles, we simply 
denote it by Dyck tilings.
Some examples of Dyck tilings are in Figure \ref{fig:DyckTilings}.
\begin{figure}[ht]
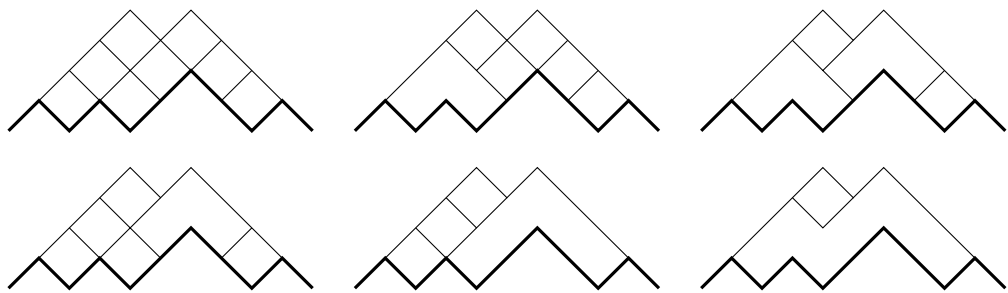

\tikzpic{-0.5}{[x=0.4cm,y=0.4cm]
\draw[very thick](0,0)--(1,1)--(2,0)--(3,1)--(4,0)--(6,2)--(8,0)--(9,1)--(10,0);
\draw(1,1)--(4,4)--(5,3)--(6,4)--(9,1);
\draw(2,2)--(3,1)--(5,3)(3,3)--(5,1)(5,3)--(6,2)(6,2)--(7,3)(7,1)--(8,2);
}
\tikzpic{-0.5}{[x=0.4cm,y=0.4cm]
\draw[very thick](0,0)--(1,1)--(2,0)--(3,1)--(4,0)--(6,2)--(8,0)--(9,1)--(10,0);
\draw(1,1)--(4,4)--(5,3)--(6,4)--(9,1);
\draw(4,2)--(5,3)(3,3)--(5,1)(5,3)--(6,2)(6,2)--(7,3)(7,1)--(8,2);
}
\tikzpic{-0.5}{[x=0.4cm,y=0.4cm]
\draw[very thick](0,0)--(1,1)--(2,0)--(3,1)--(4,0)--(6,2)--(8,0)--(9,1)--(10,0);
\draw(1,1)--(4,4)--(5,3)--(6,4)--(9,1);
\draw(4,2)--(5,3)(3,3)--(5,1)(7,1)--(8,2);
}
\\[12pt]
\tikzpic{-0.5}{[x=0.4cm,y=0.4cm]
\draw[very thick](0,0)--(1,1)--(2,0)--(3,1)--(4,0)--(6,2)--(8,0)--(9,1)--(10,0);
\draw(1,1)--(4,4)--(5,3)--(6,4)--(9,1);
\draw(2,2)--(3,1)--(5,3)(3,3)--(5,1)(7,1)--(8,2);
}
\tikzpic{-0.5}{[x=0.4cm,y=0.4cm]
\draw[very thick](0,0)--(1,1)--(2,0)--(3,1)--(4,0)--(6,2)--(8,0)--(9,1)--(10,0);
\draw(1,1)--(4,4)--(5,3)--(6,4)--(9,1);
\draw(2,2)--(3,1)--(5,3)(3,3)--(4,2);
}
\tikzpic{-0.5}{[x=0.4cm,y=0.4cm]
\draw[very thick](0,0)--(1,1)--(2,0)--(3,1)--(4,0)--(6,2)--(8,0)--(9,1)--(10,0);
\draw(1,1)--(4,4)--(5,3)--(6,4)--(9,1);
\draw(3,3)--(4,2)--(5,3);
}
\caption{Dyck tilings of ske shape $\lambda/\mu$ with $\lambda=UDUDUUDDUD$ and 
$\mu=UUUUDUDDDD$.}
\label{fig:DyckTilings}
\end{figure}

Let $\mathcal{D}(\lambda)$ be the set of Dyck tilings above $\lambda$.
Given a Dyck tiling $D\in\mathcal{D}(\lambda)$, 
we define the weight of a Dyck tile $d$ of size $n_{d}$ in $D$ by
\begin{align*}
\mathrm{wt}(d):=n_{d}.
\end{align*}
The weight of a Dyck tiling $D$ is the sum of the weights of Dyck tiles 
consisting of $D$.
The partition function $P(\lambda)$ of Dyck tilings above $\lambda$ is defined 
by
\begin{align*}
P(\lambda):=
\sum_{D\in\mathcal{D}(\lambda)}\sum_{d\in D}\mathrm{wt}(d).
\end{align*}

We introduce a {\it chord} of a Dyck path and its length.
We make a pair of $U$ and $D$ which are next to each other in $\lambda$.
We continue making a pair of $U$ and $D$ by ignoring already paired $U$ 
and $D$. 
Note that a Dyck path of length $2n$ consists of $n$ pairs of $U$ and $D$.
We call this pair of $U$ and $D$ a chord of the Dyck path $\lambda$.
Let $\Lambda_{c}(\lambda)$ be the set of chords in the Dyck path $\lambda$.
The length of a chord $c$ is one plus the number of chords between $U$ 
and $D$ in $c$.
The length of the chord $c$ is denoted by $l(c)$.

\begin{theorem}[Conjecture 1 in \cite{KW11} and Theorem 1.1 in \cite{K12}]
\label{thrm:DT}
The partition function associated with the Dyck path $\lambda$ is expressed 
in terms of $q$-integers:
\begin{align*}
P(\lambda)=\genfrac{}{}{1pt}{}{[n]!}{\prod_{c\in\Lambda_{c}(\lambda)}[l(c)]}.
\end{align*}
\end{theorem}

\subsection{Ballot tile}
For non-negative integers $n$ and $n'$, 
a ballot path of size $(n,n')$ are defined as a lattice path 
starting from $(0,0)$ to $(2n+n',n')$ which are not below
the horizontal line $y=0$.
When $n'=0$, a ballot path is nothing but a Dyck path.
Each step of a ballot path is either a $U$ step or a $D$ step.
A ballot tile of size $(n,n')$ is a ribbon such that the centers 
of boxes forming the ribbon form a ballot path of size $(n,n')$.
Note that we do not add one to define the size of a ballot tile
(see the definition of the size of a Dyck tile in the previous subsection).

\subsection{Symmetric Dyck tilings}
Let $\lambda$ be a Dyck path of size $n$ and symmetric along the 
vertical line $x=n$.
A symmetric Dyck tiling over $\lambda$ is a Dyck tiling 
that is symmetric along the vertical line $x=n$.
The left picture of Figure \ref{fig:SymDTDT} is an example 
of a symmetric Dyck tiling over $UDUDUUUDUDDDUDUD$.
\begin{figure}[ht]
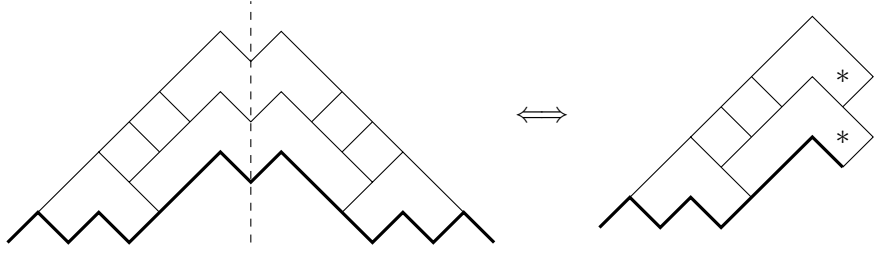

\tikzpic{-0.5}{[x=0.4cm,y=0.4cm]
\draw[very thick](0,0)--(1,1)--(2,0)--(3,1)--(4,0)--(7,3)--(8,2)--
	(9,3)--(12,0)--(13,1)--(14,0)--(15,1)--(16,0);
\draw(1,1)--(3,3)--(5,1)(4,2)--(7,5)--(8,4)--(9,5)--(12,2)
	(11,1)--(13,3)--(15,1)(3,3)--(7,7)--(8,6)--(9,7)--(13,3);
\draw(4,4)--(5,3)(5,5)--(6,4)(12,4)--(11,3)(11,5)--(10,4);
\draw[dashed](8,0)--(8,8);
}
$\Longleftrightarrow$
\tikzpic{-0.5}{[x=0.4cm,y=0.4cm]
\draw[very thick](0,0)--(1,1)--(2,0)--(3,1)--(4,0)--(7,3)--(8,2);
\draw(1,1)--(3,3)--(5,1)(3,3)--(7,7)--(9,5)--(8,4);
\draw(4,4)--(5,3)(5,5)--(6,4);
\draw(4,2)--(7,5)--(9,3)--(8,2);
\draw(8,3)node{$\ast$}(8,5)node{$\ast$};
}
\caption{A symmetric Dyck tiling}
\label{fig:SymDTDT}
\end{figure}

Given a symmetric Dyck tiling, we cut it along the line $x=n$ and put 
$\ast$ on the boxes positioned at the line $x=n$.
Then, a diagram obtained from a symmetric Dyck tiling by a cut 
is called a {\it fundamental} symmetric Dyck tiling.
The right picture of Figure \ref{fig:SymDTDT} is an example of a fundamental 
symmetric Dyck tiling.

\begin{remark}
\label{remark:bijfsD}
It is obvious that there exists the bijection between symmetric Dyck tilings
and fundamental symmetric Dyck tilings. 
Thus, we sometimes call a fundamental symmetric Dyck tiling simply a symmetric 
Dyck tiling.
\end{remark}

Figure \ref{fig:SymDT} is an example of fundamental symmetric Dyck tilings 
over the path $UDUU$.
\begin{figure}[ht]
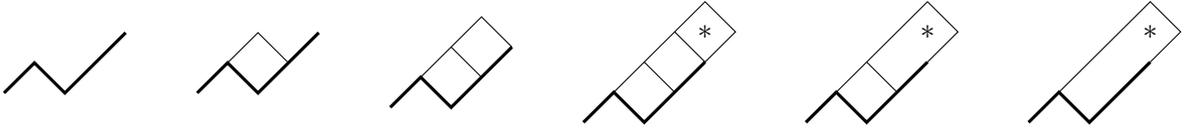

\tikzpic{-0.5}{[x=0.4cm,y=0.4cm]
\draw[very thick](0,0)--(1,1)--(2,0)--(4,2);
}
\quad
\tikzpic{-0.5}{[x=0.4cm,y=0.4cm]
\draw[very thick](0,0)--(1,1)--(2,0)--(4,2);
\draw(1,1)--(2,2)--(3,1);
}
\quad
\tikzpic{-0.5}{[x=0.4cm,y=0.4cm]
\draw[very thick](0,0)--(1,1)--(2,0)--(4,2);
\draw(1,1)--(3,3)--(4,2)(2,2)--(3,1);	
}
\quad
\tikzpic{-0.5}{[x=0.4cm,y=0.4cm]
\draw[very thick](0,0)--(1,1)--(2,0)--(4,2);
\draw(1,1)--(4,4)--(5,3)--(4,2)(2,2)--(3,1)(3,3)--(4,2);
\draw(4,3)node{$\ast$};
}
\quad
\tikzpic{-0.5}{[x=0.4cm,y=0.4cm]
\draw[very thick](0,0)--(1,1)--(2,0)--(4,2);
\draw(1,1)--(4,4)--(5,3)--(4,2);
\draw(2,2)--(3,1);
\draw(4,3)node{$\ast$};
}
\quad
\tikzpic{-0.5}{[x=0.4cm,y=0.4cm]
\draw[very thick](0,0)--(1,1)--(2,0)--(4,2);
\draw(1,1)--(4,4)--(5,3)--(4,2);
\draw(4,3)node{$\ast$};
}
\caption{Symmetric Dyck tilings over the path $UDUU$.}
\label{fig:SymDT}
\end{figure}

We denote by $\mathcal{D}^{\mathrm{sym}}$ the set of 
fundamental symmetric Dyck tilings.
We define the following weights for a ballot tile $d$.
\begin{enumerate}
\item When $d$ is a Dyck tile of size $n$, the weight of $d$ is $n$.
\item When $d$ is a ballot tile of size $(n,n')$ with $n\ge0$ and $n'\ge1$, 
the weight of $d$ (denoted by $\mathrm{wt}(d)$) is given by 
\begin{align}
\label{eqn:WTSymDT}
\mathrm{wt}(d):=
\begin{cases}
1, & \text{for } n=0, \\ 
n+1, & \text{for } n\ge1.
\end{cases}
\end{align}
\end{enumerate}

We define the partition function for fundamental symmetric Dyck tilings as follows. 
\begin{defn}
The partition function $P^{\mathrm{sym}}(\lambda)$ of fundamental symmetric Dyck 
tilings above $\lambda$ is defined by
\begin{align*}
P^{\mathrm{sym}}(\lambda):=
\sum_{D\in\mathcal{D}^{\mathrm{sym}}}
\sum_{d\in D}q^{\mathrm{wt}(d)}.
\end{align*}
\end{defn}

\begin{example}
We consider symmetric Dyck tilings above $\lambda=UDUU$ (see Figure \ref{fig:SymDT}).
The weights of the symmetric Dyck tilings are 
$1, q, q^2, q^3, q^2$ and $q$ from left to right.
The partition function is given by $(1+q+q^{2})(1+q)$.
\end{example}

\subsection{Ballot tilings}
\label{sec:bt}
In this subsection, we introduce the notion of ballot tilings 
which form a subset in the set of symmetric Dyck tilings.

\begin{defn}
\label{defn:bt}
Let $\lambda$ be a ballot path of size $(n,n')$.
A ballot tiling of size $n+n'$ above $\lambda$ 
is a fundamental symmetric Dyck tiling of size $n+n'$ above $\lambda$
satisfying the following constraints.
\begin{enumerate}
\item The rightmost box of a ballot tile of size $(m,m')$  
contains $\ast$ if it is on the vertical line $x=2n+n'$. 
\item The number of ballot tiles of size $(m,m')$ with $m'\ge1$ is 
even when $m'$ is odd, and zero when $m'$ is even.
\end{enumerate}
\end{defn}
Note that a Dyck tile whose right-most box $b$ is on the line $x=2n+n'$
has $\ast$ on $b$, since a Dyck tile is a ballot tile of size $(m,0)$.

\begin{example}
In Figure \ref{fig:BallotTilings}, we list up all ballot tilings 
above the ballot path $UDUDU$.
The last ballot tiling contains two ballot tiles of size $(0,1)$ 
and satisfies the second condition in Definition \ref{defn:bt}.
\begin{figure}[ht]
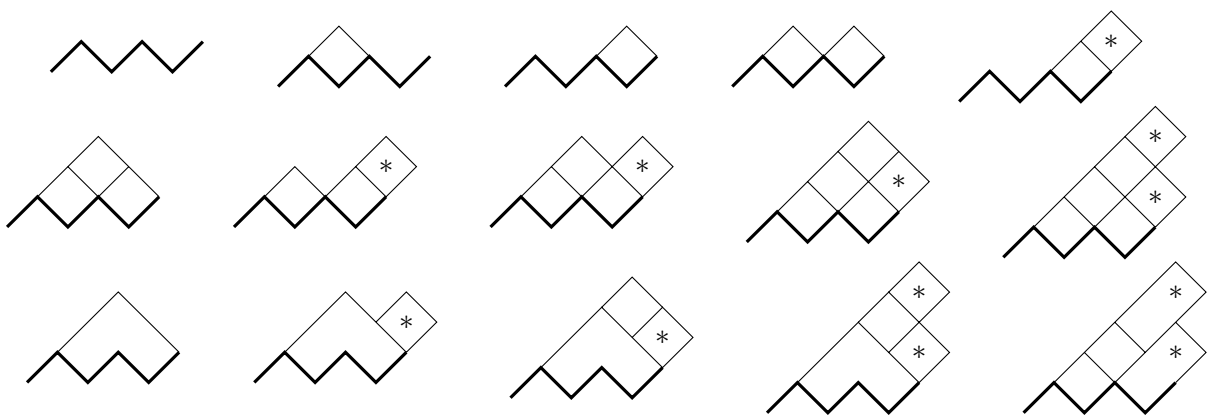

\tikzpic{-0.5}{[x=0.4cm,y=0.4cm]
\draw[very thick](0,0)--(1,1)--(2,0)--(3,1)--(4,0)--(5,1);
}
\quad
\tikzpic{-0.5}{[x=0.4cm,y=0.4cm]
\draw[very thick](0,0)--(1,1)--(2,0)--(3,1)--(4,0)--(5,1);
\draw(1,1)--(2,2)--(3,1);
}
\quad
\tikzpic{-0.5}{[x=0.4cm,y=0.4cm]
\draw[very thick](0,0)--(1,1)--(2,0)--(3,1)--(4,0)--(5,1);
\draw(3,1)--(4,2)--(5,1);
}
\quad
\tikzpic{-0.5}{[x=0.4cm,y=0.4cm]
\draw[very thick](0,0)--(1,1)--(2,0)--(3,1)--(4,0)--(5,1);
\draw(1,1)--(2,2)--(3,1)--(4,2)--(5,1);
}
\quad
\tikzpic{-0.5}{[x=0.4cm,y=0.4cm]
\draw[very thick](0,0)--(1,1)--(2,0)--(3,1)--(4,0)--(5,1);
\draw(3,1)--(5,3)--(6,2)--(4,0)(4,2)--(5,1);
\draw(5,2)node{$\ast$};
}\\
\tikzpic{-0.5}{[x=0.4cm,y=0.4cm]
\draw[very thick](0,0)--(1,1)--(2,0)--(3,1)--(4,0)--(5,1);
\draw(1,1)--(3,3)--(5,1)(2,2)--(3,1)--(4,2);
}
\quad
\tikzpic{-0.5}{[x=0.4cm,y=0.4cm]
\draw[very thick](0,0)--(1,1)--(2,0)--(3,1)--(4,0)--(5,1);
\draw(1,1)--(2,2)--(3,1)--(5,3)--(6,2)--(5,1)(4,2)--(5,1);
\draw(5,2)node{$\ast$};
}
\quad
\tikzpic{-0.5}{[x=0.4cm,y=0.4cm]
\draw[very thick](0,0)--(1,1)--(2,0)--(3,1)--(4,0)--(5,1);
\draw(1,1)--(3,3)--(5,1)(2,2)--(3,1)--(5,3)--(6,2)--(5,1);
\draw(5,2)node{$\ast$};
}
\quad
\tikzpic{-0.5}{[x=0.4cm,y=0.4cm]
\draw[very thick](0,0)--(1,1)--(2,0)--(3,1)--(4,0)--(5,1);
\draw(1,1)--(3,3)--(5,1)(2,2)--(3,1)--(5,3)--(6,2)--(5,1)(3,3)--(4,4)--(5,3);
\draw(5,2)node{$\ast$};
}
\quad
\tikzpic{-0.5}{[x=0.4cm,y=0.4cm]
\draw[very thick](0,0)--(1,1)--(2,0)--(3,1)--(4,0)--(5,1);
\draw(1,1)--(3,3)--(5,1)(2,2)--(3,1)--(5,3)--(6,2)--(5,1)
     (3,3)--(4,4)--(5,3)(4,4)--(5,5)--(6,4)--(5,3);
\draw(5,2)node{$\ast$}(5,4)node{$\ast$};
}\\
\quad
\tikzpic{-0.5}{[x=0.4cm,y=0.4cm]
\draw[very thick](0,0)--(1,1)--(2,0)--(3,1)--(4,0)--(5,1);
\draw(1,1)--(3,3)--(5,1);
}
\quad
\tikzpic{-0.5}{[x=0.4cm,y=0.4cm]
\draw[very thick](0,0)--(1,1)--(2,0)--(3,1)--(4,0)--(5,1);
\draw(1,1)--(3,3)--(5,1)(4,2)--(5,3)--(6,2)--(5,1);
\draw(5,2)node{$\ast$};
}
\quad
\tikzpic{-0.5}{[x=0.4cm,y=0.4cm]
\draw[very thick](0,0)--(1,1)--(2,0)--(3,1)--(4,0)--(5,1);
\draw(1,1)--(3,3)--(5,1)(4,2)--(5,3)--(6,2)--(5,1)
     (3,3)--(4,4)--(5,3);
\draw(5,2)node{$\ast$};
}
\quad
\tikzpic{-0.5}{[x=0.4cm,y=0.4cm]
\draw[very thick](0,0)--(1,1)--(2,0)--(3,1)--(4,0)--(5,1);
\draw(1,1)--(3,3)--(5,1)(4,2)--(5,3)--(6,2)--(5,1)
     (3,3)--(4,4)--(5,3)(4,4)--(5,5)--(6,4)--(5,3);
\draw(5,2)node{$\ast$}(5,4)node{$\ast$};
}
\quad
\tikzpic{-0.5}{[x=0.4cm,y=0.4cm]
\draw[very thick](0,0)--(1,1)--(2,0)--(3,1)--(4,0)--(5,1);
\draw(1,1)--(5,5)--(6,4)--(3,1)(2,2)--(3,1)(3,3)--(4,2)(5,3)--(6,2)--(5,1);
\draw(5,2)node{$\ast$}(5,4)node{$\ast$};
}
\caption{Ballot tilings over the path $UDUDU$.}
\label{fig:BallotTilings}
\end{figure}
\end{example}

\begin{remark}
\label{rmk:KL}
Some remarks are in order.
\begin{enumerate}
\item
Ballot tilings first appeared in the study of the Kazhdan--Lusztig polynomials 
associated with the Hermitian symmetric pair $(B_{n},A_{n-1})$ \cite{S141}.
Ballot tilings are a generalization of Dyck tilings.
If we restrict ourselves to ballot tilings above $\lambda$ of size $(n,n')$, 
whose boxes have no $\ast$, we simply have the set of Dyck tilings 
of size $n+n'$.
\item
In this paper, we consider the two generalizations of Dyck tilings: symmetric 
Dyck tilings and ballot tilings.
As already mentioned in Section \ref{sec:Intro}, ballot tilings can be regarded 
as a fundamental symmetric Dyck tiling with conditions on ballot tiles.
We study fundamental symmetric Dyck tilings in details, and apply their results to 
ballot tilings by introducing the inclusive map from ballot tilings to 
fundamental symmetric Dyck tilings (see Section \ref{sec:symDTSbt}).
\item
There are several studies on other generalizations of Dyck tilings as follows.
In \cite{JVK16}, $k$-Dyck tilings and symmetric Dyck tilings are defined and studied. 
Recall that the number of Dyck paths of size $n$ is given by the Catalan number and 
that of $k$-Dyck paths is by Fuss-Catalan number.
Thus, a generalization from Dyck tilings to $k$-Dyck tilings is associated with 
a generalization from the Catalan number to the Fuss-Catalan number.
As for symmetric Dyck tilings, the weight of a tiling given in \cite{JVK16} 
is different from the one in Eqn. (\ref{eqn:WTSymDT}).
Our choice of the weight gives a simple formula for the partition function 
of fundamental symmetric Dyck tilings (see Section \ref{sec:EnusymDT}).
Symmetric Dyck tilings, whose lower boundary path is restricted to a zig-zag path, 
are studied in \cite{ABN11} in the context of symmetric tree-like tableaux.
In our construction of fundamental symmetric Dyck tilings and tree-like tableaux 
of shifted shapes, we can consider not only zig-zag paths but also general ballot 
paths.
Tree-like tableaux in Section \ref{sec:TlTshifted} can be viewed as a natural 
generalization of tree-like tableaux studied in \cite{ABN11}.
\end{enumerate}
\end{remark}

We denote by $\mathcal{B}(\lambda)$ the set of ballot tilings above a ballot path
$\lambda$.
We define the weight for a balot tile $b$ in $B\in\mathcal{B}(\lambda)$
as follows.
\begin{enumerate}
\item When $b$ is a Dyck tile of size $(n,0)$, the weight of $b$ is n.
\item When $b$ is a ballot tile of size $(n,n')$, the weight 
of $b$ is given by $n+(n'+1)/2$ for $n\ge0$ and $n'\ge 1$ odd.
\end{enumerate}
Note that the weight is non-negative since $n'$ is always odd for a ballot 
tile.
We denote by $\mathrm{wt}(b)$ the weight of a ballot tile $b$.
The partition function for ballot tilings above $\lambda$ is defined 
as 
\begin{align*}
P^{\mathrm{bal}}(\lambda):=
\sum_{B\in\mathcal{B}(\lambda)}\sum_{b\in B}q^{\mathrm{wt}(b)}.
\end{align*}

\begin{example}
Consider a ballot tilngs above $\lambda=UDUDU$ (see Figure \ref{fig:BallotTilings}).
The last ballot tiling contains two ballot tiles of weight one and two single boxes.
The weight of this ballot tiling is four.
The partition function is $P^{\mathrm{bal}}(\lambda)=[3][5]$. 
\end{example}

Let $\lambda$ be a symmetric Dyck path of length $2n$, and 
$\lambda'$ be a ballot path obtained by taking the first $n$
steps of $\lambda$.
The sets of Dyck, symmetric Dyck or ballot tilings above $\lambda$ or $\lambda'$
satisfy 
\begin{align*}
\mathcal{D}(\lambda')\subset \mathcal{B}(\lambda')
\subset\mathcal{D}^{\mathrm{sym}}(\lambda')\subset\mathcal{D}(\lambda).
\end{align*}
Note that in $\mathcal{D}^{\mathrm{sym}}(\lambda')\subset\mathcal{D}(\lambda)$,
we identify a fundamental symmetric Dyck tiling of size $n$ with 
a symmetric Dyck tiling of size $2n$ as in Remark \ref{remark:bijfsD}.

\begin{remark}
The set $\mathcal{D}^{\mathrm{sym}}(\lambda')$ with a zig-zag path $\lambda'$ 
is studied in \cite{ABN11} in the context of symmetric tree-like tableaux 
(see Remark \ref{rmk:KL}).
There, when $\lambda'$ is a zig-zag path, the cardinality of the 
set $\mathcal{D}^{\mathrm{sym}}(\lambda')$ is $2^{n}n!$, which is equal to 
the cardinality of the Weyl group of type $B$.
On the other hand, $\mathcal{B}(\lambda')$ naturally appears as type $B$ analogue of 
the Dyck tilings in the context of Kazhdan--Lusztig polynomials for an Hermite symmetric
pair \cite{S141} (see also Remark \ref{rmk:KL}) and Temperley--Lieb Hamiltonian with boundaries studied in 
\cite{S142}. Therefore, it may be natural to view $\mathcal{B}(\lambda')$ as type $B$ 
analogue of $\mathcal{D}(\lambda)$.
The symmetric Dyck tilings $\mathcal{D}^{\mathrm{sym}}(\lambda')$ possesses 
the properties of type $A$ and type $B$ at the same time.
By using the inclusion $\mathcal{B}(\lambda')\subset\mathcal{D}^{\mathrm{sym}}(\lambda')$, 
we embed the set of the ballot tilings in the set of symmetric Dyck tilings (see Section \ref{sec:symDTSbt}).
This embedding plays a central role when we construct the ballot tableaux and 
generalized tree-like tableaux corresponding to $\mathcal{B}(\lambda')$. 
\end{remark}

\section{Trees for symmetric Dyck tilings and ballot tilings}
In this section, we introduce trees with increasing labels from the root to leaves 
for symmetric Dyck tilings and ballot tilngs.
In both cases, the lower boundary path is a ballot path.

In this paper, we consider only planted plane trees. 
A tree consists of the root, nodes, and the edges.
A node has several nodes just below it and they are connected by edges.
Each node has at most one node which is above it.
Thus, if we consider a path from a node to the root, we have a unique path 
connecting from the root to the node.

We put positive integers on edges of a tree $T$ and we call them 
a label of $T$.

When an edge $E'$ is strictly right (resp. left) to the edge $E$,
we denote this relation by $E\rightarrow E'$ (resp. $E'\leftarrow E$).
Suppose that we pass through the edge $E'$ when we consider a path 
from the edge $E$ to the root. Then, we say that the edge $E'$ is 
above $E$ and denote this relation by $E\uparrow E'$.

We follow the construction of trees in \cite{Boe88,LS81}.

\subsection{Trees for symmetric Dyck tilings}
\label{sec:treesforsymDT}
Let $\lambda$ be a ballot path of size $(n,n')$. 
A concatenation of $\lambda$ and $\underbrace{D\ldots D}_{n'}$ gives 
a Dyck path of size $n+n'$.
We denote by $\overline{\lambda}$ a Dyck path obtained from the ballot path 
$\lambda$ by the addition of $D$'s.

We define a tree associated with $\overline{\lambda}$ as follows.
\begin{enumerate}
\item When $\overline{\lambda}=\emptyset$, we have an empty tree. 
\item Suppose that $\overline{\lambda}$ is a concatenation of two Dyck 
paths $\lambda_{1}$ and $\lambda_{2}$, {\it i.e.,} $\overline{\lambda}=\lambda_{1}\circ\lambda_{2}$.
Then, the tree for $\overline{\lambda}$ is obtained by attaching the two trees for $\lambda_{1}$
and $\lambda_{2}$ at their roots.
\item Let $\mu$ be a Dyck path. When $\overline{\lambda}=U\mu D$, we attach an edge just above the 
the tree for $\mu$.
If the $D$ step in $U\mu D$ is the $D$ step added to construct $\overline{\lambda}$,
we put a dot $\bullet$ on the corresponding edge.
\end{enumerate}

For example, when $\lambda=UDUUDU$, the tree is depicted as below.
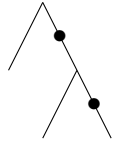
\begin{align*}
\tikzpic{-0.5}{[scale=0.6]
\coordinate
	child{coordinate(c1)}
	child{coordinate(c2)
		child{coordinate(c3)}
		child{coordinate(c4)}
	     };
\node at ($(0,0)!0.5!(c2)$){$\bullet$};
\node at ($(c2)!0.5!(c4)$){$\bullet$};
}
\end{align*}

We denote the set of edges of the tree for $\lambda$ by 
$\mathcal{E}$ and the set of dotted edges 
by $\mathcal{E}_{\bullet}\subset \mathcal{E}$.

\paragraph{\bf Natural label}
Let $\lambda$ be a ballot path and $T(\lambda)$ be a tree for a 
symmetric Dyck tiling above $\lambda$.
Let $N=n+n'$ be the number of edges in $T(\lambda)$, where 
$n$ is the number of edges without dots and $n'$ is the number of 
edges with dots.
A natural label $L$ of $T(\lambda)$ is a labelling such that 
integers on the edges of $T(\lambda)$ are increasing 
from the root to leaves and each integer in $[1,N]$ 
appears exactly once. 
We denote by $\mathcal{L}(\lambda)$ the set of natural labels of $T(\lambda)$.

We put a circle on a label of $T(\lambda)$ by the following rules.
\begin{enumerate}
\item There is no circle on a label which are on a dotted edge.
\item We may put a circle on a label which are on an edge without a dot.
\end{enumerate}

\paragraph{\bf Insertion history}
Let $L$ be a natural label of the tree $T(\lambda)$, 
and $E(i)$ be the edge with label $i$.
We define $h_{i}\in[0,2(i-1)]$ as 
\begin{align*}
h_{i}:=2\cdot|\{j<i| E(j)\leftarrow E(i)\}|
+|\{j<i | E(i)\uparrow E(j)\}|.
\end{align*}
We put a circle on $h_{i}$ if the edge $E(i)$ is circled,
and put a box on $h_{i}$ if the edge $E(i)$ has a dot.
We call an integer sequence $\mathbf{h}:=(h_{1},\ldots,h_{n+n'})$, 
whose entry $h_{i}$, $1\le i\le n+n'$, may have a circle or a box, 
an {\it insertion history}. 

Since dotted edges form a partial tree from the root to a leaf in $T(\lambda)$,
the integers with a box appear in $\mathbf{h}$ increasingly.

\begin{example}
The insertion history for the tree
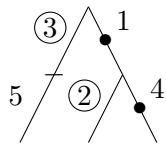
\begin{align*}
\tikzpic{-0.5}{[scale=0.6]
\coordinate
	child{coordinate(c3)
		child{coordinate(c5)}
		child[missing]}
	child{coordinate(c1)
		child{coordinate(c2)}
		child{coordinate(c4)}};
\node at(c3){$-$};
\draw[anchor=south east]($(0,0)!0.7!(c3)$)node{$\circnum{3}$};
\draw[anchor=south east]($(c3)!0.6!(c5)$)node{$5$};
\draw[anchor=south west]($(0,0)!0.5!(c1)$)node{$1$};
\draw[anchor=south east]($(c1)!0.7!(c2)$)node{$\circnum{2}$};
\draw[anchor=south west]($(c1)!0.5!(c4)$)node{$4$};
\node at ($(0,0)!0.5!(c1)$){$\bullet$};
\node at ($(c1)!0.5!(c4)$){$\bullet$};
}
\end{align*}
is $\mathbf{h}=(\boxed{0},\ \circnum{1}\ ,\ \circnum{0}\ ,\boxed{5},1)$.
\end{example}

\subsection{Trees for ballot tilings}
\label{sec:TreeBT}
We introduce a tree for a ballot path $\lambda$. 
An edge of the tree may have a dot ($\bullet$), an incoming arrow,
or an outgoing arrow. 
This tree structure first appeared in the study of the Kazhdan--Lusztig 
polynomials for the Hermitian symmetric pair \cite{Boe88}.
We follow \cite{Boe88,S141} for the definition of trees.

Let $Z$ be the set of Dyck words.
If a ballot path $\lambda$ is written as a concatenation $\lambda=\lambda'\underline{U}z$
with $z\in Z$ and a ballot path $\lambda'$, we call the underlined $U$ a {\it terminal} $U$.
If a ballot path is written as $\lambda=\lambda'\underline{U}z_{2r}\underline{U}z_{2r-1}U
\dots z_{3}Uz_{2}Uz_{1}Uz_{0}$ with $r\ge1$ and $z_{i}\in Z$ for $1\le i\le 2r$, 
we call the underlined two $U$'s a $UU$-pair.
A $U$ which is not classified above is called an {\it extra} $U$.
We construct a tree for a ballot path $\lambda$ as follows.
\begin{enumerate}
\item When $\lambda=\emptyset$, we have an empty tree.  
\item When $\lambda=z\lambda'$ with $z\in Z$ and a ballot path $\lambda'$, 
the tree for $\lambda$ is obtained by attaching the trees for $z$ and $\lambda'$ at their
roots.
\item When $\lambda=UzD$ with $z\in Z$, the tree for $\lambda$ is obtained by attaching 
an edge just above the tree for $z$.
\item When $\lambda=\underline{U}\lambda'$ and the underlined $U$ is the terminal $U$,
the tree for $\lambda$ is obtained by putting an edge just above the tree for $\lambda'$.
We mark the edge with a dot ($\bullet$).
\item When $\lambda=\underline{U}z\underline{U}\lambda'$ with $z\in Z$ and underlined 
$U$'s form a $UU$-pair, the tree for $\lambda$ is obtained by attaching an edge above 
the root of $z\lambda'$. We mark the edge with a dot ($\bullet$).
\item When $\lambda=\underline{U}\lambda'$ and the underlined $U$ is an extra $U$,
the tree for $\lambda$ is the same as the one for $\lambda'$.
\end{enumerate}
A tree for a ballot path $\lambda$ needs one more information, incoming and outgoing
arrows between edges.
Suppose that a ballot path $\lambda$ is written as 
$\lambda=\lambda'z_{2r+1}\lambda''$ with $\lambda''=Uz_{2r}Uz_{2r-1}\cdots z_{1}Uz_{0}$
and $z_{2r+1}=x_{s}x_{s-1}\cdots x_{1}$, $x_{i}\in Z$.
The path $x_{i}$ for $1\le i\le s$ cannot be decomposed into a concatenation of 
two non-empty Dyck words.
The tree for $x_{i}$ contains a unique maximal edge connecting to the root, and 
this edge corresponds to a $UD$ pair.
The tree for $\lambda''$ also contains a unique maximal edge corresponding to 
a $UU$-pair or a terminal $U$.
The maximal edge of the tree for $x_{i}$ (resp. $\lambda''$) is said to 
{\it immediately precede} the maximal edge of the tree for $x_{i+1}$ (resp. $x_{1}$).
We put an arrow between edges in the tree for $\lambda$.
\begin{enumerate}
\item[(7)] When an edge $E$ immediately precedes an edge $E'$ in the tree for $\lambda$,
we put a dashed arrow form $E$ to $E'$.
\end{enumerate}

For example, when $\lambda=UDUUDU$, the tree for a ballot path $\lambda$ is depicted as below.
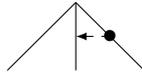
\begin{align*}
\tikzpic{-0.5}{[scale=0.6]
\coordinate
	child{coordinate(c1)}
	child{coordinate(c2)}
	child{coordinate(c3)};
\draw[dashed,latex-]($(0,0)!0.5!(c2)$)--($(0,0)!0.5!(c3)$);
\node at ($(0,0)!0.5!(c3)$){$\bullet$};
}
\end{align*}

\subsection{Enumeration}
In this subsection, we consider the enumerations of symmetric Dyck tilings
and ballot tilings with a fixed lower boundary path.
Each tiling has an algorithm to compute the partition function by making 
use of the trees associated with them.

The partition function of symmetric Dyck tilings has similar structure 
to both (non-symmetric) Dyck tilings and ballot tilings 
(see Theorem \ref{thrm:DT} for non-symmetric Dyck tilings, 
Theorem \ref{thrm:pfsymDT} for symmetric Dyck tilings and 
Theorem \ref{thrm:pfbt} for ballot tilings).

\subsubsection{Symmetric Dyck tilings}
\label{sec:EnusymDT}
Let $\lambda$ be a ballot path and $T(\lambda)$ be the tree for 
a symmetric Dyck tiling above $\lambda$.
Given an edge $E$ of $T(\lambda)$ and a natural label $L$, we denote by $l(E)$ 
the label of the edge $E$.
When $L\in\mathcal{L}(\lambda)$, we define the degree $\mathrm{deg}_{1}(L)$ by 
\begin{align*}
\mathrm{deg}_{1}(L):=
\sum_{E\in\mathcal{E}}
\#\{E'| l(E)>l(E'), E\rightarrow E'\},
\end{align*}
where $\mathcal{E}$ is the set of edges in $L$.

\begin{theorem}
\label{thrm:pfsymDT}
Let $\lambda$ be a ballot path of length $(n,n')$.
The partition function $P^{\mathrm{sym}}(\lambda)$ 
for fundamental symmetric Dyck tilings above $\lambda$ 
is given by
\begin{align}
\label{eqn:pfsym}
P^{\mathrm{sym}}(\lambda)
=\left(\sum_{L\in\mathcal{L}(\lambda)}
q^{\mathrm{deg}_{1}(L)}\right)
\cdot 
\prod_{1\le i\le n}(1+q^{i}).
\end{align}
\end{theorem}

We introduce the following lemma for the proof of Theorem \ref{thrm:pfsymDT}.
\begin{lemma}
\label{lemma:PsymNM}
Let $\lambda:=U^{N}D^{N}U^{M}$ be a ballot path.
The partition function $P(N,M):=P^{\mathrm{sym}}(\lambda)$
is given by
\begin{align}
\label{eqn:PMN}
P(N,M)=\genfrac{[}{]}{0pt}{}{M+N}{M}\cdot \prod_{i=1}^{N}(1+q^{i}).
\end{align}
\end{lemma}
\begin{proof}
Let $b_{1}$ and $b_{2}$ be the leftmost and bottom-most boxes above $\lambda$.
When $b_{1}$ is not occupied by a Dyck tile, the partition function for such symmetric 
Dyck tilings is equal to $P(N-1,M)$.
When $b_{1}$ is occupied by a Dyck tile, we have two cases:
1) $b_{2}$ is occupied by a Dyck tile of size one, and 2) $b_{2}$ is occupied by 
a ballot tile of size $(0,M)$.
In the first case, the partition function is equal to $q^{N}P(N,M-1)$.
In the second case, the partition function is equal to $q^{N}P(N-1,M)$ since 
a ballot tile of size $(0,M)$ gives a factor $q$ and the boxes between $b_1$ and 
$b_2$ have the weight $q$.
Thus, the partition function $P(N,M)$ satisfies 
\begin{align*}
P(N,M)=P(N-1,M)+q^{N}P(N-1,M)+q^{N}P(N,M-1),
\end{align*}
with the initial condition $P(1,M)=(1+q)[M+1]$.
By solving the recurrence equation, we obtain Eqn. (\ref{eqn:PMN}).
\end{proof}

\begin{proof}[Proof of Theorem \ref{thrm:pfsymDT}]
Let $\lambda$ be a ballot path of the form 
\begin{align*}
U^{m_{0}}D^{n_{0}}U^{m_{1}}D^{n_{1}}\ldots U^{m_{s-1}}D^{n_{s-1}}U^{m_{s}}.
\end{align*}
We call a configuration of partial path $DU$ a valley in $\lambda$.
By definition, the path $\lambda$ contains $s$ valleys.
We enumerate all valleys by $1,2,\ldots,s$ from left to right.
We define the subset of valleys $\mathrm{Val}(\lambda)$ as follows.
The first valley is in $\mathrm{Val}(\lambda)$.
Take a valley $v$. 
Then, $v\in\mathrm{Val}(\lambda)$ if and only if
there is no valley in $\mathrm{val}(\lambda)$ which is left to $v$
and lower than $v$.
Therefore, the heights of valleys in $\mathrm{Val}(\lambda)$ are weakly decreasing.
Let the $p_{i}$-th ($1\le i\le r$) valley be in $\mathrm{Val}(\lambda)$. 

We define a sequence of integers by
\begin{align*}
\mathbf{n}&:=(m_{0},n_{0},m_{1},n_{1},\ldots,m_{s-1},n_{s-1},m_{s}), \\
\mathbf{n}_{0}&:=(m_{0}-1,n_{0}-1,m_{1},n_{1},\ldots,m_{s-1},n_{s-1},m_{s}), \\
\mathbf{n}_{1}&:=(m_{0}+1,n_{0},m_{1}-1,n_{1},\ldots,m_{s}), \\
\mathbf{n}_{i}&:=(m_{0}+1,n_{0}-1,m_{1},\ldots,n_{p_{i}-1}+1,m_{p_{i}}-1,n_{p_{i}},\ldots,m_{s-1},n_{s-1},m_{s}), \\
\mathbf{n}_{r+1}&:=(m_{0}+1,n_{0}-1,m_{1},\ldots,n_{s-1},m_{s}-1),
\end{align*}
where $2\le i\le r$.
We define $N_{i}:=\sum_{j=0}^{p_{i}-1}n_{j}$ for $1\le i\le r$.
We denote by $P(\mathbf{n})$ the partition function of the fundamental symmetric Dyck 
tilings above $\lambda$.
The partition function satisfies the following recurrence relation:
\begin{align}
\label{eqn:recPsym}
P(\mathbf{n})=P(\mathbf{n}_{0})+
q^{n_{0}}P(\mathbf{n}_{1})+
\sum_{i=2}^{r}q^{N_{i}}P(\mathbf{n}_{i})
+q^{N_{r}}P(\mathbf{n}_{r+1}).
\end{align}
We show that 
\begin{align}
\label{eqn:facsym}
P(\mathbf{n})=
\prod_{i=0}^{s-1}\genfrac{[}{]}{0pt}{}{n_{i}+m_{i+1}+\cdots+m_{s}}{n_{i}}
\cdot \prod_{j=1}^{N}(1+q^{j}),
\end{align}
with $N:=\sum_{i=0}^{s}n_{i}$ 
satisfies the recurrence relation (\ref{eqn:recPsym}) by induction.
From Lemma \ref{lemma:PsymNM}, we have $P(1,1,1)=[2]^{2}$.
By substituting Eqn. (\ref{eqn:facsym}) into Eqn. (\ref{eqn:recPsym}) and 
taking a sum of $P(\mathbf{n_{0}})+q^{N_{r}}P(\mathbf{n}_{r+1})$ first,
we obtain the left hand side of Eqn. (\ref{eqn:recPsym}).

The remaining is to show 
\begin{align*}
\sum_{L\in\mathcal{L}(\lambda)}
q^{\mathrm{deg}_{1}(L)}=
\prod_{i=0}^{s-1}\genfrac{[}{]}{0pt}{}{n_{i}+m_{i+1}+\cdots+m_{s}}{n_{i}}.
\end{align*}
This is obvious from the fact that the partition function of 
(non-symmetric) Dyck tilings can be expressed in terms of quantum integers 
as in Theorem \ref{thrm:DT} (c.f. \cite{K12}).
\end{proof}

\begin{example}
We consider fundamental symmetric Dyck tilings above $\lambda=UDUUDU$.
We have $n=2$.
The partition function is given by 
$P^{\mathrm{sym}}(\lambda)=(1+q)^{2}(1+q^{2})(1+q+q^2+q^3)$.
\end{example}

The first factor in the right hand side of Eqn. (\ref{eqn:pfsym}) 
is nothing but the partition function of the (non-symmetric) 
Dyck tiling above $\overline{\lambda}$.
This partition function can be easily calculated in terms of trees.
First we delete dots on the tree $T(\lambda)$.
Then, when a partial tree has several children, we transform 
this partial tree into the one with one less children: 
\begin{align*}
\tikzpic{-0.5}{
\draw(0,0)--(-0.3,-0.3)(-0.7,-0.7)--(-1,-1);
\draw(-0.2,-0.2)node{\rotatebox{-45}{$-$}}(-0.8,-0.8)node{\rotatebox{-45}{$-$}};
\draw[dashed](-0.3,-0.3)--(-0.7,-0.7);
\draw[decoration={brace,mirror,raise=5pt},decorate]
  (0,0) --(-1,-1);
\draw(-0.2,-0.2)node[left=9pt]{$N$};
\draw(0,0)--(0.3,-0.3)(0.7,-0.7)--(1,-1);
\draw[dashed](0.3,-0.3)--(0.7,-0.7);
\draw(0.2,-0.2)node{\rotatebox{45}{$-$}}(0.8,-0.8)node{\rotatebox{45}{$-$}};
\draw[decoration={brace,mirror,raise=5pt},decorate]
  (1,-1)--(0,0);
\draw(0.2,-0.2)node[right=9pt]{$M$};
}
&\mapsto\genfrac{[}{]}{0pt}{}{M+N}{M}\cdot
\tikzpic{-0.4}{
\draw(0,0)--(0,-0.3)(0,-0.7)--(0,-1);
\draw(0,-0.2)node{$-$}(0,-0.8)node{$-$};
\draw[dashed](0,-0.3)--(0,-0.7);
\draw[decoration={brace,mirror,raise=5pt},decorate]
    (0,-1)-- (0,0);
\draw(0,-0.5)node[right=9pt]{$M+N$};
}
\end{align*}
We arrive at the tree with a unique leaf by successive applications
of the above operation.
This tree gives the factor one.

The second factor in the right hand side of Eqn. (\ref{eqn:pfsym}) 
comes from the number of the dotted edges in $T(\lambda)$.

\begin{example}
We consider $\lambda=UDUUDUDD$. Then, we have 
\begin{align*}
\tikzpic{-0.5}{[scale=0.6]
\coordinate
	child{coordinate(c1)}
	child{coordinate(c2)
		child{coordinate(c3)}
		child{coordinate(c4)}
	     };
}\rightarrow
\genfrac{[}{]}{0pt}{}{2}{1}\cdot
\tikzpic{-0.5}{[scale=0.5]
\coordinate
	child{coordinate(c1)}
	child{coordinate(c2)
                child[missing]
		child{coordinate(c3)
                child[missing]
		child{coordinate(c4)}}
	     };
\node at (c2) {$-$};\node at (c3) {$-$};
}\rightarrow
\genfrac{[}{]}{0pt}{}{2}{1}\genfrac{[}{]}{0pt}{}{4}{1}
=[4][2].
\end{align*}
\end{example}

\subsubsection{Ballot tilings}
\label{sec:enumBT}
We summarize the results of \cite{S17} with respect to the 
enumeration of ballot tilings above $\lambda$.

Given a tree $T$ for a ballot tiling, let $L_{\mathrm{ref}}$ 
be a label of $T$ such that the modified post-order word 
is the identity. We denote $n_{E}^{\mathrm{ref}}$ be an integer
on an edge $E$ in $L_{\mathrm{ref}}$.
Here, the modified post-order means as below.
We visit edges on a tree one-by-one following the post-order, and 
visit edges $E$'s with a dot $(\bullet)$ from top to bottom up to a ramification 
point soon after all the edges on the sub-tree which is left to $E$'s 
are visited.
Then, we continue visiting remaining edges by the modified post-order.

\begin{example}
\label{ex:bt1}
Let $\lambda=UDUUUUDU$. The tree for $\lambda$ and its $L_{\mathrm{ref}}$ are 
as below.
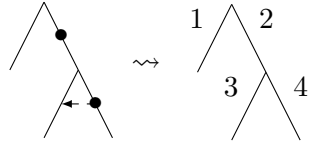
\begin{align*}
\tikzpic{-0.5}{[scale=0.6]
\coordinate
	child{coordinate(c1)}
	child{coordinate(c2)
		child{coordinate(c3)}
		child{coordinate(c4)}
	};
\node at ($(0,0)!0.5!(c2)$){$\bullet$};
\node at ($(c2)!0.5!(c4)$){$\bullet$};
\draw[dashed,latex-]($(c2)!0.5!(c3)$)--($(c2)!0.5!(c4)$);
}\leadsto
\tikzpic{-0.5}{[scale=0.6]
\coordinate
	child{coordinate(c1)}
	child{coordinate(c2)
		child{coordinate(c3)}
		child{coordinate(c4)}
	};
\node[anchor=south east]at($(0,0)!.5!(c1)$){$1$};
\node[anchor=south west]at($(0,0)!.5!(c2)$){$2$};
\node[anchor=south east]at($(c2)!.5!(c3)$){$3$};
\node[anchor=south west]at($(c2)!.5!(c4)$){$4$};
}
\end{align*}
\end{example}

Given $L_{\mathrm{ref}}$, we define 
\begin{align*}
S(E):=|T|+1-n_{E}^{\mathrm{ref}},
\end{align*}
where $|T|$ is the number of edges in $T$.
We denote by $\mathcal{E}$ and $\mathcal{E}_{\bullet}$ the set 
of edges in $T$ and the set of edges with $\bullet$.

Given an edge $E$, we consider a sequence of edges from $E$ to the root.
When $E'$ appears in the sequence, we say that $E'$ is an ancestor edge 
of $E$.
Let $A(T)\subseteq\mathcal{E}\setminus\mathcal{E}_{\bullet}$
be the set of edges such that all the ancestor edges do not have an incoming 
arrow.
Then, we define $B(T):=\mathcal{E}\setminus(\mathcal{E}_{\bullet}\cup A(T))$, 
that is, $B(T)$ is the set of edges $E$ such that an ancestor edge of $E$ 
has an incoming arrow.

We put a circle on a label of an edge in $T$ by the following rule.
\begin{enumerate}
\item An edge in $\mathcal{E}_{\bullet}$ does not have a circle.
\item An edge without an incoming arrow has a circle.
\item Suppose that we have a sequence of $p+1$ edges with arrows
$E_1\leftarrow\ldots\leftarrow E_p\leftarrow E$. 
We define $n_{E}$ as the label of the edge $E$ in the natural label.
If $n_{E_i}>n_{E}$ for all $1\le i\le p$, then we put circles on the edges 
$E_1$ to $E_p$.
\end{enumerate}

Let $\mathcal{E}_{\circ}$ be the set of edges with a circle.
Note that $\mathcal{E}_{\circ}$ depends on a label of $T$ due to 
the third rule.
By definition, we have $A(T)\subseteq\mathcal{E}_{\circ}$.
We define 
\begin{align*}
\mathrm{deg}_{2}(L):=
\sum_{E\in\mathcal{E}}
\#\{E'| n_{E}>n_{E'}, E\rightarrow E'\}.
\end{align*}

\begin{theorem}[Theorem 8.8 in \cite{S17}]
\label{thrm:pfbt}
The partition function for ballot tilings above a ballot path $\lambda$
is given by 
\begin{align*}
P^{\mathrm{bal}}(\lambda)
=\left(\sum_{L}q^{\mathrm{deg}_{2}(L)}
\prod_{E\in B\cap\mathcal{E}_{\circ}(L)}
(1+q^{S(E)-1})
\right)\cdot 
\prod_{E\in A(T)}(1+q^{S(E)}).
\end{align*}
\end{theorem}

\begin{example}
\label{ex:bt2}
We consider the same ballot path in Example \ref{ex:bt1}, {\it i.e.}, $\lambda=UDUUUUDU$. 
The tree for $\lambda$ contains two dotted edges and one arrow.
The natural labels with circles are 
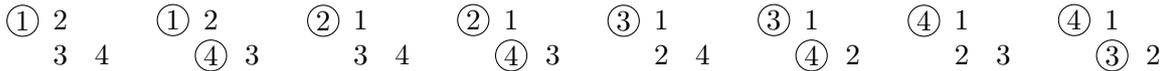
\begin{align*}
\begin{matrix}
\circnum{1} & 2 & \\
& 3 & 4
\end{matrix}
\qquad
\begin{matrix} 
\circnum{1} & 2 & \\
& \circnum{4} & 3
\end{matrix} 
\qquad
\begin{matrix} 
\circnum{2} & 1 & \\
& 3 & 4
\end{matrix} 
\qquad
\begin{matrix} 
\circnum{2} & 1 & \\
& \circnum{4} & 3
\end{matrix} 
\qquad
\begin{matrix} 
\circnum{3} & 1 & \\
& 2 & 4
\end{matrix} 
\qquad
\begin{matrix} 
\circnum{3} & 1 & \\
& \circnum{4} & 2
\end{matrix}
\qquad
\begin{matrix} 
\circnum{4} & 1 & \\
& 2 & 3
\end{matrix} 
\qquad
\begin{matrix} 
\circnum{4} & 1 & \\
& \circnum{3} & 2
\end{matrix}  
\end{align*}

The circles give the factor 
\begin{align*}
\begin{matrix}
\circnum{\ast} & \ast & \\
& \ast & \ast
\end{matrix}
\mapsto (1+q^{4}), 
\qquad
\begin{matrix}
\circnum{\ast} & \ast & \\
& \circnum{\ast} & \ast
\end{matrix}
\mapsto (1+q)(1+q^{4}).
\end{align*}
The partition function is computed as
\begin{align*}
(1+q^{4})(1+q+q^2+q^3)+(1+q)(1+q^4)(q+q^2+q^3+q^4)=[3][8].
\end{align*}
\end{example}

The partition function can be easily computed by starting from the tree
for a ballot path $\lambda$ (see Section 7.2 and 7.3 in \cite{S17}).
We define operations on a partial tree as follows.
\begin{align*}
\tikzpic{-0.5}{
\draw(0,0)--(-0.3,-0.3)(-0.7,-0.7)--(-1,-1);
\draw(-0.2,-0.2)node{\rotatebox{-45}{$-$}}(-0.8,-0.8)node{\rotatebox{-45}{$-$}};
\draw[dashed](-0.3,-0.3)--(-0.7,-0.7);
\draw[decoration={brace,mirror,raise=5pt},decorate]
  (0,0) --(-1,-1);
\draw(-0.2,-0.2)node[left=9pt]{$N$};
\draw(0,0)--(0.3,-0.3)(0.7,-0.7)--(1,-1);
\draw[dashed](0.3,-0.3)--(0.7,-0.7);
\draw(0.2,-0.2)node{\rotatebox{45}{$-$}}(0.8,-0.8)node{\rotatebox{45}{$-$}};
\draw[decoration={brace,mirror,raise=5pt},decorate]
  (1,-1)--(0,0);
\draw(0.2,-0.2)node[right=9pt]{$M$};
}
&\mapsto\genfrac{[}{]}{0pt}{}{M+N}{M}\cdot
\tikzpic{-0.4}{
\draw(0,0)--(0,-0.3)(0,-0.7)--(0,-1);
\draw(0,-0.2)node{$-$}(0,-0.8)node{$-$};
\draw[dashed](0,-0.3)--(0,-0.7);
\draw[decoration={brace,mirror,raise=5pt},decorate]
    (0,-1)-- (0,0);
\draw(0,-0.5)node[right=9pt]{$M+N$};
} \\
\tikzpic{-0.5}{
\draw(0,0)--(-0.3,-0.3)(-0.7,-0.7)--(-1,-1);
\draw(-0.2,-0.2)node{\rotatebox{-45}{$-$}}(-0.8,-0.8)node{\rotatebox{-45}{$-$}};
\draw[dashed](-0.3,-0.3)--(-0.7,-0.7);
\draw[decoration={brace,mirror,raise=5pt},decorate]
  (0,0) --(-1,-1);
\draw(-0.2,-0.2)node[left=9pt]{$N$};
\draw(0,0)--(0.35,-0.35)(0.65,-0.65)--(1,-1);
\draw[dashed](0.35,-0.35)--(0.65,-0.65);
\draw(0.3,-0.3)node{\rotatebox{45}{$-$}}(0.7,-0.7)node{\rotatebox{45}{$-$}};
\draw(0.15,-0.15)node{$\bullet$}(0.85,-0.85)node{$\bullet$};
\draw[decoration={brace,mirror,raise=5pt},decorate]
  (1,-1)--(0,0);
\draw(0.2,-0.2)node[right=9pt]{$M$};
}&\mapsto
\genfrac{[}{]}{0pt}{}{M+N}{M}_{q^{2}}\prod_{i=1}^{N}(1+q^{i})\cdot
\tikzpic{-0.4}{
\draw(0,0)--(0,-0.35)(0,-0.65)--(0,-1);
\draw(0,-0.3)node{$-$}(0,-0.7)node{$-$};
\draw[dashed](0,-0.4)--(0,-0.6);
\draw(0,-0.15)node{$\bullet$}(0,-0.85)node{$\bullet$};
\draw[decoration={brace,mirror,raise=5pt},decorate]
    (0,-1)-- (0,0);
\draw(0,-0.5)node[right=9pt]{$M+N$};
} \\
\tikzpic{-0.5}{
\draw(0,0)--(-0.45,-0.45)(-0.7,-0.7)--(-1,-1);
\draw(-0.4,-0.4)node{\rotatebox{-45}{$-$}}(-0.8,-0.8)node{\rotatebox{-45}{$-$}};
\draw[dashed](-0.3,-0.3)--(-0.7,-0.7);
\draw[decoration={brace,mirror,raise=5pt},decorate]
  (0,0) --(-1,-1);
\draw(-0.2,-0.2)node[left=9pt]{$N$};
\draw(0,0)--(0.45,-0.45)(0.65,-0.65)--(1,-1);
\draw[dashed](0.45,-0.45)--(0.65,-0.65);
\draw(0.4,-0.4)node{\rotatebox{45}{$-$}}(0.7,-0.7)node{\rotatebox{45}{$-$}};
\draw(0.15,-0.15)node{$\bullet$}(0.85,-0.85)node{$\bullet$};
\draw[decoration={brace,mirror,raise=5pt},decorate]
  (1,-1)--(0,0);
\draw(0.2,-0.2)node[right=9pt]{$M$};
\draw[latex-,dashed](-0.3,-0.3)--(0.3,-0.3);
}&\mapsto
\genfrac{[}{]}{0pt}{}{M+N}{M}_{q^{2}}
\frac{[2M+N]}{[2(M+N)]}\prod_{i=1}^{N}(1+q^{i}) \cdot
\tikzpic{-0.4}{
\draw(0,0)--(0,-0.35)(0,-0.65)--(0,-1);
\draw(0,-0.3)node{$-$}(0,-0.7)node{$-$};
\draw[dashed](0,-0.4)--(0,-0.6);
\draw(0,-0.15)node{$\bullet$}(0,-0.85)node{$\bullet$};
\draw[decoration={brace,mirror,raise=5pt},decorate]
    (0,-1)-- (0,0);
\draw(0,-0.5)node[right=9pt]{$M+N$};
\draw[latex-,dashed](-0.4,-0.15)node{$($}--(-0.1,-0.15)node{$)$};
}
\end{align*}
By successive applications of above operations, we arrive at a tree
without a ramification. 
All edges in the tree are either with dots, or without dots.
We assign factors to the trees (not a partial tree) as follows.
\begin{align*}
\tikzpic{-0.4}{
\draw(0,0)--(0,-0.3)(0,-0.7)--(0,-1);
\draw(0,-0.2)node{$-$}(0,-0.8)node{$-$};
\draw[dashed](0,-0.3)--(0,-0.7);
\draw[decoration={brace,mirror,raise=5pt},decorate]
    (0,-1)-- (0,0);
\draw(0,-0.5)node[right=9pt]{$N$};
}&\mapsto\prod_{i=1}^{N}(1+q^{i}), \\
\tikzpic{-0.4}{
\draw(0,0)--(0,-0.35)(0,-0.65)--(0,-1);	
\draw(0,-0.3)node{$-$}(0,-0.7)node{$-$};
\draw[dashed](0,-0.4)--(0,-0.6);
\draw(0,-0.15)node{$\bullet$}(0,-0.85)node{$\bullet$};
\draw[decoration={brace,mirror,raise=5pt},decorate]
    (0,-1)-- (0,0);
\draw(0,-0.5)node[right=9pt]{$N$};
}&\mapsto 1.
\end{align*}

\begin{example}
We consider $\lambda=UDUUUUDU$.
Then, the partition function can be calculated by 
\begin{align*}
\tikzpic{-0.5}{[scale=0.6]
\coordinate
	child{coordinate(c1)}
	child{coordinate(c2)
		child{coordinate(c3)}
		child{coordinate(c4)}
	};
\node at ($(0,0)!0.5!(c2)$){$\bullet$};
\node at ($(c2)!0.5!(c4)$){$\bullet$};
\draw[dashed,latex-]($(c2)!0.5!(c3)$)--($(c2)!0.5!(c4)$);
}\rightarrow
\frac{[4]}{[2]}\frac{[3]}{[4]}[2]
\tikzpic{-0.5}{[scale=0.6]
\coordinate
	child{coordinate(c1)}
	child{coordinate(c2)
		child[missing]
		child{coordinate(c3)
			child[missing]
			child{coordinate(c4)}
		}
	};
\node at (c2){$-$};\node at (c3){$-$};
\node at ($(0,0)!.5!(c2)$){$\bullet$};
\node at ($(c2)!.5!(c3)$){$\bullet$};
\node at ($(c3)!.5!(c4)$){$\bullet$};
}\rightarrow
[3]\frac{[8]}{[2]}[2]=[3][8].
\end{align*}
We have the same partition function as Example \ref{ex:bt2}.
\end{example}

\section{Symmetric Dyck tiling strip and ballot tilings}
\subsection{Symmetric Dyck tiling strip}
\label{sec:symDTS}
Fundamental symmetric Dyck tilings above $\lambda$ are bijective to 
natural labels with circles. 
We define the {\it symmetric Dyck tiling strip} (symDTS for short) for a natural label with circles.
This algorithm is a generalization of the Dyck tiling strip bijection
for (non-symmetric) Dyck tilings.
The symDTS bijection consists of three operations: spread of a symmetric Dyck tiling, 
addition of a strip, and addition of a single box with $\ast$.

\paragraph{\bf Spreads of a symmetric Dyck tiling}
We define two spreads of a symmetric Dyck tiling: 
the Dyck spread and the ballot spread.
The first one corresponds to an insertion of a path $UD$ into a ballot path,
and the second one corresponds to an insertion of a path $U$.

We define the Dyck spread of a symmetric Dyck tiling following \cite{KMPW12}.
Given a ballot path $\lambda$ of length $(n,n')$ and a line $x=s$ with $0\le s\le 2n+n'$,
we define the Dyck spread of $\lambda$ at $x=s$ to be the ballot path $\lambda'$ 
consisting of the following points:
\begin{align*}
\{(x-1,y)| (x,y)\in\lambda, x\le s\}
\cup
\{(x,y+1)| (x,y)\in\lambda, x=s\}
\cup 
\{(x+1,y)| (x,y)\in\lambda, x\ge s\}.
\end{align*}
The Dyck spread of the symmetric Dyck tile at $x=s$ is defined by the spread of 
the ballot path at $x=s$ that coordinatizes the symmetric Dyck tile.
The Dyck spread of a symmetric Dyck tiling at $x=s$ is defined by taking 
the Dyck spread of the ballot tiles contained in the symmetric Dyck tiling.

Similarly, we define the ballot spread of a symmetric Dyck tiling as follows.
Given a ballot path $\lambda$ of length $(n,n')$, and a line $x=s$ with $0\le s\le 2n+n'$,
we define the ballot spread of $\lambda$ at $x=s$ to be the 
ballot path $\lambda'$ consisting of the following points:
\begin{align*}
\{(x-1,y)| (x,y)\in\lambda, x\le s\}
\cup 
\{(x,y+1)| (x,y)\in\lambda, x\ge s\}.
\end{align*}
The ballot spreads of a Dyck tile and of a symmetric Dyck tiling are defined 
similarly as the Dyck spreads.

\paragraph{\bf Addition of a strip}
Addition of a strip is sometimes called the {\it strip-growth} of a Dyck tiling \cite{KMPW12}.
After the spread of a symmetric Dyck tiling at $x=s$, we add single tiles on the up steps, 
which is right to $x=s$, of the new top path obtained by a Dyck or ballot spread.

\paragraph{\bf addition of a single box with $\ast$}
We add a single box with $\ast$ on the symmetric Dyck tiling obtained after the addition
of a strip.
The box with $\ast$ is at $x=2n+n'$ above the ballot path $\lambda$ of length $(n,n')$.

The symDTS bijection from natural labels with circles to symmetric Dyck tilings is 
defined as follows.
Let $T(\lambda)$ be a tree for a symmetric Dyck tiling associated with the ballot path $\lambda$.
We denote by $E(n)$ the edge with the label $n$ in $T(\lambda)$.

Let an integer sequence $\mathbf{p}:=(p_{1},\ldots,p_{n})$ be 
an insertion history for a natural label of the tree $T(\lambda)$ 
defined in Section \ref{sec:treesforsymDT}.
When $\mathbf{p}=\emptyset$, we define the symmetric Dyck tiling as a single dot at $(0,0)$.
Given an insertion history $\mathbf{p}=(p_1,\ldots,p_{n})$, 
we let $\mathbf{p}':=(p_{1},\ldots,p_{n-1})$ and $D'$ be the symmetric 
Dyck tiling associated with $\mathbf{p}'$.
We obtain a symmetric Dyck tiling $D$ for $\mathbf{p}$ recursively 
by the following three operations.
\begin{enumerate}[(symDTS1)]
\item
\label{symDTS1}
If the edge $E(n)$ in $T(\lambda)$ is an edge without a dot, we perform 
a Dyck spread of $D'$ at $x=p_{n}$.
Similarly, when $E(n)$ is an edge with a dot, we perform 
a ballot spread of $D'$ at $x=p_{n}$.
\item
We perform the addition of a strip on the symmetric Dyck tiling obtained by the step (symDTS1).
\item If the edge $E(n)$ has the label $n$ with a circle, we perform the addition 
of a single box with $\ast$ on the symmetric Dyck tiling obtained by the step (symDTS2) and
denote by $D$ the new symmetric Dyck tiling.
Otherwise, we denote by $D$ the symmetric Dyck tiling obtained by the step (symDTS2).
\end{enumerate}

Figure \ref{fig:symDTS} is an example of the symDTS bijection on the following labeled tree. 
\begin{figure}[ht]
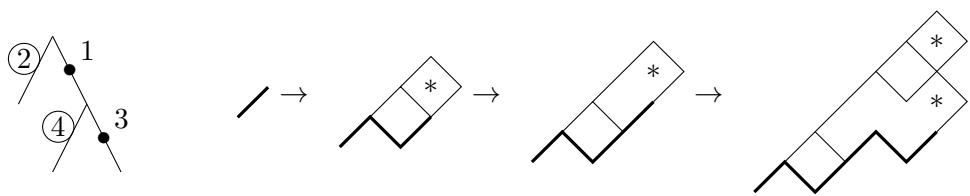

\tikzpic{-0.5}{[scale=0.6]
\coordinate
	child{coordinate(c1)}
	child{coordinate(c2)
		child{coordinate(c3)}
		child{coordinate(c4)}	
	};	
\node at ($(0,0)!0.5!(c2)$){$\bullet$};
\node at($(c2)!0.5!(c4)$){$\bullet$};
\node[anchor=south east,circle,draw,inner sep=1pt]at($(0,0)!0.5!(c1)$){$2$};
\node[anchor=south east,circle,draw,inner sep=1pt]at($(c2)!0.5!(c3)$){$4$};
\node[anchor=south west]at($(0,0)!.5!(c2)$){$1$};
\node[anchor=south west]at($(c2)!.5!(c4)$){$3$};
}
\qquad
\tikzpic{-0.5}{[x=0.4cm,y=0.4cm]
\draw[very thick](0,0)--(1,1);
}$\rightarrow$
\tikzpic{-0.5}{[x=0.4cm,y=0.4cm]
\draw[very thick](0,0)--(1,1)--(2,0)--(3,1);
\draw(1,1)--(3,3)--(4,2)--(3,1)--(2,2);
\node at (3,2){$\ast$};
}$\rightarrow$
\tikzpic{-0.5}{[x=0.4cm,y=0.4cm]
\draw[very thick](0,0)--(1,1)--(2,0)--(4,2);
\draw(1,1)--(2,2)--(3,1);
\draw(2,2)--(4,4)--(5,3)--(4,2);
\node at (4,3){$\ast$};
}$\rightarrow$
\tikzpic{-0.5}{[x=0.4cm,y=0.4cm]
\draw[very thick](0,0)--(1,1)--(2,0)--(4,2)--(5,1)--(6,2);
\draw(1,1)--(2,2)--(3,1);
\draw(2,2)--(4,4)--(5,3)--(6,4)--(7,3)--(6,2);
\draw(4,4)--(6,6)--(7,5)--(6,4)--(5,5);
\node at (6,3){$\ast$};
\node at (6,5){$\ast$};
}
\caption{A natural label with circles and the growth of a symmetric Dyck tiling.
The insertion history is $\mathbf{p}=(0,0,3,3)$.}
\label{fig:symDTS}
\end{figure}

\begin{theorem}
\label{thrm:bijfsDTS}
The symmetric DTS map is a bijection between natural lables of a tree $T(\lambda)$
and fundamental symmetric Dyck tilings over $\lambda$.
\end{theorem}
\begin{proof}
The symmetric DTS map sends a natural label of a tree $T(\lambda)$ to 
a fundamental symmetric Dyck tiling over $\lambda$. 
Then, injectivity of the map directly follows from the definition.
We will show that the map is surjective.

We prove surjectivity by induction. 
Let $\lambda$ be a ballot path of length $(n,n')$ and 
$T'$ be a tree for a ballot path $\lambda'$ where $\lambda'$ 
is obtained from $\lambda$ by deleting one edge.
By induction hypothesis, the symmetric DTS map is a bijection 
between natural labels of $T'$ and fundamental symmetric 
Dyck tilings over $\lambda'$.
Let $\mathcal{D}$ be a fundamental symmetric Dyck tiling over 
a ballot path $\lambda$.
By reversing the steps from (symDTS1) to (symDTS3), 
one can easily find the line $x=p_{n+n'}$ and a fundamental symmetric
Dyck tiling $\mathcal{D}'$ over $\lambda'$. 
The line $x=p_{n+n'}$ is the place where we perform a Dyck or ballot spread 
on $\mathcal{D}'$.
If we have to add a single box with $\ast$, the label $n+n'$ should be circled.
We can extract the following three information from the tiling $\mathcal{D}$: 
the natural label $L'$ for the fundamental symmetric 
Dyck tiling $\mathcal{D}'$, the position $x=p_{n+n'}$, and whether the label $n$ 
is circled or not.
From these, one can construct	 a unique natural label $L$ for the tiling $\mathcal{D}'$.
Thus, there exists a natural label for a given fundamental symmetric 
Dyck tiling, which implies the map is surjective.
Together with injectivity, we conclude that the map is a bijection.
\end{proof}

\subsection{An inclusive map from ballot tilings to symmetric Dyck tilings}
\label{sec:symDTSbt}

In general, the tree for a ballot tiling above $\lambda$ is different 
from the one for a fundamental symmetric Dyck tiling above $\lambda$
even if the underlying tiling is the same.
Since ballot tilings are the subset in the set of fundamental symmetric Dyck tilings,
we will construct an inclusive map from a natural label of a ballot tiling above $\lambda$
to a natural label of a fundamental symmetric Dyck tiling above $\lambda$.

Let $\mathcal{L}^{\mathrm{D}}(\lambda)$ be the set of natural labels 
for fundamental symmetric Dyck tilings above $\lambda$ and $\mathcal{L}^{\mathrm{b}}(\lambda)$ 
be the set of natural labels for ballot tilings above $\lambda$. 
Then, we define $\mathcal{L}^{\mathrm{b}\subset\mathrm{D}}(\lambda)$ as a subset in 
$\mathcal{L}^{\mathrm{D}}$ that corresponds to ballot tilings.

We summarize the properties of a natural label in $\mathcal{L}^{\mathrm{b}\subset\mathrm{D}}$
as follows.
Then, we will give a construction of such natural labels 
in $\mathcal{L}^{\mathrm{b}\subset\mathrm{D}}$.

A natural label $L\in\mathcal{L}^{\mathrm{b}\subset\mathrm{D}}\subset \mathcal{L}^{D}$ 
may have dotted edges.
By definition of a tree, the dotted edges in $L$ form a subtree starting from the root and 
ending at a unique leaf without ramifications.
We enumerate these dotted edges by $1,2\ldots,r$ from the leaf to the root, where $r$ is the 
number of dotted edges.
We divide these dotted edges into two classes: the dotted edges enumerated by odd integers, 
and the ones enumerated by even integers.
We call them odd and even dotted edges respectively.
We denote by $E_i$ for $1\le i\le r$ the label of the dotted edge enumerated by $i$.
Recall that a ballot tiling is a fundamental symmetric Dyck tiling with two conditions 
(see Section \ref{sec:bt}).
These conditions can be written in terms of a natural label 
$L\in\mathcal{L}^{\mathrm{b}\subset\mathrm{D}}$ as follows.
A natural label $L\in\mathcal{L}^{\mathrm{b}\subset\mathrm{D}}$ should satisfy 
\begin{enumerate}[(P1)]
\item
\label{Enu:P1} 
Let $s\in[1,r]$ be an even integer on a dotted edge.
There is no label $l$ such that $E_{s+1}<l<E_{s}$ and $l$ is circled.
\item
\label{Enu:P2}
 Let $s\in[1,r]$ be an odd integer on a dotted edge.
The number of labels $l$ such that $E_{s+1}<l<E_{s}$ and $l$ is circled 
is even.
\end{enumerate}

Let $L^{b}\in\mathcal{L}^{\mathrm{b}}(\lambda)$ be a natural label with circles 
on the tree associated with $\lambda$.
We will define a map $\iota$ from $\mathcal{L}^{\mathrm{b}}$
to $\mathcal{L}^{\mathrm{b}\subset\mathrm{D}}$.
The map $\iota$ is given as a composition of two maps 
\begin{align*}
L^{b} \mapsto \overline{L} \mapsto L, 
\end{align*}
where $\overline{L}\in\mathcal{L}^{\mathrm{D}}$ and $L\in\mathcal{L}^{\mathrm{b}\subset\mathrm{D}}$.
Note that $\overline{L}$ may not be in $\mathcal{L}^{\mathrm{b}\subset\mathrm{D}}$.

In $L^{b}$, a dotted edge corresponds to a terminal $U$, or a $UU$-pair.
There is no edge corresponding to an extra $U$.
As in the case of a natural label $\mathcal{L}^{\mathrm{D}}$, 
dotted edges in $L^{b}$ form a subtree starting from the root and ending at a unique leaf
without ramifications.
The labels on the dotted edges in $L^{b}$ are decreasing from the leaf to the root.
Let $T$ be a tree for a fundamental symmetric Dyck tiling for a ballot path $\lambda$.
Note that a tree $T$ for a fundamental symmetric Dyck tiling is in general different from
the one for a ballot tiling. 
Since each $U$ in a ballot path $\lambda$ has a corresponding dotted edge in $T$ 
if it is not paired with $D$, the tree $T$ has dotted edges corresponding to a 
terminal $U$, a $U$ in a $UU$-pair and an extra $U$ (if it exists) in the ballot path $\lambda$.

Recall that we enumerate dotted edges in $T$ from the leaf to the root by $1,2,\ldots,r$, 
where $r$ is the number of dotted edges in $T$.
Let $T'$ be a tree for a ballot tiling.
Given an edge $f$ in $T$, we consider a unique sequence of edges from $f$ to the root.
When this sequence contains at least one dotted edges enumerated by integers larger 
than or equal to $k$, we say that the ancestor edge with a dot is the $k$-th dotted edge.
Similarly, when the sequence contains no dotted edges, we say that the ancestor edge with a dot 
is $r+1$-th dotted edge. Since the number of dotted edges is $r$, we regard the root as 
the $r+1$-th dotted edge.

Given a ballot path $\lambda$, let $T$ and $T'$ be the trees for a fundamental symmetric 
Dyck tiling and a ballot tiling, respectively.
\begin{prop}
\label{prop:anc}
Let $f, f'$ be edges in $T$ and $T'$ respectively such that $f$ and $f'$ corresponds 
to the same $UD$ pair in $\lambda$.
By definition, the edges $f$ and $f'$ do not have dots.
When the sequence of edges from $f'$ to the root in $T'$ 
contains an edge with an incoming arrow, the ancestor edge with a dot is the $2m$-th 
edge in $T$ with some $m$.
Similarly, if the sequence in $T'$ contains only edges without an incoming arrow, 
the ancestor edge with a dot is $2m-1$-th dotted edge  in $T$ with some $m$.
\end{prop}
\begin{proof}
It is obvious from the definition of a tree for a ballot tiling and that of arrows on a tree.
\end{proof}

\begin{remark}
\label{remark:btsDT}
In general, two edges occupying the same node in a tree for a ballot tiling 
may not occupy the same node in a tree for a symmetric Dyck tiling.
This occurs due to the existence of arrows in the tree for a ballot tiling.
Note also that there is an obvious bijection between the edges without a dot in
the tree for a ballot tiling and the ones in the tree for a symmetric Dyck tiling.
\end{remark}

\begin{example}
We consider a ballot path $\lambda:=UDUUDU$.
In Figure \ref{fig:twotrees}, we have a tree for a ballot tiling (left picture) 
and a tree for a symmetric Dyck tiling (right picture).  
The middle edge in the left picture has an incoming arrow.
From Proposition \ref{prop:anc}, the corresponding edge in the right picture
has the second ancestor edge.
Similarly, the left-most edge in the left picture has the third ancestor edge 
in the right picture. 
As Remark \ref{remark:btsDT}, although the two edges without dots share the same node 
in the left picture, they do not share the same node in the right picture.
\begin{figure}[ht]
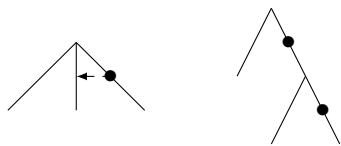

\tikzpic{-0.5}{[scale=0.6]
\coordinate
	child{coordinate(c4)}
	child{coordinate(c3)}
	child{coordinate(c1)
		child[missing]
		child[missing]
		child[missing]};
\node at ($(0,0)!0.5!(c1)$){$\bullet$};
\draw[dashed,-latex]($(0,0)!0.5!(c1)$)--($(0,0)!0.5!(c3)$);
}\qquad
\tikzpic{-0.5}{[scale=0.6]
\coordinate
	child{coordinate(c1)}
	child{coordinate(c2)
		child{coordinate(c3)}
		child{coordinate(c4)}};
\node at ($(0,0)!0.5!(c2)$){$\bullet$};
\node at ($(c2)!0.5!(c4)$){$\bullet$};
}
\caption{Two trees associated with $\lambda=UDUUDU$.}
\label{fig:twotrees}
\end{figure}
\end{example}

We define a natural label $\overline{L}\in\mathcal{L}^{\mathrm{D}}$ on $T$ as follows.
We construct $\overline{L}$ such that we have a bijection between edges without a dot 
in $L^{b}$ and the ones in $\overline{L}$ as in Remark \ref{remark:btsDT}.
We put the label of an edge without a dot in $L^{b}$ on the corresponding 
edge without a dot in $\overline{L}$. 
In this process, if a label in $L^{b}$ has a circle, the corresponding label in $\overline{L}$
also has a circle.
We put the label of the $i$-th (enumerated from the root) dotted edge in $L^{b}$ on the $(2i-1)$-th (resp. $2i$-th) 
dotted edge from the root in $\overline{L}$ (whose underlying tree is $T$)
if a ballot path $\lambda$ (regarded as a path for a ballot tiling) 
does not have (resp. has) the extra $U$.
The natural label $\overline{L}$ may have dotted edges without a label.
Suppose that the $i$-th (enumerated from the root) dotted edge has no label after the above-mentioned operation
and this edge in $\overline{L}$ has a subtree $T''$ just below it.
Let $\min(T'')$ be the minimum label of the subtree $T''$.
We first increase the labels of $\overline{L}$ which are larger than $\min(T'')-1$, and 
put the label $\min(T'')$ on the $i$-th dotted edge in $\overline{L}$.
We perform this procedure until all the dotted edges without a label in $\overline{L}$ have labels.

We perform the following operations on $\overline{L}$ and obtain a new label $L$ on $T$.
Let $E_{i}$ be the label of the dotted edge in $\overline{L}$ enumerated by $1\le i\le r$ 
from the leaf to the root.
We define $E_{r+1}=0$.
We start from $i=r$ and recursively modify $\overline{L}$ up to $i=1$ to obtain a natural label $L$.
We have two cases according to the parity of $i$.
\paragraph{\bf In case of $i$ even}
Let $l$ be the smallest circled label such that $l$ is strictly left to $E_{i}$ and 
$E_{i+1}<l<E_{i}$.
Suppose $c$ be the number of circled labels which are strictly larger than $l-1$ and smaller than $E_{i}$.
We have two cases according to the parity of $c$:
\begin{enumerate}
\item When $c$ is odd. We delete the circle of the label $l$. Then, we change the labels 
$l\le j\le E_{i}-1$ to $j+1$ and $E_{i}$ to $l$ by keeping the circle if it is circled.
\item When $c$ is even. 
We change the labels $l\le j\le E_{i}-1$ to $j+1$ and $E_{i}$ to $l$ by keeping the circle 
if it is circled.
\end{enumerate} 
\paragraph{\bf In case of $i$ odd}
Let $l$ be the smallest circled label such that $l$ is strictly left to $E_{i}$ and 
$E_{i+1}<l<E_{i}$.
Suppose $c$ be the number of circled labels which is strictly larger than $l-1$ and smaller than $E_{i}$.
We have two cases according to the parity of $c$:
\begin{enumerate}
\item When $c$ is even, the labels of $L$ are the same as $\overline{L}$.
\item When $c$ is odd, we put a circle on the maximum label $j$ such that $j$ is smaller than $l$
and $j$ cannot be circled in $L^{b}$. 
\end{enumerate}

In both cases for $i$ even and odd, the number of labels $l$ with a circle 
such that $E_{i+1}<l<E_{i}$ is even.

We denote by $L$ the newly obtained natural label from $\overline{L}$ 
by the above procedures on dotted edges.
\begin{prop}
\label{prop:iotawd}
The map 
$\iota:\mathcal{L}^{\mathrm{b}}(\lambda)\rightarrow
\mathcal{L}^{\mathrm{b}\subset\mathrm{D}}(\lambda)$, $L^{\mathrm{b}}\mapsto L$, 
is well-defined.
That is, there exists $j$ satisfying the conditions for the process ``in case of $i$ odd". 
\end{prop}
\begin{proof}
When we construct a natural label $\overline{L}\in\mathcal{L}^{D}$ from $L^{b}$, 
we put the same label as the one on a dotted edge in $L^{b}$ 
on the $(2i-1)$-th or $2i$-th dotted edge $E$ in $\overline{L}$ according to the existence 
of an extra $U$.
We denote the label on the edge $E$ by $l(E)$. At this moment, we have no label on the $2i$-th 
or $(2i-1)$-th dotted edges in $\overline{L}$.
Then, we put a label $\min(T')$ on the $2i$-th or $(2i-1)$-th 
dotted edge, respectively, where $T'$ is a subtree below the $2i$-th or $(2i-1)$-th 
dotted edge.

Let $r$ be the total number of dotted edges in $\overline{L}$.
When we have an extra $U$, $r$ is even.
Take the $2i$-th dotted edge $E$ enumerated by $1,2,\ldots,r$ from the root. 
This implies that the parity of $i$ is odd when we put labels on dotted edges without 
a label to construct $L$ from $\overline{L}$.
We first consider the case where the subtree below $2i-1$-th dotted edge 
has a unique dotted edge below its root.
Then, the $(2i-1)$-th dotted edge from the root has the label $\min(T')=l(E)$ and 
the label of the $2i$-th dotted edge is changed to $l(E)+1$.
Note that there exists no circled labels in $[l(E),l(E)+1]$. 
Thus the operation ``in case of $i$ odd" does not do anything on $\overline{L}$.
Secondly, we consider the case the subtree $T'$ below $2i-1$-th dotted edge has 
at least two edges (one edge is a dotted edge) below its root.
From Proposition \ref{prop:anc}, the edges whose ancestor is the $2i-1$-th dotted 
edge have incoming arrows.
Especially, note that the $2i$-th dotted edge has the outgoing arrow.
Let $l$ be the smallest circled label as in ``in case of $i$ odd".
Recall that $E_{r-2i+2}<l<E_{r-2i+1}$.
By the definition of the rules to put a circle on a label of edges in 
the tree for a ballot tiling (see Section \ref{sec:enumBT}), 
there exists a sequence of edges with arrows 
\begin{align}
\label{eq:lj}
l\leftarrow \ldots \leftarrow j,
\end{align}
with $j$ without a circle.
This $j$ cannot be circled due to the fact that there exists a smallest 
label in the sequence (\ref{eq:lj}) which cannot be circled.
Thus, there exists at least one $j$ satisfying the conditions in 
``in case of $i$ odd". 
This implies the map $\iota$ is well-defined.

On the other hand, when we do not have an extra $U$, $r$ is odd.
Let a dotted edge $E$ be the $2i-1$-th dotted edge from the root.
Then, we apply the operation ``in case of $i$ odd" to $E$.
When $i\ge2$, by the same argument as $r$ even case, the operation 
``in case of $i$ odd" is well-defined on $\overline{L}$.
Finally, we consider the $i=1$ case.
Since we do not have an extra $U$, we have a sequence of edges with incoming 
arrows starting from $E$.
Let $l$ be the smallest circled label as in ``in case of $i$ odd".
Recall that $0=E_{r+1}<l<E_{r}$. 
By applying the same argument as $r$ even again, one can easily show 
that the operation is well-defined. 
\end{proof}

\begin{cor}
The map $\iota$ is injective.
\end{cor}
\begin{proof}
It is obvious that if $L, L'\in\mathcal{L}^{b}(\lambda)$ and $L\neq L'$, then 
$\iota(L)\neq\iota(L')$.
\end{proof}

\begin{theorem}
\label{thrm:symDTSforBT}
Let $L^{b}\in\mathcal{L}^{\mathrm{b}}$ and $L:=\iota(L^{b})\in\mathcal{L}^{\mathrm{b}\subset\mathrm{D}}$.
Then, the application of the symmetric DTS bijection on $L$ gives a ballot tiling.
\end{theorem}
\begin{proof}
By construction and the proof of Proposition \ref{prop:iotawd}, 
it is straightforward to show that a natural label $\iota(L^{b})$
satisfies the two properties (P\ref{Enu:P1}) and (P\ref{Enu:P2}).
By the symmetric DTS bijection, we perform a ballot spreads when 
we have a dotted edge. 
Two properties (P\ref{Enu:P1}) and (P\ref{Enu:P2}) 
insure that we spread an even number of ballot tiles at once. 
In the operation ``in case of $i$ even", we change the labels of 
edges for any value $c$.
Let $p$ be the label of the dotted edge $E_{i}$ after the relabeling.
Then, the number of circled labels $l$ satisfying $E_{i+1}<l<p$ 
is zero, which means that, after the application of symmetric DTS bijection,
we have no ballot tiles of size $(2n+n')$ with $n'$ odd (see Section \ref{sec:bt}).
From these observations, we obtain a ballot tiling after the application 
of symmetric DTS bijection on $\iota(L^{b})$.
\end{proof}

\begin{remark}
Theorem \ref{thrm:symDTSforBT} gives a realization of DTS bijection for ballot tilings.
To find a variant of DTS bijection, which maps directly from a natural label for a ballot tiling 
to a ballot tiling  without embedding the set of ballot tilings into the set of symmetric 
Dyck tiling, is an open problem.
\end{remark}

\begin{example}
Let $\lambda=UDUDUDUUU$.
The tree associated with $\lambda$, a natural label with circles $L$, and the modified 
natural label with circles associated with $L$ are depicted in Figure \ref{fig:BTS1}. 
The number of circled labels is the same after the action of ``in case of $i$ even".
\begin{figure}[ht]
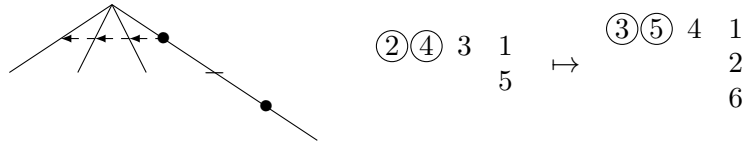

\begin{align*}
\tikzpic{-0.5}{[scale=0.6]
\coordinate
	child{coordinate(c2)}
	child{coordinate(c4)}
	child{coordinate(c3)}
	child{coordinate(c1)
		child[missing]
		child[missing]
		child[missing]
		child{coordinate(c5)}};
\node at(c1){$-$};
\node at ($(0,0)!0.5!(c1)$){$\bullet$};
\node at ($(c1)!0.5!(c5)$){$\bullet$};
\draw[dashed,-latex]($(0,0)!0.5!(c1)$)--($(0,0)!0.5!(c3)$);
\draw[dashed,-latex]($(0,0)!0.5!(c3)$)--($(0,0)!0.5!(c4)$);
\draw[dashed,-latex]($(0,0)!0.5!(c4)$)--($(0,0)!0.5!(c2)$);
}\qquad
\begin{matrix}
\circnum{2} & \circnum{4} & 3 & 1 \\
 & & & 5 
\end{matrix}
\quad
\mapsto
\quad  
\begin{matrix}
\circnum{3} & \circnum{5} & 4 & 1 \\
            &             &   & 2 \\
            &             &   & 6
\end{matrix}
\end{align*}
\caption{A tree, its natural label, and the modified natural label.}
\label{fig:BTS1}
\end{figure}
Then, we obtain the following growth diagram of ballot tilings by the 
symmetric DTS bijection.
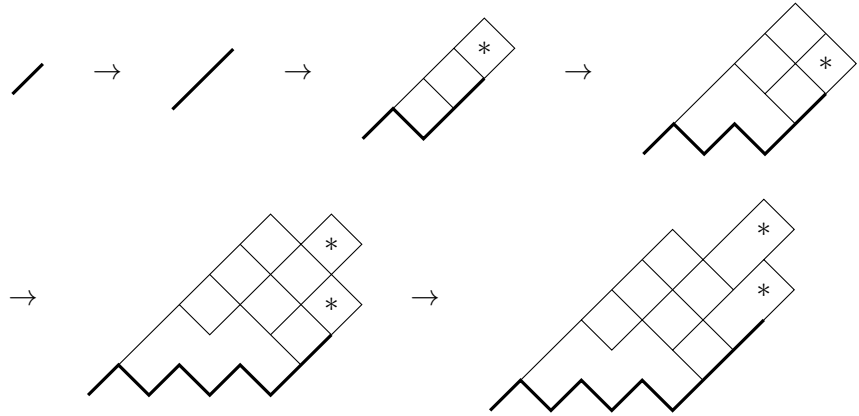
\begin{align*}
&
\tikzpic{-0.5}{[x=0.4cm,y=0.4cm]
\draw[very thick](0,0)--(1,1);
}
\quad\rightarrow\quad
\tikzpic{-0.5}{[x=0.4cm,y=0.4cm]
\draw[very thick](0,0)--(2,2);
}
\quad\rightarrow\quad
\tikzpic{-0.5}{[x=0.4cm,y=0.4cm]
\draw[very thick](0,0)--(1,1)--(2,0)--(4,2);
\draw(1,1)--(4,4)--(5,3)--(4,2)(2,2)--(3,1)(3,3)--(4,2);
\node at(4,3){$\ast$};
}
\quad\rightarrow\quad
\tikzpic{-0.5}{[x=0.4cm,y=0.4cm]
\draw[very thick](0,0)--(1,1)--(2,0)--(3,1)--(4,0)--(6,2);
\draw(1,1)--(5,5)--(7,3)--(6,2);
\draw(3,3)--(5,1)(4,4)--(6,2)(4,2)--(6,4);
\node at(6,3){$\ast$};
} \\[12pt]
&\rightarrow\quad
\tikzpic{-0.5}{[x=0.4cm,y=0.4cm]
\draw[very thick](0,0)--(1,1)--(2,0)--(3,1)--(4,0)--(5,1)--(6,0)--(8,2);
\draw(3,3)--(6,6)--(7,5)(5,3)--(8,6)--(9,5)--(8,4)(6,2)--(8,4)--(9,3)--(8,2);
\draw(1,1)--(3,3)--(4,2)--(5,3)--(7,1);
\draw(4,4)--(5,3)(5,5)--(8,2)(7,5)--(8,4);
\node at(8,5){$\ast$};\node at(8,3){$\ast$};
}
\quad\rightarrow\quad
\tikzpic{-0.5}{[x=0.4cm,y=0.4cm]
\draw[very thick](0,0)--(1,1)--(2,0)--(3,1)--(4,0)--(5,1)--(6,0)--(9,3);
\draw(1,1)--(3,3)--(4,2)--(5,3)--(7,1)(3,3)--(6,6)--(8,4)
 	(5,3)--(9,7)--(10,6)--(6,2)(9,5)--(10,4)--(7,1);
\draw(4,4)--(5,3)(5,5)--(8,2);
\node at(9,6){$\ast$};\node at(9,4){$\ast$};
}
\end{align*}
\end{example}

We give another example of the action of the inclusive map on a ballot tiling
to obtain a symmetric Dyck tiling.
\begin{example}
Let $\lambda=UDUDUDUUU$. The tree for $\lambda$ is same as the one in Figure \ref{fig:BTS1}.
We consider the following natural label and its modified natural label.
Note that the number of circled labels is increased by one by the operation in ``in case of $i$ odd".
\begin{align*}
\begin{matrix}
\circnum{3} & 2 & 1 & 4 \\
            &   &   & 5
\end{matrix}
\qquad \mapsto \qquad
\begin{matrix}
\circnum{3} & 2 & \circnum{1} & 4 \\
            &   &   & 5 \\
            &   &   & 6
\end{matrix}
\end{align*}
We obtain the growth diagram of ballot tilings by the symmetric DTS bijection
as in Theorem \ref{thrm:symDTSforBT}.
One can easily verify that these ballot tilings satisfy the conditions in 
Definition \ref{defn:bt}.
\begin{align*}
&
\tikzpic{-0.5}{[x=0.4cm,y=0.4cm]
\draw[very thick](0,0)--(1,1)--(2,0);
\draw(1,1)--(2,2)--(3,1)--(2,0);
\node at(2,1){$\ast$};
}
\quad\rightarrow\quad
\tikzpic{-0.5}{[x=0.4cm,y=0.4cm]
\draw[very thick](0,0)--(1,1)--(2,0)--(3,1)--(4,0);
\draw(1,1)--(3,3)--(5,1)--(4,0)(2,2)--(3,1)--(4,2);
\node at(4,1){$\ast$};
}
\quad\rightarrow\quad
\tikzpic{-0.5}{[x=0.4cm,y=0.4cm]
\draw[very thick](0,0)--(1,1)--(2,0)--(3,1)--(4,0)--(5,1)--(6,0);
\draw(1,1)--(4,4)--(7,1)--(6,0)(2,2)--(3,1)--(6,4)--(7,3)--(6,2)(3,3)--(5,1)--(6,2);
\node at(6,1){$\ast$};\node at(6,3){$\ast$};
}\\[12pt]
&\rightarrow\quad
\tikzpic{-0.5}{[x=0.4cm,y=0.4cm]
\draw[very thick](0,0)--(1,1)--(2,0)--(3,1)--(4,0)--(5,1)--(6,0)--(7,1);
\draw(1,1)--(4,4)--(6,2)(2,2)--(3,1)--(7,5)--(8,4)--(5,1)(3,3)--(5,1);
\draw(7,3)--(8,2)--(7,1);
\node at(7,2){$\ast$};\node at(7,4){$\ast$};
}
\quad\rightarrow\quad
\tikzpic{-0.5}{[x=0.4cm,y=0.4cm]
\draw[very thick](0,0)--(1,1)--(2,0)--(3,1)--(4,0)--(5,1)--(6,0)--(9,3);
\draw(1,1)--(4,4)--(6,2)(2,2)--(3,1)--(9,7)--(10,6)--(9,5)(3,3)--(5,1);
\draw(5,1)--(9,5)--(10,4)--(9,3);
\node at(9,4){$\ast$};\node at(9,6){$\ast$};
}
\end{align*}
\end{example}

\subsection{Top boundary of the symmetric DTS bijection}
In this subsection, we describe the procedure to produce the top 
boundary $\mu$ of a natural label $L$ for a fundamental symmetric 
Dyck tiling with the lower boundary path $\lambda$.

Let $\mathbf{p}:=(p_1,\ldots,p_n)$ be the insertion history of 
a symmetric Dyck tiling with the low boundary $\lambda$.
Recall that the label $i\in[1,n]$ may have a circle in $L$.
We construct a binary word $\sigma(L)$ consisting of $U$ and $D$ from 
$\mathbf{p}$ and $L$ as follows:
\begin{enumerate}[(a)]
\item Set $i=1$ and define $\sigma(L)$ as an empty word.
\item We insert $UD$ at the $p_{i}$-th position of $\sigma(L)$ from left 
when the edge labeled by $i$ in $L$ does not have a dot.
We append $U$ from right when the label $i$ in $L$ has a dot.
\item We move the inserted $D$ in (b) right to the right-most step.
We perform this process if such $D$ exists. 
\item If the edge labeled by $i$ in $L$ has a circle, we replace the right-most 
$D$ by $U$.
\item Increase $i$ by one and go to (b). The algorithm stops when $i=n+1$.
\end{enumerate}

With above notation, we have the following proposition.
\begin{prop}
The word $\sigma(L)$ gives the top boundary $\mu$ of the symmetric
Dyck tiling constructed from the natural label $L$ by the symmetric 
DTS bijection.
\end{prop}
\begin{proof}
The symmetric DTS bijection consists of three processes: a spread of a symmetric 
Dyck tile (Dyck spread and ballot spread), addition of a strip, and addition 
of a single box with $\ast$.

When we perform a spread of a tile, we insert a local path $UD$ onto a symmetric 
Dyck tiling. Thus, the process (b) corresponds to this insertion.

Suppose that $\nu$ be the top boundary of length $m$ after the processes (a) and (b),
and $\nu_{i}\in\{U,D\}$ be each step for $i\in[1,m]$.
Since we add a strip on a symmetric Dyck tiling, which starts from the $i$-th column,
we replace $\nu_{i}, \nu_{i+1}, \ldots, \nu_{m}$ by 
$\nu_{i+1}, \nu_{i+2},\ldots, \nu_{m}, \nu_{i}$.
This process is nothing but the process (c).

The addition of a single box with $\ast$ corresponds to replace the right-most 
$D$ which is the right-most element of $\sigma(L)$ by an up step $U$.
By the processes (b) and (c), it is obvious that the right-most element of 
$\sigma(L)$ is a down step. Thus, the process (d) is well-defined and equivalent
to the addition of a box with $\ast$.

Therefore, the word $\sigma(L)$ is equivalent to the top path of the symmetric 
Dyck tiling obtained by the symmetric DTS bijection.
\end{proof}

\section{Hermite histories and perfect matchings}
In this section, we study fundamental symmetric Dyck tilings in details by use of 
Hermite histories and perfect matchings.
These combinatorial objects are compatible with the symmetric DTS bijection 
studied in Section \ref{sec:symDTS}.
In case of (non-symmetric) Dyck tilings of size $n$, we have a corresponding 
Hermite histories and perfect matchings of size $n$ (see \cite{KMPW12}). 
One of the differences between non-symmetric Dyck tilings and fundamental symmetric Dyck 
tilings is that the latter may have circles in natural labels.
To connect symmetric Dyck tilings with Hermite histories and perfect matchings,
we embed symmetric Dyck tilings into symmetric Dyck tilings of larger size whose 
natural labels do not have circles.
Then, we make use of the correspondence between Hermite histories and perfect matchings
and symmetric Dyck tilings without circles.

\subsection{Symmetric Dyck tilings}
\label{sec:HhsymDT}
A (non-symmetric) Dyck tiling of size $n$ is bijective to an {\it Hermite history} of size $n$ 
(see \cite{KMPW12}). 
Further, one can obtain a perfect matching of $\{1,2,\ldots,2n\}$ from 
an Hermite history (see \cite{KMPW12}). 
In this subsection, we study the relations among symmetric Dyck tilings, Hermite histories,
and perfect matchings.
The main difference between non-symmetric Dyck tilings and symmetric Dyck tilings is that 
the size of a symmetric Dyck tiling can be different from the one of an Hermite history.
Actually, we construct a map from a symmetric Dyck tiling to an Hermite history 
of larger size. 
This discrepancy comes from the addition of a single box with $\ast$ in
the symDTS bijection.

Let $\lambda$ be a ballot path of length $(n,n')$, $T(\lambda)$ be a tree with $n+n'$ 
edges associated with $\lambda$, and $L$ be a natural label with circles of $T(\lambda)$.
Let $N$ be the number of circled label in $L$.
We will construct a tree $\overline{T(\lambda)}$ of size $n+n'+N$ and its natural 
label $\overline{L}$ of the tree $\overline{T(\lambda)}$ from $T(\lambda)$ and $L$. 
Note that we will construct a label $\overline{L}$ without circles.

Take a circled label $l$ in $L$ and erase the circle on $l$.
We increase the labels by one, which are bigger than or equal to $l$.
Here, if a label $k\neq l$ and $k> l$ has a circle, the label $k+1$ 
in a new natural label has a circle.

We add an edge with a dot whose label is $l$ as follows.
Let $l'$ be the maximum label on a edge with a dot that is smaller than $l$.
We have two cases according to the relative position of $l'$ with respect 
to the label $l+1$.
We denote by $E(l)$ the edge labeled by $l$ in the tree.
\begin{enumerate}
\item When the edge labeled by $l+1$ is just below the edge labeled by $l'$, 
{\it i.e.}, $E(l+1)\uparrow E(l')$.
We insert a dotted edge with the label $l$ just below the dotted edge with 
the label $l'$ at the same depth as the child edges of $E(l')$.
Here, the depth means the distance of an edge from the root.

\item When the edge labeled by $l+1$ is left to the edge labeled by $l'$, 
{\it i.e.}, $E(l+1)\rightarrow E(l')$.
We insert a dotted edge with the label $l$ just below the dotted edge with the 
label $l'$ and just above the child edges of $E(l')$.
\end{enumerate}
We call the operations (1) and (2) the addition of an edge with a dot.

\begin{remark}
\label{remark:addedge}
In the above procedure, the insertion of a dotted edge with label 
$l$ is well-defined operation on a tree for a symmetric Dyck tiling.
By construction of a tree, dotted edges form a subtree without ramifications, 
that is, the subtree has a single leaf.
Let $E$ be an edge with a dot in a tree. 
There is no edges right to any edge $E$ with a dot in the tree.
\end{remark}

We repeat the above procedure until we visit all circled labels from the smallest
to the largest.
As a result, we obtain a tree with dots of size $n+n'+N$ and 
its natural label without circles.
We denote the new tree and the new label by $\overline{T(\lambda)}$ and $\overline{L}$.

As in Section \ref{sec:symDTS}, we construct an insertion history 
$\overline{\mathbf{p}}:=(p_1,p_2,\ldots,p_{m})$ with $m=n+n'+N$ 
from the natural label $\overline{L}$.
We construct a {\it mini-word} $w$ from $\overline{\mathbf{p}}$, that 
is a $2\times m$ array of integers in $[1,2m]$.
Let $S_{2m}:=[1,2m]$. 
For $1\le i\le m$, we recursively define $a_{i}$ as the $p_{i}+1$-th smallest 
element in $S_{i}$, $b_{i}$ as the maximum integer $S_{i}$, and 
$S_{i-1}:=S_{i}\setminus\{a_{i},b_{i}\}$. 
By definition, we have $a_{m}:=p_{m}+1$ and $b_{m}:=2m$.
Then, the mini-word $w$ is a $2\times m$ array of integers such that 
the first row is $\mathbf{a}:=(a_{1},\ldots,a_{m})$ and the second row 
is $\mathbf{b}:=(b_{1},\ldots, b_{m})$.

We consider a graph with $2m$ nodes which lie in a line from 
left to right.
Given a pair $(a_{i},b_{i})\in w$ for $1\le i\le m$, 
we connect the $a_{i}$-th node and the $b_{i}$-th node by an arch.
A graph with $m$ arches is called a {\it perfect matching}.

For each pair $(a,b)$ of a perfect matching, the number of pairs 
$(c,d)$ such that $c<a<b<d$ is called the {\it nesting number} of $(a,b)$.
We record the nesting number of a pair $(a,b)$ below the $a$-th node 
in the perfect matching. 
Let $\{i_{1},\ldots,i_{N}\}$ be the set of labels with a circle in $L$.
We put a circle on the $i_{j}+j-1$-th nesting number in the perfect 
matching for $1\le j\le N$.

Recall that a symmetric Dyck tiling is a tiling between the lowest 
ballot path $\lambda$ and the top ballot path $\mu$.
We put the nesting numbers on the up steps of $\mu$ and this labeled 
ballot path is called an {\it Hermite history}.
See \cite{KMPW12} for the definition of an Hermite history for a (non-symmetric)
Dyck tiling.

An Hermite history can also be constructed from a symmetric Dyck tiling
as follows.
For a ballot tile $B$, we define the {\it entry} (resp. {\it exit}) 
as the northwest (resp. southeast) edge of the leftmost (resp. rightmost)
box of $B$.
We connect the entry and the exit of the tile $B$ by a line which lies 
inside $B$.
Given a symmetric Dyck tiling, we consider a path starting from an up step 
$u$ of $\mu$. If $u$ is not an entry of a ballot tile, then the label of 
$u$ is zero.
If $u$ is an entry of a ballot tile, we extend the path by passing through 
the ballot tiles and arrive at an edge which is not an entry.
Then, the label of $u$ is the number of tiles which the path starting from $u$
travels. Note that each tile is counted exactly once.
If the path starting from $u$ passes through a ballot tile with $\ast$, 
we put a circle on the label of $u$.

We have two definitions of a Hermite history: one is by the nesting numbers
of a perfect matching, and the other is by the numbers of tiles associated with 
up steps on the top path in a symmetric Dyck tiling.
\begin{theorem}
\label{thrm:Hh}
The above two definitions of a Hermite history coincides with each other.
\end{theorem}

\begin{example}
We consider the same natural label as in Figure \ref{fig:symDTS}.
See Figure \ref{fig:HhsymDT} for the perfect matching and the Hermite 
history for this label.
Note that nesting numbers with a circle in the perfect matching correspond
to ballot tiles with $\ast$ in the Hermite history.
\begin{figure}[ht]
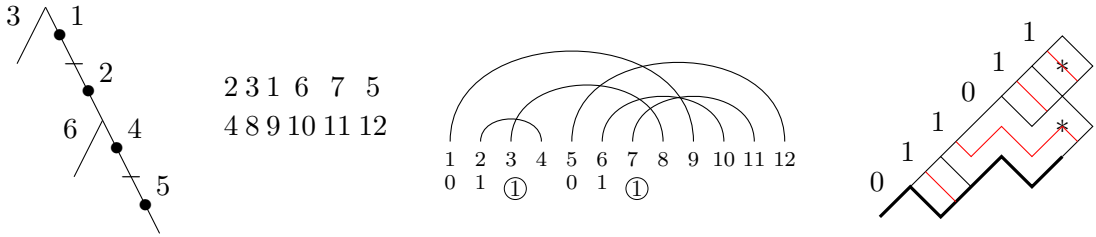

\begin{align*}
\tikzpic{-0.5}{[scale=0.5]
\coordinate
	child{coordinate(c3)}
	child{coordinate(c1)
		child[missing]
		child{coordinate(c2)
			child{coordinate(c6)}
			child{coordinate(c4)
				child[missing]
				child{coordinate(c5)}
			}
		}	
	};
\node at(c1){$-$};\node at(c4){$-$};
\draw[anchor=south east]($(0,0)!.5!(c3)$)node{$3$}($(c2)!.5!(c6)$)node{$6$};
\draw[anchor=south west]($(0,0)!.5!(c1)$)node{$1$}($(c1)!.5!(c2)$)node{$2$}
			($(c2)!.5!(c4)$)node{$4$}($(c4)!.5!(c5)$)node{$5$};
\draw[anchor=center]($(0,0)!.5!(c1)$)node{$\bullet$}($(c1)!.5!(c2)$)node{$\bullet$}
			($(c2)!.5!(c4)$)node{$\bullet$}($(c4)!.5!(c5)$)node{$\bullet$};
}\quad
\genfrac{}{}{0pt}{}{2}{4}
\genfrac{}{}{0pt}{}{3}{8}
\genfrac{}{}{0pt}{}{1}{9}
\genfrac{}{}{0pt}{}{6}{10}
\genfrac{}{}{0pt}{}{7}{11}
\genfrac{}{}{0pt}{}{5}{12}
\quad
\tikzpic{-0.5}{[x=0.4cm,y=0.4cm]
\draw(0,0)..controls(0,4)and(8,4)..(8,0);
\draw(1,0)..controls(1,1)and(3,1)..(3,0);
\draw(2,0)..controls(2,2.5)and(7,2.5)..(7,0);
\draw(4,0)..controls(4,3.5)and(11,3.5)..(11,0);
\draw(5,0)..controls(5,2)and(9,2)..(9,0);
\draw(6,0)..controls(6,2)and(10,2)..(10,0);
\node[anchor=north] at(0,0){$\genfrac{}{}{0pt}{}{1}{0}$};
\node[anchor=north] at(1,0){$\genfrac{}{}{0pt}{}{2}{1}$};
\node[anchor=north] at(2,0){$\genfrac{}{}{0pt}{}{3}{
\scalebox{0.8}{\tikzpic{-0.5}{\node[circle,draw,inner sep=0.5pt]at(0,0){$1$};}}
}$};
\node[anchor=north] at(4,0){$\genfrac{}{}{0pt}{}{5}{0}$};
\node[anchor=north] at(5,0){$\genfrac{}{}{0pt}{}{6}{1}$};
\node[anchor=north] at(6,0){$\genfrac{}{}{0pt}{}{7}{
\scalebox{0.8}{\tikzpic{-0.5}{\node[circle,draw,inner sep=0.5pt]at(0,0){$1$};}}
}$};
\node[anchor=north] at(3,0){$\genfrac{}{}{0pt}{}{4}{}$};
\node[anchor=north] at(7,0){$\genfrac{}{}{0pt}{}{8}{}$};
\node[anchor=north] at(8,0){$\genfrac{}{}{0pt}{}{9}{}$};
\node[anchor=north] at(9,0){$\genfrac{}{}{0pt}{}{10}{}$};
\node[anchor=north] at(10,0){$\genfrac{}{}{0pt}{}{11}{}$};
\node[anchor=north] at(11,0){$\genfrac{}{}{0pt}{}{12}{}$};
}
\quad
\tikzpic{-0.5}{[x=0.4cm,y=0.4cm]
\draw[very thick](0,0)--(1,1)--(2,0)--(4,2)--(5,1)--(6,2);
\draw(1,1)--(2,2)--(3,1);
\draw(2,2)--(4,4)--(5,3)--(6,4)--(7,3)--(6,2);
\draw(4,4)--(6,6)--(7,5)--(6,4)--(5,5);
\node at (6,3){$\ast$};
\node at (6,5){$\ast$};
\node[anchor=south east]at(0.5,0.5){$0$};
\node[anchor=south east]at(1.5,1.5){$1$};
\node[anchor=south east]at(2.5,2.5){$1$};
\node[anchor=south east]at(3.5,3.5){$0$};
\node[anchor=south east]at(4.5,4.5){$1$};
\node[anchor=south east]at(5.5,5.5){$1$};
\draw[red](1.5,1.5)--(2.5,0.5)(2.5,2.5)--(3,2)--(4,3)--(5,2)--(6,3)--(6.5,2.5);
\draw[red](4.5,4.5)--(5.5,3.5)(5.5,5.5)--(6.5,4.5);
}
\end{align*}
\caption{The natural label on $\overline{T(\lambda)}$ 
with $\lambda=UDUUDU$ (the first picture), 
the mini-word (the second picture), the perfect matching of size $12$ 
(the third picture), and the Hermite history (the fourth picture). 
}
\label{fig:HhsymDT}
\end{figure}
\end{example}

Before proceeding to the proof of Theorem \ref{thrm:Hh}, we introduce 
several Propositions in the following two subsections.
We will consider Hermite histories without 
circles in Section \ref{sec:Hwoc}, and Hermite histories with circles 
in Section \ref{sec:Hwc}.
We prove Theorem \ref{thrm:Hh} in Section \ref{sec:prHh}.

\subsection{Hermite histories for trees without circles}
\label{sec:Hwoc}
In this subsection, we consider the trees without circles.
In terms of symmetric Dyck tilings, 
this is equivalent to consider symmetric Dyck tilings 
such that they do not have a ballot tile with $\ast$.

Given a word (or equivalently a permutation) $w:=(w_1,\ldots,w_{n})$ on $[1,n]$, 
we define sequences of non-negative integers 
$\mathrm{inv}(w):=(d_1,\ldots,d_n)$ and $\mathrm{inv'}(w):=(d'_1,\ldots,d'_n)$ by
\begin{align*}
d_{i}&:=\#\{ j | w_{i}>w_{j} \text{ and } w_{j} \text{ is right to } w_{i} \}, \\
d'_{i}&:=\#\{j | w_{j}>w_{i} \text{ and } w_{j} \text{ is left to } w_{i} \}.
\end{align*} 
We call $\mathrm{inv}(w)$ (resp. $\mathrm{inv'}(w)$) the {\it inversion} 
(resp. {\it reverse inversion}) sequence of $w$, 
and denote the sum of the entries of $\mathrm{inv}(w)$ by 
\begin{align*}
\mathrm{Inv}(w)&:=\sum_{i=1}^{n}d_{i}, \\
\mathrm{Inv'}(w)&:=\sum_{i=1}^{n}d'_i.
\end{align*}

It is straightforward from the definition that 
$\mathrm{inv}$ and $\mathrm{inv'}$ satisfy the following relations.
Given an integer sequence $\mathbf{b}:=(b_1,\ldots,b_n)$, we define
$\mathrm{Rev}(\mathbf{b}):=(b_n,\ldots,b_{1})$.
\begin{prop}
\label{prop:Rev}
Let $w:=(w_1,w_2,\ldots,w_{n})$ be a permutation and
$\bar{w}:=(\bar{w}_1,\ldots,\bar{w}_{n})$ be also a permutation 
given by $\bar{w}_{i}=n+1-w_{n+1-i}$ for $i\in[1,n]$.
We have 
\begin{align*}
\mathrm{inv}(w)=\mathrm{Rev}(\mathrm{inv'}(\bar{w})).
\end{align*} 
\end{prop}

Given a natural label $L$ of a tree, we denote by $w(L)$ 
the word obtained from $L$ by reading the labels in the pre-order.
Here, the pre-order means that we visit a node before both of its left
and right subtrees.
By construction, we have a permutation $w$ on $[1,n]$.
The following Lemma can be directly obtained from the same argument 
as in the case of non-symmetric Dyck tilings.
Note that since we have no circles in the natural label $L$, the corresponding 
symmetric Dyck tiling is equivalent to the non-symmetric Dyck tiling obtained 
from $L$ by ignoring dots.
\begin{lemma}[Theorem 5 in \cite{KMPW12}]
\label{lemma:wtDT}
Let $\mathcal{D}$ be a symmetric Dyck tiling and $d$ be 
a Dyck tile in $\mathcal{D}$.
We denote by $w(L)$ the pre-order word for the natural label $L$
corresponding to $\mathcal{D}$ by the symmetric Dyck tiling strip.
The weight of $\mathcal{D}$ is written in terms of $\mathrm{Inv}$ by
\begin{align*}
\sum_{d\in \mathcal{D}}\mathrm{wt}(d)=\mathrm{Inv}(w(L)).
\end{align*}
\end{lemma}

Recall that we denote by $E(i)$ the edge with label $i$ in a label $L$.
Here, the label $L$ may be neither a natural label nor a decreasing label. 
We define a set $I'$ of labels as 
\begin{align*}
I':=\{i<n | E(n)\rightarrow E(i) \text{ in }L\},
\end{align*}
and an another set of labels $J$ as 
\begin{align*}
J:=\{j<n | E(j)\rightarrow E(n) \text{ in } L \}.
\end{align*}
Note that there exits no edge labeled by $j$ with a dot such 
that $E(j)\rightarrow E(p)$ and the edge $E(p)$ does not 
have a dot.

\begin{defn}
\label{defn:kappa}
We define an operation $\kappa_{n}$ on a pre-order word of a natural tree
$L$ as follows.
\begin{enumerate}[(a)]
\item Start with $p=1$, a word $w_{0}:=w(L)$, $r=0$ and $I''=I'$.
\item Find the $p$-th largest label in $J$ and denote it by $j_{p}$.
\item Find the largest label $i_{p}\in I''$ such that 
\begin{enumerate}[(i)]
\item $i_{p}<j_{p}$, 
\item there exists no $i'\in I'$ satisfying 
$E(j_p)\rightarrow E(i')\rightarrow E(i_p)$ for $i'<i_p$. 
\end{enumerate}
If such $i_{p}$ exists, replace $I''$ by $I''\setminus\{i_p\}$ 
and go to (d). Otherwise, go to (e).
\item Increase $r$ by one.
\item If we visit all elements in $J$, go to (f).
Otherwise, replace $p$ by $p+1$ and go to (b).

\item 
We move the element $n$ right by $r$ steps in $w_{0}$.
We denote the new word by $w_{1}$. 
\end{enumerate}
We denote by $\kappa_{n}(w):=w_{1}$ the word obtained by the above procedure.
We define $\kappa_{i}(w)$ for $i<n$ similarly for the subtree consisting 
of the labels in $[1,i]$.
\end{defn}

We extend the definition of $\kappa_{n}$ to a map on a pair of a word $w$ 
and a natural label $L$ 
\begin{align*}
\kappa_{i}(w,L):=(\kappa_{i}(w),L),
\end{align*}
where $i\in[1,n]$.
Here, the underlying tree of $L$ is the tree $T$. 

We define an operation $\phi$ on $w$ by successive actions of $\kappa_{i}$
with $1\le i\le n$ as follows.
\begin{defn}
\label{defn:phi}
Let $L_{0}$ be a natural label and $w$ be the pre-order word of $L_{0}$.
Let $\kappa_{i}(w)$ be a map on a pair of a word and a label defined as above.
Then, we define the action of $\phi$ on $w$ as
\begin{align*}
\phi(w):=\kappa_1\circ\kappa_{2}\circ\ldots\circ\kappa_{n}(w,L_{0}).
\end{align*}
We simply denote by $\phi(w)$ the word obtained from $(w,L_{0})$ by 
successive actions of $\kappa_{i}$ for $1\le i\le n$.
\end{defn}

\begin{example}
Figure \ref{fig:exphi} shows an example of a natural label.
\begin{figure}[ht]
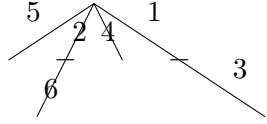

\tikzpic{-0.5}{[scale=0.5]
\coordinate
	child{coordinate(c5)}
	child{coordinate(c2)
		child[missing]
		child{coordinate(c6)}
		child[missing]
		child[missing]}
	child{coordinate(c4)}
	child{coordinate(c1)
		child[missing]
		child[missing]
		child[missing]
		child{coordinate(c3)}};
\node at(c2){$-$};\node at(c1){$-$};
\draw[anchor=south west]($(0,0)!.5!(c1)$)node{$1$}($(c1)!.5!(c3)$)node{$3$};
\draw($(0,0)!.5!(c4)$)node{$4$};
\draw($(0,0)!.5!(c2)$)node{$2$}($(c2)!.5!(c6)$)node{$6$};
\draw[anchor=south east]($(0,0)!.5!(c5)$)node{$5$};
}
\caption{A natural label with the pre-order word $526413$.
}
\label{fig:exphi}
\end{figure}
The actions of $\kappa_{i}$ with $i=1,2,3$ and $5$ are trivial.
Thus we have 
\begin{align*}
526413\rightarrow_{6}524613\rightarrow_{4}521643,
\end{align*}
where $\rightarrow_{i}$ is the action of $\kappa_{i}$.
As for $\kappa_{4}$, we have a pair $(i_1,j_1)=(1,2)$ and $r=1$.
Note that in the process (f), we move $4$ right by $r=1$, but 
we do not move the positions of $5$ and $6$.
Therefore, $\phi(w)=521643$.

Suppose that the tree whose pre-order word $w=526413$ is a tree with $6$ leaves,
{\it i.e.}, the lower path is $(UD)^{6}$.
Then, when $n=6$, we have two pairs of $i_p$ and $j_p$: $(i_1,j_1)=(4,5)$ and 
$(i_2,j_2)=(1,2)$. 
Then, we have $\phi(w)=521463$.
Compare the result in the previous paragraph.
The action of $\phi$ is affected by the tree structure.
\end{example}

\begin{remark}
It is clear from Figure \ref{fig:exphi} 
that labels after the actions of $\kappa_i$ on $w$ may not
be compatible with a natural label on the tree.
\end{remark}

Recall that we have two operations, a Dyck spread and addition of 
a strip, in the symmetric DTS operation.
We associate two sets of labels $S_{SE}(d)$ and $S_{SW}(d)$ 
with a Dyck tile $d$ in the symmetric Dyck tiling. 
Note that we do not need to consider ballot tiles since we restrict 
ourselves to the case of natural labels without circles.
The set $S_{SW}(d)$ is the set of labels $l$ such that 
the size of the Dyck tile $d$ is increased by one by the $l$-th 
Dyck spread or addition of a strip.

Given a Dyck tile $d$, we have a unique path of the Hermite history 
passing through $d$.
Recall that the path of the Hermite history is associated with an up step in the 
top ballot path and also the up step in the lower ballot path.
This up step in the lower ballot path 
corresponds to an edge $E$ in the natural label $L$.
We define $S_{SE}(d)$ is the set of a single label of $E$ in $L$.
Since a path of the Hermite history passes through a Dyck tile $d$ 
exactly once, the cardinality of $S_{SE}(d)$ is one, {\it i.e.,}
$|S_{SE}(d)|=1$. 
Figure \ref{fig:Hh} gives an example of a symmetric DTS bijection 
and $S_{SW}$ and $S_{SE}$ for a few Dyck tiles.

\begin{figure}[ht]
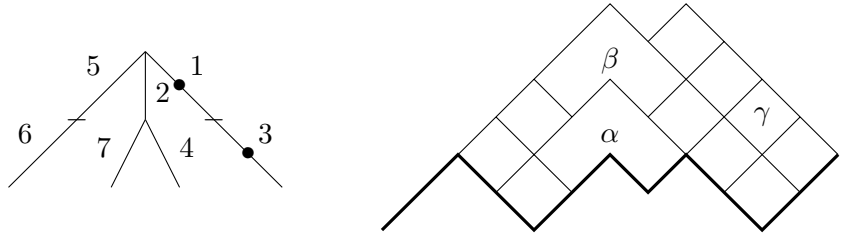

\tikzpic{-0.5}{[scale=0.6]
\coordinate
	child{coordinate(c5)
		child{coordinate(c6)}
		child[missing]
		child[missing]}
	child{coordinate(c2)
		child{coordinate(c7)}
		child{coordinate(c4)}}
	child{coordinate(c1)
		child[missing]
		child[missing]
		child{coordinate(c3)}};
\draw[anchor=south west]($(0,0)!.5!(c1)$)node{$1$}($(c1)!.5!(c3)$)node{$3$};
\draw[anchor=west]($(0,0)!.6!(c2)$)node{$2$};
\draw[anchor=south east]($(c2)!.7!(c7)$)node{$7$};
\draw[anchor=south west]($(c2)!.7!(c4)$)node{$4$};
\draw[anchor=south east]($(0,0)!.5!(c5)$)node{$5$}($(c5)!.5!(c6)$)node{$6$};
\draw[anchor=center]($(0,0)!.5!(c1)$)node{$\bullet$}($(c1)!.5!(c3)$)node{$\bullet$};
\node at(c1){$-$};\node at(c5){$-$};
}
\qquad
\tikzpic{-0.5}{[scale=0.5]
\draw[very thick](0,0)--(2,2)--(4,0)--(6,2)--(7,1)--(8,2)--(10,0)--(12,2);
\draw(2,2)--(6,6)--(11,1)(3,3)--(5,1)(4,4)--(5,3)(3,1)--(6,4)--(8,2);
\draw(7,5)--(8,6)--(12,2)(9,5)--(7,3)(10,4)--(8,2)(11,3)--(9,1);
\draw(6,2.5)node{$\alpha$}(6,4.5)node{$\beta$}(10,3)node{$\gamma$};
}
\caption{An example of the symmetric DTS bijection. 
For Dyck tiles $\alpha,\beta$ and $\gamma$, 
we have $S_{SW}(\alpha)=\{5,7\}, S_{SE}(\alpha)=\{4\}$, 
$S_{SW}(\beta)=\{6,7\}, S_{SE}(\beta)=\{1\}$, 
$S_{SW}(\gamma)=\{5\}$ and $S_{SE}(\gamma)=\{3\}$.
}
\label{fig:Hh}
\end{figure}

Given two Dyck tiles $d_1$ and $d_2$,
let $S_{SW}(d_1):=\{i_1<i_2<\ldots<i_{n}\}$ and 
$S_{SW}(d_2):=\{j_1<j_2<\ldots<j_{m}\}$.
By definition and construction, we have no pair of $(d_1,d_2)$ 
such that $S_{SW}(d_1)\subseteq S_{SW}(d_2)$ as sets
when $d_1$ and $d_2$ are not single boxes as Dyck tiles.
Then, we introduce a lexicographic order by 
\begin{align*}
S_{SW}(d_1)<S_{SW}(d_2), 
\end{align*}
if and only if 
\begin{align*}
i_{k}=j_{k} \text{ and } i_{k+1}<j_{k+1},
\end{align*}
for some $1\le k$.

We define an operation $\psi$ on the pre-order word $w$ as follows.
The algorithm consists of five steps:
\begin{enumerate}[(a)]
\item Set $p=1$.
\item Find a set $S_{SW}$ satisfying $|S_{SW}|\ge2$ 
such that $S_{SW}$ is the $p$-th smallest among them in the lexicographic order.
\item Set $\widetilde{S_{SW}}=S_{SW}\setminus\{\min(S_{SW})\}$.
Find the left-most $w_{q}$ right to all elements of $S_{SW}$ in $w$ 
such that $w_{q}<s$ for all $s\in\widetilde{S_{SW}}$.
\item Move $w_{q}$ left in $w$ by $|\widetilde{S_{SW}}|$.
\item Replace $p$ by $p+1$. Then, go to (b). The algorithm stops when 
we visit all the sets $S_{SW}$ with $|S_{SW}|\ge2$.
\end{enumerate}

\begin{defn}
\label{defn:psi}
We denote by $\psi(w)$ the word obtained from the pre-order word $w(L)$ for 
a natural label $L$ by the procedures (a) to (e).
\end{defn}

\begin{example}
We consider the same example as Figure \ref{fig:Hh}.
We have two Dyck tiles of size $2$, and $S_{SW}(\alpha)<S_{SW}(\beta)$ 
since $\min(S_{SW}(\alpha))<\min(S_{SW}(\beta))$ .
Thus, we have a sequence of words starting from 
the pre-order word $5627413$:
\begin{align*}
5627413\quad \rightarrow_{\alpha} \quad 5624713 \quad \rightarrow_{\beta} \quad 5624173.
\end{align*}
where $\rightarrow_{X}$ indicates the operation with respect to $S_{SW}(X)$ for $X=\alpha,\beta$.
Therefore, we have $\psi(w)=5624173$. 
\end{example}

Recall that given a Dyck tile $D$, $S_{SE}(D)$ consists of a single label. 
Then, we have the following lemma.
\begin{lemma}
\label{lemma:psic}
In the step (c) just above Definition \ref{defn:psi}, we have $w_{q}=S_{SE}$.
\end{lemma}
\begin{proof}
Let $d$ be a Dyck tile of size $s\ge2$ and 
$S_{SW}(d):=\{i_1<i_2<\ldots<i_{s}\}$.
By construction we have $|S_{SW}(d)|\ge2$, $|S_{SE}(d)|=1$,
and $S_{SE}(d):=\{j\}<\min(S_{SW}(d))$.

We denote by $E(k)$ the edge with the label $k$ in $L$.
When we perform a Dyck spread on a symmetric Dyck tiling, 
there exist three labels $a<b<c$ such that 
$E(b)\rightarrow E(c)\rightarrow E(a)$ in the natural label $L$.

By an addition of a strip in the symmetric DTS, 
we add a single boxes $d$ on an up step in a Dyck tiling. 
Then, the two sets $S_{SW}(d):=\{i_1\}$ and $S_{SE}(d):=\{j\}$ consist of 
a single label, and it is obvious $i_1>j$.
For a non-trivial Dyck tile $d$, we add labels $i_{l}>i_1$ with $l\ge2$ to 
$S_{SW}(d)$ since we enlarge a single box by a Dyck spread.
Further, this implies that the edge $E(i_{l})$ with $l\ge2$ is right to the 
edge $E(i_1)$ and is left to the edge $E(j)$ in the natural label. 
By summarizing the above observations, we have, in the natural label $L$, 
\begin{align*}
E(i_1)\rightarrow E(i_l)\rightarrow E(j), \quad \forall l\in[2,s].
\end{align*}
Thus, the label $j$ is right to the labels $i_l$ for $l\in[2,s]$ 
in the labeled tree $L$.

Suppose that $w_{q}\neq j$.
From the condition in (c), $w_q<i_{l}$ for all $l\in[1,s]$.
Since the word $w$ is read from $L$ by use of the pre-order,
we have 
\begin{align}
\label{eqn:SW}
E(i_l)\rightarrow E(w_q) \quad \forall l\in[1,s].
\end{align}
It is obvious from the definition of the symmetric DTS bijection,
Eqn. (\ref{eqn:SW}) implies that there exists a Dyck tile $d'$ 
such that $S_{SW}(d')\supseteq\{i_1<\ldots<i_s\}$ and $S_{SE}(d')=\{w_q\}$.
However, this contradicts the fact that there exists no pair 
of Dyck tiles $(d,d')$ such that $S_{SW}(d)\subseteq S_{SW}(d')$.
Thus we have $w_{q}=S_{SE}(d)$.
\end{proof}

\begin{lemma}
\label{lemma:phipsi}
We have $\phi(w)=\psi(w)$.
\end{lemma}
\begin{proof}
The action of $\kappa_{m}$ for $1\le m\le n$ on a pair of a word and 
a label $(w,L)$ is not trivial when there exists at least one 
pair of $(i_p,j_p)$ satisfying the step (c) in Definition \ref{defn:kappa}.
In the language of a symmetric Dyck tiling, we have a Dyck tile
of size larger than one.
The choice of $i_p$ and $j_p$ in the step (c) in Definition \ref{defn:kappa} 
corresponds to $\min(S_{SW}(d))$ and $S_{SE}(d)$ for some Dyck tile $d$.
We move $m$ right by one element in $w$ in both cases.
We decrease $m$ one by one starting from $m=n$. 
This implies we visit all elements $S_{SW}(d)\setminus\{\min(S_{SW}(d))\}$
and on the other hand we visit all pairs $(i_p,j_p)$.
It is obvious that there is a one-to-one correspondence between them.
Thus, the actions of $\phi$ and $\psi$ on $w$ should coincide with each 
other.
\end{proof}

\begin{prop}
Let $\mathcal{D}$ be a symmetric Dyck tiling and $w$ be pre-order word
for a natural label $L$ corresponding to $\mathcal{D}$.
We have 
\begin{align*}
\mathrm{tiles}(\mathcal{D})&=\mathrm{Inv}(\phi(w)), \\
&=\mathrm{Inv}(\psi(w)),
\end{align*}
where $\mathrm{tiles}(\mathcal{D})$ is the number of Dyck tiles 
in $\mathcal{D}$.
\end{prop}
\begin{proof}
From Lemma \ref{lemma:phipsi}, it is enough to show that 
$\mathrm{tiles}(\mathcal{D})=\mathrm{Inv}(\phi(w))$.

Recall that the map $\phi$ is a composition of $\kappa_{i}$ 
for $1\le i\le n$.
By the action of $\kappa_{i}$ on a pair of a word $w$ and a label $L$,
we move $i$ right by $r$ steps in $w$ where $r$ is defined in step (d) 
in Definition \ref{defn:kappa}.
This operation corresponds to an exchange of labels in $L$ and decrease 
$\mathrm{inv}(w)$ by $r$ and equivalently decrease the size of $r$ Dyck 
tiles of size $s>1$ by one.
It is obvious that this operation stops when the size of the corresponding 
Dyck tile becomes one.
From Lemma \ref{lemma:wtDT}, $\mathrm{Inv}(w)$ is the sum of the sizes of 
all Dyck tiles in a symmetric Dyck tiling.
Therefore, we have $\mathrm{tiles}(\mathcal{D})=\mathrm{Inv}(\phi(w))$.
\end{proof}

\begin{remark}
Given a labeled tree $L$, we have the pre-order word $w$.
By the action of $\phi$ or $\psi$, we have the word $\phi(w)=\psi(w)$.
Then, the sequence of non-negative integers 
$\mathrm{inv}(\psi(w)):=(d_1,\ldots,d_n)$ counts the numbers of 
boxes added in the addition of a strip at the $i$-th steps in the 
symmetric DTS bijection. 
More precisely, the integer $d_{i}$ is given by the number of 
the set $S_{SW}$ such that $\min(S_{SW})=i$, {\it i.e.},
\begin{align*}
d_i=\#\{ d | \min(S_{SW}(d))=i \text{ for a Dyck tile } d \}.
\end{align*}  
\end{remark}

Let $w$ be the pre-order word for a natural label $L$.
\begin{lemma}
\label{lemma:invpsi}
The reverse inversion sequence $\mathrm{inv'}(\psi(w)):=(d'_1,\ldots,d'_n)$ gives 
the numbers of Dyck tiles associated with the $i$-th up step of the top boundary 
$\mu$. 
\end{lemma}
\begin{proof}
The word $\psi(w)$ is constructed from $w$ by moving $S_{SE}(d)$ to left in $w$ 
corresponding to a Dyck tile of size $s\ge2$.
Thus, the number $d'_{i}$ counts the number of Dyck tiles added by the addition 
of a strip after $i$-th process of the symmetric DTS bijection.
This implies that $d'_{i}$ counts the number of Dyck tiles associated with 
the edge labeled by $i$ in $\lambda$ at one hand, and equivalently the number 
of Dyck tiles associated with the $i$-th up step of the top boundary $\mu$.
\end{proof}

Let $\mathbf{p}_{n}=(p_1,\ldots,p_n)$ be an insertion history for a natural
label $L$.
For a word $w:=(w_1,\ldots,w_n)$ of length $n$, we call the left of $w_1$ is the 
zero-th blank, the space between $w_{i}$ and $w_{i+1}$ the $i$-th blank
and the right of $w_{n}$ is the $n$-th blank. 

We construct a word $w'(\mathbf{p})$ with $\ast$ from $\mathbf{p}$ by the following 
recursive procedure. 

\begin{enumerate}[(a)]
\item $w'(\mathbf{p})=1$ if $p_1=0$ and $n=1$.
\item Let  $w$ be a word for the insertion history $\mathbf{p}_{n-1}$.
We put $n$ in $w$ at the $p_n$-th blank from left if $p_{n}\le n-1$.
If $p_{n}\ge n$, we put $p_{n}-n$ $*$'s right to $w$ and put $n$ right to 
the right-most $*$.
\end{enumerate}

\begin{lemma}
\label{lemma:w'psi}
Let $w'(\mathbf{p})$ be a word as above. Then, the word 
obtained from $w'(\mathbf{p})$ by removing $*$'s is nothing but
$\psi(w_{0})$ where $w_{0}$ is the pre-order word for 
the natural label $L$.
\end{lemma}
\begin{proof}
Let $\mathbf{p'}$ be an integer sequence defined by 
$\mathbf{p'}:=(p'_1,\ldots,p'_{n})=\mathbf{p}_{0}-\mathbf{p}$ where 
we define $\mathbf{p}_{0}:=(0,1,2,\ldots,n-1)$.
Then, $p'_{i}$ is given by 
\begin{align*}
p'_{i}=\#\{j<i|E(i)\rightarrow E(j)\}-\#\{j<i|E(j)\rightarrow E(i)\},
\end{align*}
where $E(j)$ is the edge labeled by $j$ in the natural label $L$.
Note that the element of $\mathbf{p'}$ may be a negative integer though 
all the elements of $\mathbf{p}$ are non-negative integers.
We construct a word $w''(\mathbf{p})$ from $\mathbf{p'}$ by the following 
recursive procedures.
\begin{enumerate}[(a')]
\item $w''(\mathbf{p})=1$ if $p'_1=0$ and $n=1$.
\item Let $y$ be a word for $(p'_1,\ldots,p'_{n-1})$. 
Then, we put $n$ in $y$ at the $p'_{n}$-th blank from right. 
If $p'_{n}$ is a negative integer, we put $*$'s left to $n$ if necessary.
And we redefine the right blank of $n$ as the $p'_{n}$-th blank if 
$p'_{n}$ is negative.
\end{enumerate}

\begin{claim}
We have $w'(\mathbf{p})=w''(\mathbf{p})$.
\label{claim:w''}
\end{claim}
\begin{proof}[Proof of Claim \ref{claim:w''}]
By construction, we have $\mathbf{p}+\mathbf{p'}=(0,1,2,\ldots,n-1)$.
We insert an integer $i$ according to $\mathbf{p}$ (resp. $\mathbf{p'}$)
from left (resp. right).
The positions of $i$ inserted by the steps (b) and (b') are the same.
Therefore, the word constructed from $\mathbf{p}$ coincides with the one 
constructed from $\mathbf{p'}$.
\end{proof}

It is enough to show that $w''(\mathbf{p})$ is the same (up to ignoring of $\ast$'s) 
as $\phi(w_{0})$ from Lemma \ref{lemma:phipsi}.
Recall that $\phi$ is defined by the compositions of $\kappa_{i}$ 
for $1\le i\le n$ (see Definition \ref{defn:phi}).
For each $\kappa_{i}$,
we visit all the pairs $(i_p,j_p)$ for $i_p\in I'$ and $j_p\in J$ satisfying 
the condition (c) in Definition \ref{defn:kappa}.
Further, we delete an element of $I''$ when there exists a pair of 
$(i_p,j_p)$ for $i_p\in I'$ and $j\in J$ satisfying the condition 
in the step (c).
After we visit all the pairs $(i_p,j_p)$ associated with $I'$ and $J$, 
the number of elements right to $i$ is equal to $p'_{i}$ if $p'_{i}$ 
is non-negative and to zero if $p'_{i}$ is negative.
When $p'_{i}$ is negative, we add $\ast$'s for $w''(\mathbf{p})$ and 
append $i$ at the right-most position.
From these observations, it is easy to see that $w''(\mathbf{p})$ gives 
the same word as $\phi(w_0)$ by ignoring $\ast$'s.
This completes the proof of Lemma \ref{lemma:w'psi}.
\end{proof}

\begin{example}
We consider the same example as Figure \ref{fig:Hh}.
The insertion history is given by $\mathbf{p}=(0,0,3,1,0,1,5)$ and 
$\mathbf{p'}=(0,1,-1,2,4,4,1)$.
Thus, $\mathbf{p}$ and $\mathbf{p}'$ yields 
\begin{align*}
1\rightarrow 21 \rightarrow 21*3\rightarrow 241*3 \rightarrow 5241*3
\rightarrow 56241*3 \rightarrow 562417*3.
\end{align*}
Note that we count the insertion positions from left for $\mathbf{p}$, 
and from right for $\mathbf{p'}$.
Then, if we ignore $*$, we have $5624173$.
\end{example}

Recall that one can construct a mini-word ($2\times n$ array 
of integers) from a labeled tree through its insertion history, 
and the first row is denoted by $\mathbf{a}$. 
We consider the standardization of $\mathbf{a}$, 
denoted by $\mathrm{stand}(\mathbf{a})$.
Here, standardization means that we replace $a_{i}$ by $j\in[1,s+1]$ 
if the entry $a_{i}$ is the $j$-th smallest elements in $\mathbf{a'}$.
For example, we have $\mathrm{stand}(\mathbf{a})=(1,3,4,2)$ 
when $\mathbf{a}=(1,4,5,3)$.

Given a standardized word $w$, we denote by $w^{-1}$ the 
inverse of $w$ as a permutation.
We have the following proposition.

Let $w'$ be a word with $*$'s. We define the inverse of $w'$ as a word 
$w'^{-1}:=(t_1,\ldots,t_n)$ such that 
$t_i$ is the position of $i$ in $w'$ from left.

The first row $\mathbf{a}$ of the mini-word and the word 
constructed from the insertion history $\mathbf{p}$ are 
related in the following manner.
\begin{prop}
\label{prop:standa}
Let $w':=w'(\mathbf{p})$ be a word defined as above.
Then, $\mathbf{a}=w'^{-1}$.
\end{prop}
\begin{proof}
We prove the statement by induction.
We first consider $n=1$ case.
Since $\mathbf{p}=\mathbf{p'}=\mathbf{a}=(1)$, 
we have $\mathrm{stand}(\mathbf{a})=w'(\mathbf{p})^{-1}=1$.

We assume that the statement is true for $n-1$.
Let $\mathbf{p}:=(p_1,\ldots,p_n)$ be the insertion history for a natural label
$L$ of size $n$ and $\mathbf{p}^{\vee}:=(p_1,\ldots,p_{n-1})$ the insertion 
history for the size $n-1$.
By induction assumption, $\mathbf{a}$ for $n-1$ is the inverse of $w'(\mathbf{p}^{\vee})$.
Recall that we have $a_{n}=p_{n}+1$ by the definition of a mini-word.
On the other hand, we insert $n$ at the $p_{n}+1$-th position in $w'(\mathbf{p}^{\vee})$.
Thus, for the integer $n$, $\mathbf{a}$ and $w'(\mathbf{p}^{\vee})$ are inverse with each 
other.
Combining these observations with the induction assumption, it is 
obvious that $\mathbf{a}$ is the inverse of $w'(\mathbf{p})$.
\end{proof}

\begin{example}
When $w'(\mathbf{p})=562417*3$, we have $\mathbf{a}=(5,3,8,4,1,2,6)$.
\end{example}

\begin{cor}
Let $\phi$ and $\psi$ be maps defined in Definitions \ref{defn:phi} and \ref{defn:psi}.
We denote by $w_{0}$ the pre-order word for a natural label $L$. 
Then, we have 
\begin{align}
\label{eqn:stanaphipsi}
\begin{aligned}
\mathrm{stand}(\mathbf{a})&=\phi(w_{0})^{-1}\\
&=\psi(w_{0})^{-1}.
\end{aligned}
\end{align}
\end{cor}
\begin{proof}
From Proposition \ref{prop:standa}, we have $\mathbf{a}=w'(\mathbf{p})$.
From Lemma \ref{lemma:w'psi} and Lemma \ref{lemma:phipsi}, 
we have $\mathrm{stand}(w'(\mathbf{p}))=\psi(w_{0})=\phi(w_0)$.
Combining these observations, we have Eqn. (\ref{eqn:stanaphipsi}).
\end{proof}

Given a symmetric Dyck tiling, we have the Hermite history consisting of paths 
which connect an up step in the top path $\mu$ and an up step in the lower path
$\lambda$.

Similarly, we have the Hermite history consisting of paths which connect 
a down step in the lower path $\lambda$ and a down step in the top path $\mu$.
We say that a Dyck tile $D$ is associated with a down step $d$ in $\lambda$ 
if the path which goes through in $D$ starts from $d$.

When a Dyck tiling consists of only single boxes and its lower path
is a zig-zag path $(UD)^{n}$, the above mentioned two Hermite histories are related as follows.
\begin{prop}
Let $D$ be a Dyck tiling consisting of only single boxes and its lower path is a zig-zag path.
Let $\mathrm{inv}(\mathbf{a}_{0}):=(r_1,\ldots,r_{n})$ be a sequence of non-negative integers 
for a standardized word $\mathbf{a}_{0}:=\mathrm{stand}(\mathbf{a})$ for the tiling $D$. 
The integer $r_{i}$ counts the number of tiles associated with the down step corresponding 
to the edge with the label $i$ in $\lambda$.

\end{prop}
\begin{proof}
Let $w$ be the pre-order word for the natural label $L$ corresponding to $\mathbf{a}$.
From Lemma \ref{lemma:invpsi}, we have $\mathrm{inv'}(\psi(w))=\mathrm{inv'}(w)$ which counts the number 
of Dyck tiles associated with the $i$-th step of the top boundary $\mu$.
From Lemma \ref{lemma:w'psi} and Proposition \ref{prop:standa}, we have 
$\mathrm{inv'}(w)=\mathrm{inv'}(\mathbf{a}^{-1})$.

Let $w':=\bar{w}^{-1}$ where $\bar{w}$ is defined in Proposition \ref{prop:Rev}.
To prove the propostion, from Proposition \ref{prop:Rev}, it is enough to show that 
$\mathrm{inv'}(w')$ counts the number of Dyck tiles associated with the 
$i$-th up step of the top boundary in the Dyck tiling which is obtained 
from $D$ by a mirror image.

Since all Dyck tiles are single boxes, the Dyck tiling $D$ is the same tiling
obtained from $w$ by the DTS bijection.
Further, if $w_{i}=n$, we have that $(w_1,\ldots,w_{i-1})$ is a permutation in 
$[1,i-1]$ and $(w_{i+1},\ldots,w_{n})$ is in $[i,n-1]$.
This is because if we have a pair of $p<q$ such that $w_{i_1}=q$, $w_{i_2}=p$
and $i_1<i<i_2$, we have a non-trivial Dyck tile by the DTS bijection.
Thus, there is no such pair of $(p,q)$.

In $w'$, we have $w'_{1}=n+1-i$, $(w'_{2},\ldots,w'_{n+1-i})$ is a permutation 
in $[1,n-i]$, and $(w'_{n+2-i},\ldots, w'_{n})$ is in $[n+2-i,n]$.
Note that if $w_{1}=n$, we have $w'_{1}=n$. If $w_i=n$, $i>1$, the Dyck tiling
constructed from $w'$ by the DTS bijection can be obtained by gluing two Dyck tilings, 
which are obtained from $(w'_1,\ldots,w'_{n+2-i})$ and $(w'_{n+1-i},\ldots,w'_{n})$ 
by the DTS bijection.
Then, by induction on $n$, it is easy to show that 
$\mathrm{inv'}(w')$ counts the number of Dyck tiles associated with 
the $i$-th up step of the top boundary of the Dyck tiling, which is a mirror image of $D$.
The Hermite history in a mirror image of $D$ gives an Hermite history whose paths are starting 
from down steps in the lower boundary path in $D$.
This completes the proof.
\end{proof}

\begin{example}
Let $\mathbf{a}^{-1}=4132$. We have $\mathrm{inv'}(\mathbf{a}^{-1})=(0,1,1,2)$.
The left picture in Figure \ref{fig:Hh2} is an example of a Dyck tiling 
with the top path $\mu=UUDUUDDD$ and the lower path $\lambda=(UD)^{4}$.
The Hermite history associated with the top path $\mu$ is $(0,1,1,2)$.

Similarly, we have $\mathbf{a}=2431$ and $\mathrm{inv}(\mathbf{a})=(1,2,1,0)$.
The right picture in Figure \ref{fig:Hh2} is the Hermite history associated 
with the lower path $\lambda$.
\begin{figure}[ht]
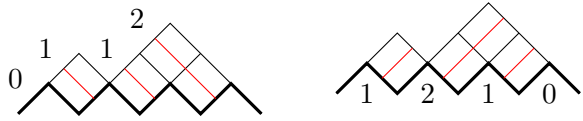

\tikzpic{-0.5}{[scale=0.4]
\draw(0,0)--(2,2)--(3,1)--(5,3)--(8,0);
\draw[very thick](0,0)--(1,1)--(2,0)--(3,1)--(4,0)--(5,1)--(6,0)--(7,1)--(8,0);
\draw(4,2)--(5,1)--(6,2);
\draw[red](1.5,1.5)--(2.5,0.5)(3.5,1.5)--(4.5,0.5)(4.5,2.5)--(6.5,0.5);
\draw(0.5,0.5)node[anchor=south east]{$0$};
\draw(1.5,1.5)node[anchor=south east]{$1$};
\draw(3.5,1.5)node[anchor=south east]{$1$};
\draw(4.5,2.5)node[anchor=south east]{$2$};
}
\quad
\tikzpic{-0.5}{[scale=0.4]
\draw(0,0)--(2,2)--(3,1)--(5,3)--(8,0);
\draw[very thick](0,0)--(1,1)--(2,0)--(3,1)--(4,0)--(5,1)--(6,0)--(7,1)--(8,0);
\draw(4,2)--(5,1)--(6,2);
\draw[red](1.5,0.5)--(2.5,1.5)(3.5,0.5)--(5.5,2.5)(5.5,0.5)--(6.5,1.5);
\draw(1,0)node{$1$};
\draw(3,0)node{$2$};
\draw(5,0)node{$1$};
\draw(7,0)node{$0$};
}
\caption{A Dyck tiling consisting of single boxes.}
\label{fig:Hh2}
\end{figure}

\end{example}

\subsection{Hermite histories for trees with circles}
\label{sec:Hwc}
Let $L$ be a natural label of the tree of size $n$ with circles.
We embed the natural label $L$ of size $n$ into a natural label $\widetilde{L}$ of a tree
without circles of larger size.
We note that the actions of the symmetric DTS bijections on $L$ and $\widetilde{L}$ give 
different symmetric Dyck tilings, however, the nesting numbers of the perfect matching 
of $\widetilde{L}$ give the number of tiles in the paths of the Hermite history of the 
symmetric Dyck tiling $L$.

In this subsection, we only consider a natural label $L$ of the tree for a symmetric Dyck tiling 
such that the labels from $1$ to $n-1$ are not circled and the label $n$
is circled.
By definition, the label $n$ is not on the edge with a dot.
We define a natural label $\widetilde{L}$ of size $n+1$ from $L$ by
adding an edge with a dot as follows.
\begin{enumerate}[(a)]
\item Change the label $n$ in $L$ to $n+1$.
\item We add the dotted edge with the label $n$ to $L$:
\begin{enumerate}[(i)]
\item if $L$ has at least one edges with dots, we append the new edge $E(n)$ with a dot just below 
the edge with a dot and with the largest label.
\item if $L$ has no edge with a dot, we append the new edge $E(n)$ with a dot 
at the root such that $E(i)\rightarrow E(n)$ for all $i\neq n$.
\end{enumerate}
\item We denote by $\widetilde{L}$ the new natural label of size $n+1$.
\end{enumerate}

Note that the newly added edge labeled by $n$ with a dot in $\widetilde{L}$ is connected 
to the right-most leaf in $\widetilde{L}$ in both cases.

We denote by $\mathbf{p}:=(p_1,\ldots,p_n)$ 
(resp. $\widetilde{\mathbf{p}}:=(\widetilde{p}_1,\ldots,\widetilde{p}_{n+1})$) the insertion history
for the natural label $L$ (resp. $\widetilde{L}$), and 
by $\mathbf{a}$ (resp. $\widetilde{\mathbf{a}}$) the first row of the mini-word
constructed from $\mathbf{p}$ (resp. $\widetilde{\mathbf{p}}$).
By construction, we have $p_{i}=\widetilde{p}_{i}$ for $1\le i\le n-1$ and 
$p_{n}=\widetilde{p}_{n+1}$ since the edge with a dot and labeled by $n$ is always 
right to the edge with the label $n+1$.
The edge with a dot and with the label $n$ is connected to the right-most leaf in $\widetilde{L}$.
Thus, in the language of $\widetilde{\mathbf{a}}:=(\widetilde{a}_1,\ldots,\widetilde{a}_{n+1})$, 
we have $\widetilde{a}_{i}<\widetilde{a}_{n}$ for $1\le i\le n-1$.
Since we have $E(n+1)\rightarrow E(n)$ in the natural label $\widetilde{L}$,
we have $\widetilde{a}_{n+1}<\widetilde{a}_{n}$.
From these observations, the integer $\widetilde{a}_{n}$ is the largest 
in $\widetilde{\mathbf{a}}$.
Thus, the nesting number of the $n$-th arch of the perfect matching for $\widetilde{L}$
is zero.
By construction, even if we remove $\widetilde{p}_{n}$ from $\widetilde{\mathbf{p}}$,
the relative positions of the arches except $n$-th in the perfect matching for 
$\widetilde{L}$ are the same as those of the arches in the perfect matching for 
$L$.
Thus, the nesting number of the $n+1$-th arch of the perfect matching for $\widetilde{L}$
is one plus the nesting number of the $n$-th arch of the perfect matching for $L$.
We put a circle on the $n$-th arch for $\widetilde{L}$. 
This circle corresponds to a ballot tile with $\ast$.

The above observations are summarized in the following proposition.
\begin{prop}
\label{prop:Ltilde}
Let $L$ and $\widetilde{L}$ be natural labels as above. 
The nesting numbers of the perfect matching for $\widetilde{L}$ give 
the Hermite history for $L$.
\end{prop}

\subsection{Proof of Theorem \ref{thrm:Hh}}
\label{sec:prHh}
We prove Theorem \ref{thrm:Hh} by induction.
We assume that the two descriptions of an Hermite history coincide with 
each other for the tree $\mathrm{Tree}(\lambda)$ where 
$\lambda$ is of size $n$.
We add one edge to the tree $\mathrm{Tree}(\lambda)$ and denote by $T'$ the tree of size $n+1$.
We have three cases:
1) the edge $E(n+1)$ is without a circle and without a dot, 
2) the edge $E(n+1)$ is with a circle and without a dot,
and 
3) the edge $E(n+1)$ has a dot.

\paragraph{\bf Case 1}

Let $\mathbf{p}$ (resp. $\mathbf{p'}$) be an insertion history for the label 
of $\mathrm{Tree}(\lambda)$ (resp. $T'$). 
Since $E(n+1)$ does not have a circle, when the size of $\mathbf{p}$ is $s$, 
then the size of $\mathbf{p'}$ is $s+1$.
We have $p'_{i}=p_{i}$ for $1\le i\le s$ by the definitions of $\mathbf{p}$ and 
$\mathbf{p'}$.
We denote by $\mathbf{a}:=(a_1,\ldots,a_{s})$ and $\mathbf{b}:=(b_1,\ldots,b_{s})$ 
the first and second rows of $2\times s$ array of integers, and similarly 
define $\mathbf{a'}$ and $\mathbf{b'}$ for $s+1$.
By definition, we have $a'_{s+1}=p'_{s+1}+1$ and $b'_{s+1}=2(s+1)$.

Let $E'$ be the edge strictly right to the edge $E(n+1)$ in $T'$.
If the label of $E'$ is smaller than $n+1$, we add a box by the symmetric DTS bijection.
This operation implies that we increase the number of ballot tiles associated with an up step 
of the top path $\mu$ by one when we modify the top path from $\mu$ to $\mu'$.
Further, the number of ballot tiles associated with the up step corresponding to the edge 
$E(n+1)$ is zero.

We have two cases for the tree $T'$: (a) $T'$ has no edge with a circle, and 
(b) $T'$ has edges with a circle.

\paragraph{\bf Case 1a}
We define $I$ be the set of labels as
\begin{align*}
I:=\{ i<n+1 |\ i \text{ is right to } n+1 \text{ in } w'(\mathbf{p'})\}.
\end{align*}
From Proposition \ref{prop:standa}, we have $\mathbf{a'}=w'(\mathbf{p'})^{-1}$.

\begin{claim}
\label{Claim:522}
Let $i\in I$. Then, the $i$-th arch in the $2\times (s+1)$ array
is inside the arch corresponding to the edge $E(n+1)$.
\end{claim}
\begin{proof}[Proof of Claim \ref{Claim:522}]
Since the sequence $\mathbf{a'}$ is given by the inverse of $w'(\mathbf{p'})$, 
we have $a_{i}>a_{s+1}$ if $i\in I$.
By construction, we have an arch connecting $a_{s+1}$ and $2(s+1)$ in 
the perfect matching corresponding to the mini-word.
The condition $a_{i}>a_{s+1}$ implies 
\begin{align*}
a_{s+1}<a_{i}<b_{i}<2(s+1).
\end{align*}
This means that the $i$-th arch is inside the arch for the edge $E(n+1)$.
\end{proof}

From Claim \ref{Claim:522}, it is clear that the nesting number of the 
$i$-th arch for $i\in I$ is increased by one by the insertion of $n+1$ 
in $w'(\mathbf{p'})$.
Combining this with Lemma \ref{lemma:invpsi}, the sequence of nesting 
numbers of the perfect matching gives the number of Dyck tiles associated 
with the $i$-th up step in the top boundary $\mu$.

\paragraph{\bf Case 1b}

We introduce a set of labels $I$ from $T'$ 
to state which arches in a perfect matching constructed 
from $\overline{L}$ are inside the arch corresponding to the edge $E(n+1)$ in 
$\overline{L}$.

Let $N$ be the number of circled labels strictly left to the edge $E(n+1)$ in 
the tree $T'$.
We denote by $i_1,i_2,\ldots,i_{M}$ the circled labels right to or above the edge 
$E(n+1)$ in $T'$. 
We perform an operation to remove a circle from the label $i_{j}, 1\le j\le M$,
such that we increase the labels larger than or equal to $i_{j}$ by one and 
add a dotted edge with the label $i_{j}$ by the addition of a dotted edge.
Here, we adopt the same algorithm to add a dotted edge as the steps (1) and (2)
as the operations just above Remark \ref{remark:addedge}
We denote by $\widetilde{T}$ the obtained tree by performing the above operations 
starting from $i_1$ and ending with $i_{M}$.
Let $I'$ be a set of labels strictly right to the edge $E(n+M+1)$ in $\widetilde{T}$.
We denote by $I$ the set of labels such that we add $N$ to an each element of $I'$.

Recall that we have a perfect matching obtained from the insertion history $\mathbf{p'}$
of size $s+1$.
We have the following claim.

\begin{claim}
\label{Claim:523}
Let $i\in I$. The $i$-th arch in the $2\times(s+1)$ array is inside the arch 
corresponding to the edge $E(n+1)$. 
\end{claim}
\begin{proof}[Proof of Claim \ref{Claim:523}]
First, we consider a natural label $L$ of size $n$ such that the label $n$ 
has a circle and other labels do not have circles.
Then, Claim holds from Proposition \ref{prop:Ltilde}.
So, we consider a natural label $L$ of size $n$ such that 
the label $n$ does not have a circle, and at least one of the other 
labels has a circle.
A map from $L$ to $\overline{L}$ yields a label without circles.
Further, by applying Proposition \ref{prop:Ltilde} recursively the subtree 
consisting of the edges in $[1,i]$ for $1\le i\le n$, 
it is clear that the $i$-th arch for $i\in I$ is inside the arch 
for the edge $E(n+1)$. 
\end{proof}
By a similar argument to Case 1a, the nesting numbers of the perfect matching
constructed from $\overline{L}$ gives the number of ballot tiles associated 
with an up step in $\mu$.

\paragraph{\bf Case 2}
Since we consider the case where the edge with the label $n$ is circled,
we divide the proof into two parts.
The first part is the case where $n$ is not circled.
This case is reduced to the Case 1b.
Then, the second part is to put a circle on the label $n$.
This situation is nothing but Proposition \ref{prop:Ltilde}, since 
the map from $L$ to $\overline{L}$ maps the subtree consisting of labels $[1,n-1]$
to a tree without circles.
Therefore, the statements for this case hold due to Case 1b and 
Proposition \ref{prop:Ltilde}.

\paragraph{\bf Case 3}
We first consider the Hermite history by the number of tiles in a symmetric
Dyck tiling. 
We perform a symmetric DTS bijection on 
the symmetric Dyck tiling for the label of $\mathrm{Tree}(\lambda)$.
Since the edge $E(n+1)$ has a dot, the sizes of ballot tiles with $\ast$
are increased by one, and other tiles are the same before and after the 
action of symmetric DTS bijection.
Let $\mathbf{H}:=(h_1,h_2,\ldots,h_{s})$ be a sequence of non-negative integers
such that $h_{i}$ is equal to the number of ballot tiles associated with the $i$-th up step
of the top $\mu$.
We similarly define $\mathbf{H'}:=(h'_1,h'_2,\ldots,h'_{s+1})$ for the label of $T'$.
From the above observations, we have 
\begin{align*}
&h'_{i}=h_{i}, \quad \text{ for } 1\le i\le s, \\
&h'_{s+1}=0.
\end{align*}

By induction assumption, the nesting numbers of the perfect matching is equal to 
$\mathbf{H}$ for the label of $\mathrm{Tree}(\lambda)$.
We will show that $\mathbf{H'}$ gives the nesting numbers of the perfect matching
for a label of $T'$.
Let $\mathbf{p}:=(p_1,p_2,\ldots,p_{s})$ and $\mathbf{p'}:=(p'_1,p'_2,\ldots,p'_{s+1})$ 
be insertion histories for the trees of larger size constructed from 
$\mathrm{Tree}(\lambda)$ and $T'$ as in Section \ref{sec:Hwc}.
It is obvious that we have $p'_{i}=p_{i}$ for $1\le i\le s$.
Recall that we construct from $\mathbf{p}$ a $2\times s$ array of integers whose 
first and second rows are given by $(a_1,\ldots,a_{s})$ and $(b_1,\ldots,b_{s})$.
Similarly, we have $(a'_1,\ldots,a'_{s+1})$ and $(b'_1,\ldots,b'_{s+1})$ from $\mathbf{p'}$.
Since there is no edge $E$ satisfying $E(n+1)\rightarrow E$, 
we have $a_{i}<p'_{s+1}+1$ for all $1\le i\le s$.
This means that the nesting numbers of $(a_{i},b_{i})$ is equal to the nesting number of 
$(a'_{i},b'_{i})$ for $1\le i\le s$, 
and the nesting number of $(a'_{s+1},b'_{s+1})$ is equal to zero.
Further, the pair $(a'_{s+1},b'_{s+1})$ is the rightmost pair, which implies that 
the positions of the circled nesting numbers are the same for $(a_i,b_i)$ and 
$(a'_{i},b'_{i})$ for $i\in[1,s]$. 
From the induction assumption and the above observations, it is easy to show
that $\mathbf{H'}$ gives the nesting number of the perfect matching for a label of $T'$.

These observations complete the proof of Theorem \ref{thrm:Hh}.

\subsection{Ballot tilings}
\label{sec:btHh}
In Section \ref{sec:symDTSbt}, we have the map from a natural label 
in $\mathcal{L}^{b}$ to a natural label in $\mathcal{L}^{\mathrm{b}\subset\mathrm{D}}$.
Note that the ballot tiling constructed from the natural label 
$L\in\mathcal{L}^{\mathrm{b}\subset\mathrm{D}}$ by the symmetric DTS bijection 
satisfies the conditions in Definition \ref{defn:bt}.
In Section \ref{sec:HhsymDT}, we have a correspondence between a natural label 
in $\mathcal{L}^{\mathrm{D}}$ and a perfect matching.
Since $\mathcal{L}^{\mathrm{b}\subset\mathrm{D}}\subset\mathcal{L}^{\mathrm{D}}$, 
we can have a natural correspondence between a natural label in $\mathcal{L}^{b}$
and a perfect matching. 

\subsection{The top path of an Hermite hisotry}
From Section \ref{sec:HhsymDT} to Section \ref{sec:btHh}, we consider 
an Hermite history associated with the top path of a symmetric Dyck tiling.
In this subsection, we consider an Hermite history associated with the 
lower boundary path of a symmetric Dyck tiling.
As observed in Section 6.4 in \cite{S19}, we have a Dyck tiling 
through its Hermite history if a decreasing label $L_{\mathrm{dec}}$ 
is given.
We generalize this correspondence to the case of symmetric Dyck tilings.

Let $\mathcal{D}$ be a symmetric Dyck tiling and $L$ be its natural label of 
the tree $T$ of size $n$ obtained by the symmetric Dyck tiling strip.
We define the decreasing label $L_{\mathrm{dec}}$ from $L$ by replacing 
the label $i$ on $L$ by $n+1-i$.
Since $L$ is increasing from the root to leaves, $L_{\mathrm{dec}}$ is 
obviously decreasing from the root to leaves. 
Further, a label may have a circle in $L_{\mathrm{dec}}$ 
if it is not on the edge with a dot.

We construct a fundamental symmetric Dyck tiling from $L_{\mathrm{dec}}$ 
by the following three steps.

\paragraph{\bf Step 1}
Let $\lambda$ be a ballot path for the tree $T$.
We attach a label of $L_{\mathrm{dec}}$ to a down step in $\lambda$ corresponding to 
an edge without a dot in $L_{\mathrm{dec}}$.
Similarly, we attach a label to an up step in $\lambda$ 
corresponding to an edge with a dot.
In Section \ref{sec:HhsymDT}, we construct a symmetric Dyck tiling from 
labels on the top path. However, in case of decreasing labels, Dyck tilings
are obtained from labels on the lower ballot path $\lambda$.
We define the entry and exit of a ballot tile in Section \ref{sec:HhsymDT}.
Here, we define the entry as the southwest edge of the leftmost box and 
the exit as the northeast edge of the rightmost box in a ballot tile.
By ignoring the circles in $L_{\mathrm{dec}}$, we have a fundamental symmetric 
Dyck tiling with the lower path $\lambda$ as follows.
Suppose that a down step $d$ in $\lambda$ has a label $l$.
Let $m(l)$ be the number of down steps right to $d$ such that their 
labels are smaller than $l$.
We attach Dyck tiles at the down step $d$ such that the sum of weights 
of them is equal to $m(l)$.

\paragraph{\bf Step 2}
In Step 1, we add Dyck tiles from a down step corresponding to the edge $E$.
Note that this down step is an entry of the Dyck tiles.
Since we connect the entries and the exits of Dyck tiles by a path, we have a 
rightmost exit of the path starting from the down step.
We say that this rightmost edge is the exit edge associated with $E$.
When an edge $E$ in $L_{\mathrm{dec}}$ has an circle, we attach a ribbon 
from the exit edge associated with $E$ to a box with $\ast$.
In the addition of ribbons, we start from the edge with the maximal label and end 
at the edge with the minimal label.
When $E$ has a label $l(E)$, we say that the added ribbon has a label $l(E)$.

\paragraph{\bf Step 3}
Given a ballot path $\lambda$, we say its peak as the partial path $UD$ or 
the right-most up step $U$ which has no down step $D$ at its right and corresponds to
the edge with a dot.
Suppose we have a ribbon with a label $l(E)$ above a peak $p$. 
The peak in $\lambda$ corresponds to an edge in $T$ connected to a leaf.
Note that we have a unique sequence of the edges from the leaf to the root.
Let $T'$ be a unique subtree in $T$ which has a single edge connected to the root and 
contains the sequence of edges.
Suppose that we have $m$ edges without dots and $m'$ edges with dots in 
the $T'$ such that their labels are smaller than $l(E)$. 
Since we add ribbons in Step 2, we have single boxes above the peak $p$.
We transform single boxes to a non-trivial ballot tile of length $(m,m')$ with 
$m+m'\ge1$ above the peak $p$ in the path $\lambda$.
If two different ballot tiles share a single box in the above construction,
we merge these ballot tiles into a ballot tile of larger size.

\begin{example}
\label{ex:decDT}
In Figure \ref{fig:decDT}, we consider a decreasing label and its fundamental 
symmetric Dyck tiling.  
The right-most picture is the corresponding to the decreasing label.
\begin{figure}[ht]
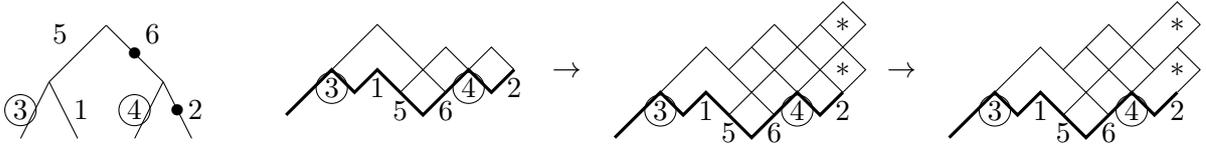

\tikzpic{-0.5}{[scale=0.5]
\coordinate
	child{coordinate(c2)
		child{coordinate(c4)}
		child{coordinate(c6)}
	}
	child[missing]
	child{coordinate(c1)
		child{coordinate(c3)}
		child{coordinate(c5)}
	};
\node[anchor=south east] at($(0,0)!0.5!(c2)$){$5$};
\node[anchor=east] at($(c2)!0.5!(c4)$){$\circnum{3}$};
\node[anchor=west] at($(c2)!0.5!(c6)$){$1$};
\node[anchor=south west] at($(0,0)!0.5!(c1)$){$6$};
\node[anchor=east] at($(c1)!0.5!(c3)$){$\circnum{4}$};
\node[anchor=west] at($(c1)!0.5!(c5)$){$2$};
\node at($(0,0)!0.5!(c1)$){$\bullet$};
\node at($(c1)!0.5!(c5)$){$\bullet$};
}
\quad 
\tikzpic{-0.5}{[scale=0.3]
\draw[very thick](0,0)--(2,2)--(3,1)--(4,2)--(6,0)--(8,2)--(9,1)--(10,2);
\draw(2,2)--(4,4)--(7,1)(5,1)--(7,3)--(8,2)(8,2)--(9,3)--(10,2);
\node at(2,1.2){$\circnum{3}$};
\node at(4,1.2){$1$};
\node at(5,0.2){$5$};
\node at(7,0.2){$6$};
\node at(8,1.2){$\circnum{4}$};
\node at(10,1.2){$2$};
}
$\rightarrow$
\tikzpic{-0.5}{[scale=0.3]
\draw[very thick](0,0)--(2,2)--(3,1)--(4,2)--(6,0)--(8,2)--(9,1)--(10,2);
\draw(2,2)--(4,4)--(7,1)(5,1)--(7,3)--(8,2)(8,2)--(9,3)--(10,2);
\draw(5,3)--(7,5)--(9,3)(6,4)--(7,3)--(10,6)--(11,5)--(9,3);
\draw(9,5)--(11,3)--(10,2);
\node at(2,1.2){$\circnum{3}$};
\node at(4,1.2){$1$};
\node at(5,0.2){$5$};
\node at(7,0.2){$6$};
\node at(8,1.2){$\circnum{4}$};
\node at(10,1.2){$2$};
\node at(10,3){$\ast$};
\node at(10,5){$\ast$};
}
$\rightarrow$
\tikzpic{-0.5}{[scale=0.3]
\draw[very thick](0,0)--(2,2)--(3,1)--(4,2)--(6,0)--(8,2)--(9,1)--(10,2);
\draw(2,2)--(4,4)--(7,1)(5,1)--(7,3)--(8,2)(8,2)--(9,3);
\draw(5,3)--(7,5)--(9,3)(6,4)--(7,3)--(10,6)--(11,5)--(9,3);
\draw(10,4)--(11,3)--(10,2);
\node at(2,1.2){$\circnum{3}$};
\node at(4,1.2){$1$};
\node at(5,0.2){$5$};
\node at(7,0.2){$6$};
\node at(8,1.2){$\circnum{4}$};
\node at(10,1.2){$2$};
\node at(10,3){$\ast$};
\node at(10,5){$\ast$};
}
\caption{A decreasing label and Dyck tilings after Step 1, 2, and 3 from left to right.}
\label{fig:decDT}
\end{figure}
\end{example}

\begin{prop}
A map from decreasing labels with circles to a fundamental symmetric 
Dyck tilings gives a bijection between them.
\end{prop}
\begin{proof}
We have an obvious bijection between natural labels and decreasing labels, that is,
$i\mapsto n+1-i$, $i\in[1,n]$ where $n$ is the number of edges.
An edge without a dot may have a circle in both natural and decreasing labels.
This implies the total number of labels are the same.
In case of natural labels, they are bijective to symmetric Dyck tilings by 
the symmetric DTS defined in Section \ref{sec:symDTS}.
Thus, it is enough the map $\nu$ from decreasing labels to symmetric Dyck 
tilings is injective.
Let $L_{1}$ and $L_{2}$ be two distinct decreasing labels.
By the processes from Step 1 to Step 3, it is obvious that 
$\nu(L_{1})$ and $\nu(L_{2})$ are symmetric Dyck tilings and 
they are distinct.
This completes the proof.
\end{proof}

Let $n$ be the number of edges without a dot in $T$ and $n'$ be the number of 
dotted edges in $T$. We denote $N=2n+n'$.
Let $(k_1, k_2, \ldots, k_{N})$ be a bi-word consisting 
of $U$ and $D$ of length $N$ and of an integer sequence in $[1,n+n']$.
Thus, $k_{i}=\begin{bmatrix}X \\ p\end{bmatrix}$ with $X=U$ or $D$ and $p\in[1,n+n']$.
Then, a ballot bi-word $\mathbf{K}_{N}:=(k_1,k_2,\ldots,k_{N})$ is defined as follows.
The first word is the word in the first row of $\mathbf{K}_{N}$ consisting of a ballot 
word of $U$'s and $D$'s. 
The second word is the word in the second row of $\mathbf{K}_{N}$ consisting of positive
integers in $[1,n+n']$.

We define 
\begin{align*}
\mathbf{K'}_{1}
:=\begin{cases}
\begin{bmatrix}
U & D \\
1 & 1
\end{bmatrix}, & h_{1} \text{ does not have a box}, \\[12pt]
\begin{bmatrix}
U \\
1 
\end{bmatrix}, & h_{1} \text{ has a box}. \\[12pt] 
\end{cases}
\end{align*}
We recursively construct $\mathbf{K'}_{m}$ with $1\le m\le N$ 
by using the insertion history $\mathbf{h}$ associated with the natural
label following \cite{S19} with small modifications. 
Then, we will construct a ballot bi-word $\mathbf{K}_{N}$ from $\mathbf{K'}_{N}$.
We have the following algorithms:
\begin{enumerate}
\item We define $\mathbf{K}_{1}:=\mathbf{K'}_{1}$.
\item Increase all integers in the second word of $\mathbf{K'}_{m-1}$ by $1$
and denote it by $\mathbf{K}''_{m-1}$. 
\item Find the $h_{m}$-th position in $\mathbf{K''}_{m-1}$ and insert 
$\mathbf{K'}_{1}$ there. 
Note that $\mathbf{K'}_{1}$ depends on whether $h_{m}$ has a box or not.
\item If $\begin{matrix}D \\ a\end{matrix}$ with $a\ge2$ is left to $\begin{matrix}U \\ 1\end{matrix}$,
we move $\begin{matrix}U \\ 1\end{matrix}$ left to the leftmost $\begin{matrix}D \\ b\end{matrix}$ with 
$b\ge2$.
As a sequence, we have 
\begin{align*}
\cdots
\begin{matrix}
D  \\
b 
\end{matrix}
\cdots
\begin{matrix}
D  \\
a 
\end{matrix}
\cdots
\begin{matrix}
U \\
1 
\end{matrix}\cdots
\longrightarrow
\cdots
\begin{matrix}
U & D \\
1 & b 
\end{matrix}
\cdots
\begin{matrix}
D\\
a
\end{matrix}
\cdots,
\end{align*}
where the dotted parts are unchanged.
If such $D$ does not exist, we do not change the words.
We denote by $\mathbf{K'}_{m}$ the bi-words obtained from $\mathbf{K''}_{m}$ 
by the above procedure.
\item 
If a label $N+1-l$, $l\in[1,N]$, in $\mathbf{L}_{\mathrm{dec}}$ has a circle,
we change 
$\begin{matrix}D \\ l\end{matrix}$ to $\begin{matrix}U \\ l\end{matrix}$ 
in $\mathbf{K'}_{N}$.
We denote by $\mathbf{K}_{N}$ the new ballot bi-word.
\end{enumerate}

\begin{remark}
In the step (5), we replace $D$'s by $U$'s for some cases.
When $h_{N+1-l}$ in $\mathbf{h}$ has a circle, this means that the value $h_{N+1-l}$ 
does not have a box. Therefore, we have 
$\begin{matrix}D \\ l\end{matrix}$ somewhere in $\mathbf{K'}_{N}$ 
since we insert $\mathbf{K'}_{1}$ at somewhere in $\mathbf{K'}_{m}$, $m\le N$, 
at the step (3).
\end{remark}

\begin{prop}
Let $L_{\mathrm{dec}}$ be a decreasing label, $\mathbf{h}$ be its insertion history,
and $\mathbf{K}_{n}$ be a ballot bi-word associated with $\mathbf{h}$.
Then, the top path of the Hermite history of $\mathbf{L}_{dec}$ is the 
first ballot word of $\mathbf{K}_{n}$.
\end{prop}

\begin{proof}
Recall that we have three steps from Step 1 to Step 3 to obtain 
a symmetric Dyck tiling from a decreasing label.
Since Step 3 does not change the top path of the Dyck tiling,
it is enough to show that the first word of $\mathbf{K}_{N}$ 
is the top path of the Dyck tiling obtained by up to Step 2.

By the same argument as Proposition 6.28 in \cite{S19}, 
the top path of a symmetric Dyck tiling after Step 1 can be 
obtained by the processes from (1) to (4).
The addition of a ballot ribbon to a symmetric Dyck tiling
is nothing but the operation to change a down step to a up path.
This is realized by changing a down step corresponding to an edge 
with a circle in the natural label to an up step.
Therefore, we obtain the top path of the symmetric Dyck tiling
after the process (5).
\end{proof}

\begin{example}
We consider the same example as \ref{ex:decDT}.
The ballot bi-word $\mathbf{K}_{6}$ is given by
\begin{align*}
&\left[\begin{matrix}
U \\
1
\end{matrix}\right]
\rightarrow
\left[\begin{matrix}
U&D&U \\
1&1&2
\end{matrix}\right]
\rightarrow
\left[\begin{matrix}
U&D&U&U&D \\
2&2&3&1&1
\end{matrix}\right]
\rightsquigarrow
\left[\begin{matrix}
U&U&D&U&D \\
2&1&2&3&1
\end{matrix}\right]\\
&\rightarrow
\left[\begin{matrix}
U&U&D&U&D&U&D \\
3&1&1&2&3&4&2
\end{matrix}\right] \\
&\rightarrow
\left[\begin{matrix}
U&U&D&U&D&U&D&U \\
4&2&2&3&4&5&3&1
\end{matrix}\right]
\rightsquigarrow
\left[\begin{matrix}
U&U&U&D&U&D&U&D \\
4&2&1&2&3&4&5&3
\end{matrix}\right]\\
&\rightarrow
\left[\begin{matrix}
U&U&U&U&D&D&U&D&U&D \\
5&3&2&1&1&3&4&5&6&4
\end{matrix}\right]\\
&\Rightarrow
\left[\begin{matrix}
U&U&U&U&D&U&U&D&U&U \\
5&3&2&1&1&3&4&5&6&4
\end{matrix}\right],
\end{align*}
where $\rightsquigarrow$ indicates the step (4) and $\Rightarrow$ 
indicates the step (5).
The top path of the symmetric Dyck tiling in Example \ref{ex:decDT}
is $UUUUDUUDUU$.
\end{example}

\section{Symmetric Dyck tiling ribbon}
\subsection{Symmetric Dyck tilings}
In Section \ref{sec:symDTS}, we define the symDTS bijection.
In this subsection, we introduce another bijection that we call 
{\it symmetric Dyck tiling ribbon} (symDTR for short).
Similar to the symDTS bijection, the symDTR bijection consists 
of two operations: spread of a symmetric Dyck tiling, and 
addition of a ribbon.
Spread of a symmetric tiling for the symDTR bijection is the same as 
the one for the symDTS bijection.

Below, we define the addition of a ribbon on a symmetric Dyck tiling.
Ribbons are classified into two types: Dyck ribbons and ballot ribbons.

As in the case of (non-symmetric) Dyck tilings, we introduce the notion 
of {\it special column} of a Dyck tiling (see \cite{KMPW12}). 
Given a fundamental symmetric Dyck tiling $D$ in the skew shape $\lambda/\mu$,
a column $t$ is {\it eligible} if: 
\begin{enumerate}
\item The top path $\mu$ contains an up step that ends $x=t$.
\item The top path $\mu$ at $x=t$ is the top corner of a ballot tile of length 
$(n,n')$ with $(n,n')\neq(0,0)$, or coincides with the lower boundary path $\lambda$ at $x=t$.
\end{enumerate}
Note that there is at least one eligible column of a fundamental symmetric Dyck 
tiling since the leftmost step in $\mu$ is the up step and it is eligible.
The special column is defined to be the rightmost eligible column.

After the spread of a symmetric Dyck tiling at $x=s$, we perform the 
addition of a Dyck or a ballot ribbon on the tiling.

\paragraph{\bf Addition of a Dyck ribbon}
We add a ribbon of single boxes from $x=s$ to the special column $t$ 
which is right to the column $s$.
In other words, a partial path of $\mu$ from $x=s$ to $x=t$ is starting from 
a down step $d$ and ending with a up step $u$. 
The addition of a Dyck ribbon is to exchange the down step $d$ to an up step 
and the up step $u$ to a down step, and fill the added region with single boxes.

\paragraph{\bf Addition of a ballot ribbon}
We add a ribbon of single boxes from $x=s$ to the rightmost column such that 
the rightmost box of the ribbon has $\ast$. 
In other words, a partial path of $\mu$ right to $x=s$ is starting from 
an down step $d$. We change this down step $d$ to an up step, and fill the 
added region with single boxes.

\begin{remark}
A ribbon is said to be a ``Dyck" ribbon when the ribbon does not contain 
a box with $\ast$.
On the other hand, a ribbon is said to be a ``ballot" ribbon when it 
contains a box with $\ast$ as its rightmost box. 
\end{remark}

We define the symDTR bijection from natural labels with circles to symmetric 
Dyck tilings as follows. 
We follow the notations in Section \ref{sec:symDTS}.
Let $T$ be a labeled tree of size $n$ and $T'$ be the partial tree of size $n-1$.
The tree $T'$ is obtained from $T$ by deleting the edge with the label $n$.
Suppose we have a correspondence between the labeled tree $T'$ of size $n-1$ 
and a symmetric Dyck tiling $D$ of size $n-1$.
We construct a symmetric Dyck tiling of size $n$ from $D$ and $T'$ as follows.
The first procedure is the same as (symDTS\ref{symDTS1}).
We add a ribbon after the operation (symDTS1): 
\begin{list}{}{}
\item[(symDTR2)] 
Suppose that the special column $t$ is right to the column $s$.
We perform the addition of a Dyck ribbon from $x=s$ to $x=t$ 
if the label $n$ is not circled.
\item[(symDTR3)]
We perform the addition of a ballot ribbon from $x=s$
if the label $n$ is circled.
\end{list}

\begin{remark}
By construction, the number of $\ast$ in a symmetric Dyck tiling 
is equal to the number of circled labels in the natural label.
\end{remark}

\begin{example}
\label{ex:symDTR}
We consider the following natural label:
\begin{center}
\tikzpic{-0.5}{[scale=0.6]
\coordinate
	child{coordinate(c2)
		child{coordinate(c4)}
		child{coordinate(c6)}
	}
	child[missing]
	child{coordinate(c1)
		child{coordinate(c3)}
		child{coordinate(c5)}
	};
\node[anchor=south east] at($(0,0)!0.5!(c2)$){$2$};
\node[anchor=east] at($(c2)!0.5!(c4)$){$\circnum{4}$};
\node[anchor=west] at($(c2)!0.5!(c6)$){$6$};
\node[anchor=south west] at($(0,0)!0.5!(c1)$){$1$};
\node[anchor=east] at($(c1)!0.5!(c3)$){$\circnum{3}$};
\node[anchor=west] at($(c1)!0.5!(c5)$){$5$};
\node at($(0,0)!0.5!(c1)$){$\bullet$};
\node at($(c1)!0.5!(c5)$){$\bullet$};
}
\end{center}
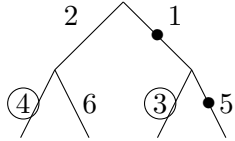
The growth of a fundamental symmetric Dyck tiling is shown in 
Figure \ref{fig:symDTR}
\begin{figure}[ht]
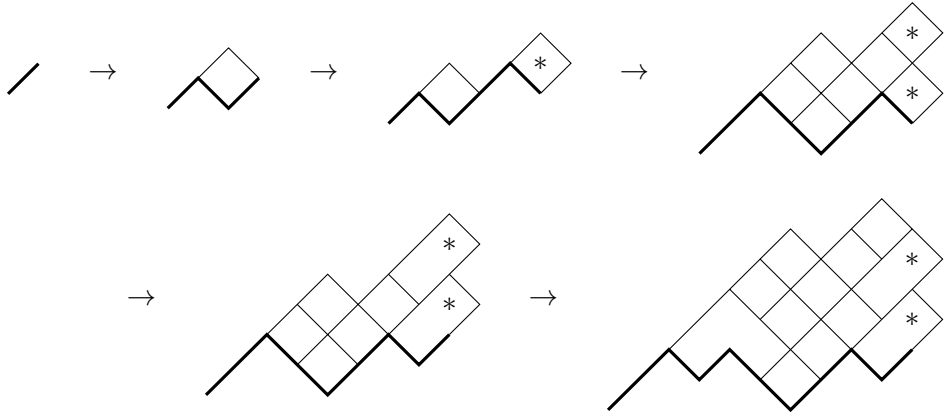

\begin{align*}
\tikzpic{-0.5}{[x=0.4cm,y=0.4cm]
\draw[very thick](0,0)--(1,1);
}
\quad\rightarrow\quad
\tikzpic{-0.5}{[x=0.4cm,y=0.4cm]
\draw[very thick](0,0)--(1,1)--(2,0)--(3,1);
\draw(1,1)--(2,2)--(3,1);
}
\quad\rightarrow\quad
\tikzpic{-0.5}{[x=0.4cm,y=0.4cm]
\draw[very thick](0,0)--(1,1)--(2,0)--(4,2)--(5,1);
\draw(1,1)--(2,2)--(3,1)(4,2)--(5,3)--(6,2)--(5,1);
\node at(5,2){$\ast$};
}
\quad\rightarrow\quad
\tikzpic{-0.5}{[x=0.4cm,y=0.4cm]
\draw[very thick](0,0)--(2,2)--(4,0)--(6,2)--(7,1);
\draw(2,2)--(4,4)--(6,2)(3,3)--(5,1)(3,1)--(7,5)--(8,4)--(6,2);
\draw(6,4)--(8,2)--(7,1);
\node at(7,2){$\ast$};\node at(7,4){$\ast$};
}\\[12pt]
\rightarrow\quad
\tikzpic{-0.5}{[x=0.4cm,y=0.4cm]
\draw[very thick](0,0)--(2,2)--(4,0)--(6,2)--(7,1)--(8,2);
\draw(2,2)--(4,4)--(6,2)(3,3)--(5,1)(3,1)--(8,6)--(9,5)--(6,2);
\draw(6,4)--(7,3)(8,4)--(9,3)--(8,2);
\node at(8,3){$\ast$};\node at(8,5){$\ast$};
}\quad
\rightarrow\quad
\tikzpic{-0.5}{[x=0.4cm,y=0.4cm]
\draw[very thick](0,0)--(2,2)--(3,1)--(4,2)--(6,0)--(8,2)--(9,1)--(10,2);
\draw(2,2)--(6,6)--(9,3)(4,4)--(7,1)(5,5)--(8,2);
\draw(5,3)--(9,7)--(11,5)--(8,2)(6,2)--(10,6);
\draw(8,6)--(9,5)(10,4)--(11,3)--(10,2)(5,1)--(6,2);
\node at(10,3){$\ast$};\node at(10,5){$\ast$};
}
\end{align*}
\caption{A growth of a fundamental symmetric Dyck tiling by the symDTR bijection.} 
\label{fig:symDTR}
\end{figure}
\end{example}

\begin{theorem}
Let $L$ be a natural label of the tree $\mathrm{Tree}(\lambda)$. 
The map symmetric DTR gives a bijection between symmetric Dyck tilings 
above $\lambda$ and natural labels $L$.
\end{theorem}
\begin{proof}
The proof is similar to the proof of Thereom \ref{thrm:bijfsDTS}.
By reversing the procedures (symDTS1), (symDTR2) and (symDTR3), 
we can have the following three information: 
a natural tree $L'$ for a fundamental symmetric Dyck 
tiling $\mathcal{D}'$, special column $x=s$, and whether the label $n$
is circled or not.
The rest of the proof remains the same as Theorem \ref{thrm:bijfsDTS}.
\end{proof}

\subsection{Ballot tilings}
Recall the map $\iota$ introduced in Section \ref{sec:symDTSbt}.
By the correspondence between natural labels in $\mathcal{L}^{\mathrm{b}}$ to 
natural labels in $\mathcal{L}^{\mathrm{b}\subset\mathrm{D}}$, and the symDTR 
bijection, we define the ribbon growth of a ballot tiling as the symDTR bijection
on a natural label in $\mathcal{L}^{\mathrm{b}\subset\mathrm{D}}$.

\begin{theorem}
A symDTR bijection on $\mathcal{L}^{\mathrm{b}\subset\mathrm{D}}(\lambda)$ gives 
a ballot tiling above $\lambda$, which satisfies the conditions in Definition \ref{defn:bt}.
\end{theorem}
\begin{proof}
A natural label in $\mathcal{L}^{\mathrm{b}\subset\mathrm{D}}(\lambda)$ 
has even number of circled labels.
The circled label corresponds to an addition of a ballot ribbon. 
This ribbon ends at the 
right-most column.
By symmetric DTR, we put a labeled box at the right-most column when we have a dotted
edge. 
By the same argument as the proof of Theorem \ref{thrm:symDTSforBT}, 
we have only ballot tiles of size $2n+n'$ with $n'$ odd and the number of such 
tiles is always even.
These observations imply that we have a ballot tiling after the symmetric 
DTR bijection, which satisfies the conditions in Definition \ref{defn:bt}.
\end{proof}

\section{Ballot tableaux}
\label{sec:BTab}
In this section, we introduce the notion of ballot tableaux.
Ballot tableaux are a generalization of Dyck tableaux, which are 
bijective to Dyck tilings obtained from the DTR bijection.
See \cite{ABDH11} for Dyck tableaux with a zig-zag Dyck path and  
\cite{S19} for Dyck tableaux with a general lower boundary path.

\subsection{Ballot tableaux for symmetric Dyck tilings}
\label{sec:BTsDT}
Let $\lambda$ be a ballot path of length $(n,n')$ and $T(\lambda)$ be the tree for 
symmetric Dyck tilings associated with $\lambda$.
An edge without a dot in $T(\lambda)$ consists of a pair of an up step $U_1$
and a down step $D_1$ in $\lambda$. 
There exists a unique box below $\lambda$ that is in the south-east direction
from $U_1$ and in the south-west direction from $D_1$.
A dotted edge of $T(\lambda)$ consists of an up step $U_{2}$ in $\lambda$.
There exists a unique box below $\lambda$ that is in the south-east direction 
from $U_{2}$ and is on the line $x=2n+n'$.
We call these unique boxes associated with edges in $T(\lambda)$ {\it anchor} boxes.
Especially, the boxes at $x=2n+n'$ are called {\it boundary anchor} boxes.

Let $\lambda_{i}$ for $1\le i\le m-1$ be Dyck paths and $\lambda_{m}$ be a 
ballot path such that $\lambda_{i}$ cannot be written as a concatenation of 
ballot paths. 
We denote a ballot path $\lambda$ by $\lambda:=\lambda_1\circ\cdots\circ\lambda_{m}$.
The path $\lambda_{i}$ contains a unique edge connected to the root.

We define a path $\underline{\lambda}$ by
\begin{align*}
\underline{\lambda}
:=
\vee_{|\lambda_{1}|}\circ\cdots\circ\vee_{|\lambda_{m-1}|}\circ \diagdown_{|\lambda_{m}|}.
\end{align*}
where $\vee_{n}$ and $\diagdown_{n}$ are the following paths:
\begin{align*}
\vee_{n}&:=D^{n}U^{n}, \\
\diagdown_{n}&:=D^{n}.
\end{align*}
We call the region surrounded by $\lambda$, $\underline{\lambda}$ and the line $x=2n+n'$
the {\it frozen region} associated with $\lambda$.
We have $n+n'$ anchor boxes in the frozen region of $\lambda$.
We say that these anchor boxes in the {\it zeroth floor}.
When we translate an anchor box in the zeroth floor upward by $(0,2m)$,
we call the new box the anchor box in the {\it $m$-th floor}.

Let $a$ be an anchor box in the zeroth floor, that is associated with an edge without a dot.
This anchor box corresponds to a pair of an up step $u$ and a down step $d$.
We consider the partial path $\nu$ starting from $u$ and ending with $d$.
Then, we have a partial frozen region $R_{f}$ surrounded by $\nu$ and $\underline{\nu}$.
This partial frozen region $R_{f}$ is said to be associated with the anchor box $a$.
Similarly, let $b$ be a boundary anchor box in the zeroth floor, that is associated with a dotted 
edge in $T(\lambda)$.
This anchor box $b$ corresponds to an up step $u'$.
We consider the partial path $\nu'$ starting from this $u'$.
Then, we have a partial frozen region $R_{f'}$ surrounded by $\nu$, $\underline{\nu}$ and 
the line $x=2n+n'$.
This partial frozen region $R_{f'}$ is said to be associated with the boundary anchor box $b$.

We introduce five classes of boxes which appear in ballot tableaux.
The first four types of boxes are the same as the classification of boxes for generalized 
Dyck tableaux introduced in \cite{S19}.
\begin{enumerate}
\item An empty box.
\item A box with a lable $i\in[1,n+n']$.
\item A parallel box. A line passes through from its north-west edge to its south-east edge or
from its south-west edge to its north-east edge.
\item A turn box. A $\vee$-turn (resp. $\wedge$-turn) box is a box with a line passing through from the
north-west (resp. south-west) edge to the north-east (resp. south-east) edge.
\item A terminal box. A line comes from the north-west or south-west edge and ends at the center of the box.
\end{enumerate}
The terminal boxes are newly introduced to define a ballot tableau.
Five classes of boxes are depicted in Figure \ref{fig:fiveboxes}.
\begin{figure}[ht]
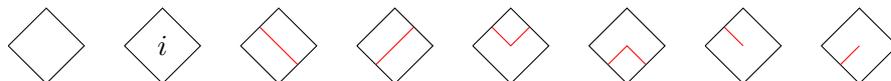

\tikzpic{-0.5}{[x=0.5cm,y=0.5cm]
\draw(0,0)--(1,1)--(2,0)--(1,-1)--(0,0);
}
\tikzpic{-0.5}{[x=0.5cm,y=0.5cm]
\draw(0,0)--(1,1)--(2,0)--(1,-1)--(0,0);
\node at (1,0){$i$};
}
\tikzpic{-0.5}{[x=0.5cm,y=0.5cm]
\draw(0,0)--(1,1)--(2,0)--(1,-1)--(0,0);
\draw[red](0.5,0.5)--(1.5,-0.5);
}
\tikzpic{-0.5}{[x=0.5cm,y=0.5cm]
\draw(0,0)--(1,1)--(2,0)--(1,-1)--(0,0);
\draw[red](0.5,-0.5)--(1.5,0.5);
}
\tikzpic{-0.5}{[x=0.5cm,y=0.5cm]
\draw(0,0)--(1,1)--(2,0)--(1,-1)--(0,0);
\draw[red](0.5,0.5)--(1,0)--(1.5,0.5);
}
\tikzpic{-0.5}{[x=0.5cm,y=0.5cm]
\draw(0,0)--(1,1)--(2,0)--(1,-1)--(0,0);
\draw[red](0.5,-0.5)--(1,0)--(1.5,-0.5);
}
\tikzpic{-0.5}{[x=0.5cm,y=0.5cm]
\draw(0,0)--(1,1)--(2,0)--(1,-1)--(0,0);
\draw[red](0.5,0.5)--(1,0);
}
\tikzpic{-0.5}{[x=0.5cm,y=0.5cm]
\draw(0,0)--(1,1)--(2,0)--(1,-1)--(0,0);
\draw[red](0.5,-0.5)--(1,0);
}
\caption{Five classes of boxes. An empty box (the first picture), 
a box with the label $i$ (the second picture), parallel boxes (the third and the fourth picture), 
turn boxes (the fifth and the sixth picture), and terminal boxes (the seventh and eighth picture).}
\label{fig:fiveboxes}
\end{figure}

Let $\mu$ be a ballot path such that $\mu\ge\lambda$, and 
$R_{0}$ be the region surrounded by $\underline{\lambda}$, $\mu$ and 
the line $x=2n+n'$.
We put integers $[1,n+n']$ in the region $R_{0}$, and obtain a ballot 
tableau for a symmetric Dyck tiling.
The algorithm to produce a ballot tableau is as follows.
Let $L$ be a natural label with circles for a fundamental symmetric Dyck tiling.
We construct a ballot tableau by the following algorithm.
\begin{enumerate}
\item Set $i=1$.
\item 
\label{enu:BT2}
Find an anchor box $B$ in the zeroth floor corresponding to the edge of $L$
with the label $i$.
\item If the anchor boxes up to the $p-1$-th floor are occupied by boxes with a line,
or if the anchor box just below $B$ in the $p$-th floor, we put the label $i$ on 
the anchor box (that corresponds to the edge with the label $i$) in the $p$-th floor. 
\item If the edge with the label $i-1$ is strictly right to the edge with the label $i$
in $L$ and the label $i$ is not circled in $L$, 
then we connect the anchor boxes with the labels $i-1$ and 
$i$ by a line in the following way:
\begin{enumerate}
\item 
\label{cond:Dtab1}
The line starts from the north-east edge of the anchor box with the label $i$ and ends with 
the north-west edge of the anchor box with the label $i-1$.
The line consists of parallel boxes and turn boxes.
\item
\label{cond:Dtab2}
The line is above the anchor boxes with labels smaller than $i$.
\item When an anchor box in the $q$-th floor is labeled with the integer $1\le k\le i-1$,
we translated the partial frozen region associated with this anchor box (at the zeroth floor) 
upward by $(0,2q)$. Then, we redefine the frozen region as a union of the translated frozen region
and the original frozen region.
\item 
\label{cond:Dtab4}
The line can pass through a box (that is not an anchor box) 
in the frozen region as a parallel box, that is, from the south-west edge to 
the north-east edge, or from the north-west edge to the south-east edge.
The line passes through an anchor box at the $q+1$-th floor as a turn box or a parallel box. 
\item The line is the lowest path satisfying from (4-a) to (4-d).
\end{enumerate}
\item If the label $i$ is circled in $L$, then:
\begin{enumerate}
\item The line starts from the north-east edge of the anchor box with the label $i$
and ends with the center of a boundary anchor box, that is, a terminal box.
\item The line is the lowest path satisfying the conditions (5-a) and from (4-b) to (4-e).
\end{enumerate}
\item Increase $i$ by one. If $i\in[1,n+n']$, then go to (\ref{enu:BT2}).
Otherwise, the algorithm stops.
\end{enumerate}

\begin{example}
\label{ex:symDTab1}
We consider the same natural label as Example \ref{ex:symDTR}.
The ballot tableau for the fundamental symmetric Dyck tiling is depicted 
as below.	
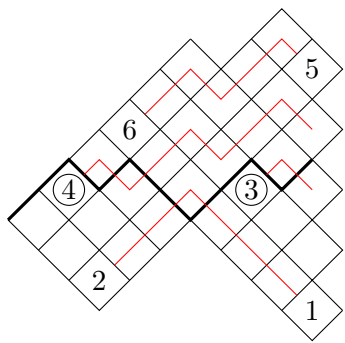
\begin{center}
\tikzpic{-0.5}{[x=0.4cm,y=0.4cm]
\draw[very thick](0,0)--(2,2)--(3,1)--(4,2)--(6,0)--(8,2)--(9,1)--(10,2);
\draw(0,0)--(3,-3)--(6,0)--(10,-4)--(11,-3)(2,2)--(6,6)(4,2)--(9,7)--(11,5)--(8,2);
\draw(1,1)--(4,-2)(3,1)--(5,-1)(3,3)--(4,2)(4,4)--(7,1)(5,5)--(8,2)(6,6)--(11,1);
\draw(1,-1)--(3,1)(2,-2)--(10,6)(8,6)--(11,3)--(10,2);
\draw(7,-1)--(9,1)(8,-2)--(11,1)(9,-3)--(11,-1)--(9,1)(11,-3)--(7,1);
\draw(10,-3)node{$1$}(3,-2)node{$2$}(8,1)node{$\circnum{3}$}(2,1)node{$\circnum{4}$};
\draw(10,5)node{$5$}(4,3)node{$6$};
\draw[red](3.5,-1.5)--(6,1)--(9.5,-2.5)(8.5,1.5)--(9,2)--(10,1)
	(2.5,1.5)--(3,2)--(4,1)--(6,3)--(7,2)--(9,4)--(10,3)
	(4.5,3.5)--(6,5)--(7,4)--(9,6)--(9.5,5.5);
}
\end{center}
\end{example}

\begin{remark}
\label{remark:terminalbox}
A fundamental symmetric Dyck tiling is equivalent to a symmetric 
Dyck tiling. 
A ballot tableau $B$ is also equivalent to a symmetric Dyck tableau $D$, which 
is obtained from $B$ by gluing two ballot tableaux $B$ and $\widetilde{B}$ where 
$\widetilde{B}$ is a mirror image of $B$.
Then, a terminal box in $B$ corresponds to a turn box in $D$, which 
means that we have no terminal boxes in $D$.
\end{remark}

\subsection{Insertion procedure for ballot tableaux}
The insertion procedure for a ballot tableau 
is the process to insert a labeled box into a fundamental 
symmetric Dyck tableau.
The insertion procedure consists of two operations: addition of 
a labeled box and addition of a ribbon.
As in the case of symDTR, a ribbon added to a ballot tableau 
is classified into two types: a Dyck ribbon, and a ballot ribbon.

Let $\lambda$ be a ballot path and $L$ be a natural label of tree $T(\lambda)$.
Let $\mathbf{h}:=(h_1,\ldots,h_{n+n'})$ be the insertion history for $L$, 
where $n$ (resp. $n'$) is the numbers of edges without (resp. with) a dot.

We will construct a ballot tableau from an insertion history $\mathbf{h}$.
Suppose that we have a ballot tableau for $\mathbf{h}:=(h_{1},\ldots,h_{n+n'})$. 
Then, we want to construct a ballot tableau for $\mathbf{h}':=(h_{1},\ldots,h_{n+n'},h_{n+n'+1})$. 

The insertion procedure for addition of a labeled box is given as follows.
\begin{enumerate}
\item We divide the ballot tableau for $\mathbf{h}$ into two pieces along the vertical
line $x=h_{n+n'+1}$. Then, we translated the right piece right by $(2,0)$ (resp. $(1,1)$) 
if $h_{n+n'+1}$ does not have a box (resp. $h_{n+n'+1}$ has a box).
This process is independent of the existence of a circle on $h_{n+n'+1}$.

\item We connect the top paths of the two pieces by the Dyck path $UD$ if $h_{n+n'+1}$ does 
not have a box, or by the ballot path $U$ if $h_{n+n'+1}$ has a box.
Then, we put a box with the label $n+n'+1$ on the top box at $x=h_{n+n'+1}$.
Similarly, since the ballot path $\lambda$ is cut into two pieces, we connect the two paths 
by the Dyck path $UD$ if $h_{n+n'+1}$ does not have a box, or by the ballot path $U$ 
if $h_{n+n'+1}$ has a box.	
We denote by $\lambda_{\mathrm{new}}$ the new ballot path.

\item 
Let $B$ be a box with label smaller than $n+n'+1$. 
Suppose that a labeled box $B$ is at the $m$-th floor, and let $B_{0}$ be an anchor 
box at the zeroth floor below $B$.
This anchor box $B_0$ is characterized by a pair of an up step $s_{u}$ and a 
down step $s_{d}$ or by a single up step $s_{u}'$.
We have two cases according to the existence of a box on $h_{n+n'+1}$.

\begin{enumerate}
\item The entry $h_{n+n'+1}$ does not have a box.
If the line $x=h_{n+n'+1}$ is placed between $s_{u}$ and $s_{d}$, we move the labeled box 
$B$ by $(1,-1)$.
If the both $s_{u}$ and $s_{d}$ are right to the line $x=h_{n+n'+1}$, we move $B$ by $(2,0)$.
If $s_{u}'$ is left to the line $x=h_{n+n'+1}$, we move $B$ by $(2,-2)$. 
If $s_{u}'$ is right to the line $x=h_{n+n'+1}$, we move $B$ by $(2,0)$.
Otherwise, we do not change the position of a labeled box.
\item The entry $h_{n+n'+1}$ has a box.
If $B$ is characterized by the up step $s_{u}'$, we move $B$ by $(1,-1)$.
Otherwise, we do not change the position of a labeled box.
\end{enumerate}
\item When we divide the ballot tableau into two pieces in the step (1), 
we also divide ribbons into two pieces. 
Recall that a ribbon consists of parallel and turn boxes and we have a 
line inside these boxes. 
When a ribbon connects two boxes labeled by $i$ and $i+1$, the new ribbon 
also connects the boxes labeled by $i$ and $i+1$.
We connect the divided lines by a new line such that we have a turn box, 
a parallel box, a labeled box, or a terminal box below the box labeled by $n+n'+1$. 
The new path is as high as possible in the tableau.
\item We change the bottom path of a ballot tableau from $\underline{\lambda}$ 
to $\underline{\lambda_{\mathrm{new}}}$ ($\lambda_{\mathrm{new}}$ is defined 
in the step (2)).
We have fixed the new path $\lambda_{\mathrm{new}}$, the top path, and the positions of the labeled boxes.
Thus, we put single boxes in the remaining region, which results in a ballot tableau.
\end{enumerate}

\begin{remark}
The definition of ballot tableaux in Section \ref{sec:BTsDT} 
gives the same ballot tableaux obtained by the insertion procedures.
The introduction of frozen regions in Section \ref{sec:BTsDT} may seem to be artificial.
However, in the insertion procedure, we connect the divided lines representing a ribbon
by a turn box or a parallel box below a labeled box in the step (4). 
This construction of a ribbon is compatible with the introduction of frozen regions.
\end{remark}

\begin{example}
We consider the same natural label as Example \ref{ex:symDTR}.
We list up ballot tableaux up to the fifth step (see Figure \ref{fig:symDTab}).
We obtain the same diagram as in Example \ref{ex:symDTab1} after the 
sixth insertion.
\begin{figure}[ht]
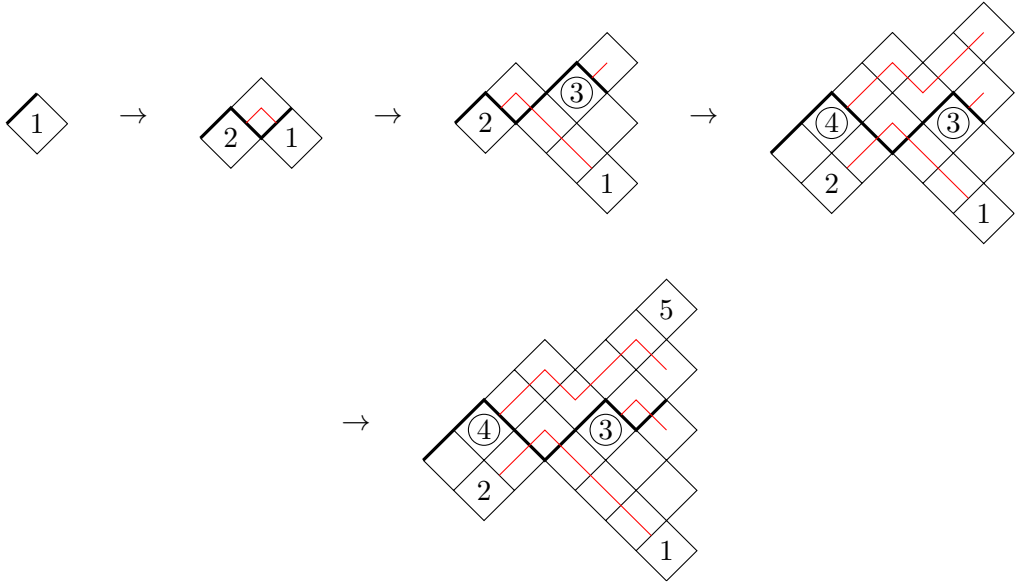

\tikzpic{-0.5}{[x=0.4cm,y=0.4cm]
\draw[very thick](0,0)--(1,1);
\draw(0,0)--(1,-1)--(2,0)(1,1)--(2,0);
\draw(1,0)node{$1$};
}
\quad$\rightarrow$\quad
\tikzpic{-0.5}{[x=0.4cm,y=0.4cm]
\draw[very thick](0,0)--(1,1)--(2,0)--(3,1);
\draw(1,1)--(2,2)--(3,1);
\draw(0,0)--(1,-1)--(2,0)--(3,-1)--(4,0)(3,1)--(4,0);
\draw(1,0)node{$2$};\draw(3,0)node{$1$};
\draw[red](1.5,0.5)--(2,1)--(2.5,0.5);
}
\quad$\rightarrow$\quad
\tikzpic{-0.5}{[x=0.4cm,y=0.4cm]
\draw[very thick](0,0)--(1,1)--(2,0)--(4,2)--(5,1);
\draw(1,1)--(2,2)--(3,1);
\draw(0,0)--(1,-1)--(2,0)--(5,-3)--(6,-2)(3,1)--(6,-2);
\draw(5,1)--(6,0)--(4,-2)(3,-1)--(5,1)(4,2)--(5,3)--(6,2)--(5,1);
\draw(1,0)node{$2$}(4,1)node{$\circnum{3}$}(5,-2)node{$1$};
\draw[red](1.5,0.5)--(2,1)--(4.5,-1.5)(4.5,1.5)--(5,2);
}
\quad$\rightarrow$\quad
\tikzpic{-0.5}{[x=0.4cm,y=0.4cm]
\draw[very thick](0,0)--(2,2)--(4,0)--(6,2)--(7,1);
\draw(1,1)--(3,-1)(3,3)--(8,-2)(4,4)--(6,2)(7,1)--(8,0)
     (2,2)--(4,4)(5,3)--(7,5)--(8,4)--(6,2)(6,4)--(8,2)--(7,1)
     (1,-1)--(5,3)(5,-1)--(7,1)(6,-2)--(8,0)
     (0,0)--(2,-2)--(4,0)--(7,-3)--(8,-2);
\draw(2,-1)node{$2$}(2,1)node{$\circnum{4}$}(7,-2)node{$1$}(6,1)node{$\circnum{3}$};
\draw[red](2.5,-0.5)--(4,1)--(6.5,-1.5)(6.5,1.5)--(7,2)(2.5,1.5)--(4,3)--(5,2)--(7,4);
}\\[12pt]
\quad$\rightarrow$\quad
\tikzpic{-0.5}{[x=0.4cm,y=0.4cm]
\draw[very thick](0,0)--(2,2)--(4,0)--(6,2)--(7,1)--(8,2);
\draw(2,2)--(4,4)--(9,-1)--(7,-3)(1,-1)--(8,6)--(9,5)--(8,4)
     (3,3)--(9,-3)--(8,-4)(0,0)--(2,-2)--(4,0)--(8,-4)(1,1)--(3,-1)
     (6,4)--(9,1)--(6,-2)(7,5)--(9,3)--(5,-1)(6,2)--(8,4);
\draw(2,-1)node{$2$}(2,1)node{$\circnum{4}$}(8,-3)node{$1$}
     (6,1)node{$\circnum{3}$}(8,5)node{$5$};
\draw[red](2.5,-0.5)--(4,1)--(7.5,-2.5)(6.5,1.5)--(7,2)--(8,1)
          (2.5,1.5)--(4,3)--(5,2)--(7,4)--(8,3);
}
\caption{An example of the insertion procedure.}
\label{fig:symDTab}
\end{figure}
\end{example}

We list all ballot tableaux of size $2$ in Figure \ref{fig:BTn2}.
\begin{figure}[ht]
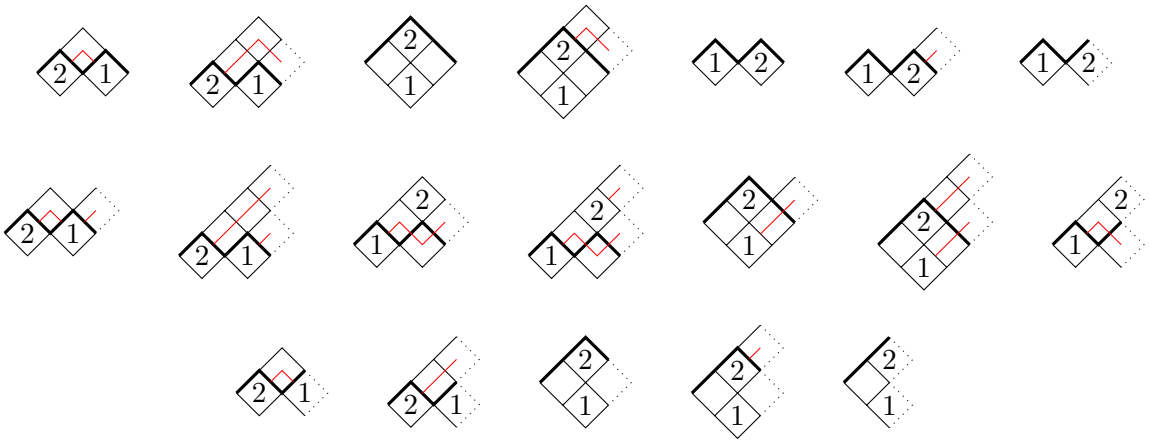

\tikzpic{-0.5}{[x=0.3cm,y=0.3cm]
\draw[very thick](0,0)--(1,1)--(2,0)--(3,1)--(4,0);
\draw(0,0)--(1,-1)--(2,0)--(3,-1)--(4,0)(1,1)--(2,2)--(3,1);
\draw[red](1.5,0.5)--(2,1)--(2.5,0.5);
\draw(1,0)node{$2$}(3,0)node{$1$};
}\quad
\tikzpic{-0.5}{[x=0.3cm,y=0.3cm]
\draw[very thick](0,0)--(1,1)--(2,0)--(3,1)--(4,0);
\draw(0,0)--(1,-1)--(2,0)--(3,-1)--(4,0)(1,1)--(3,3)--(4,2)(2,2)--(3,1)--(4,2);
\draw[dotted](4,2)--(5,1)--(4,0);
\draw[red](1.5,0.5)--(3,2)--(4,1);
\draw(1,0)node{$2$}(3,0)node{$1$};
}\quad
\tikzpic{-0.5}{[x=0.3cm,y=0.3cm]
\draw[very thick](0,0)--(2,2)--(4,0);
\draw(1,1)--(3,-1)(1,-1)--(3,1)(0,0)--(2,-2)--(4,0);
\draw(2,1)node{$2$}(2,-1)node{$1$};
}\quad
\tikzpic{-0.5}{[x=0.3cm,y=0.3cm]
\draw[very thick](0,0)--(2,2)--(4,0);
\draw(1,1)--(3,-1)(1,-1)--(3,1)(0,0)--(2,-2)--(4,0)(2,2)--(3,3)--(4,2)(3,1)--(4,2);
\draw[dotted](4,2)--(5,1)--(4,0);
\draw[red](2.5,1.5)--(3,2)--(4,1);
\draw(2,1)node{$2$}(2,-1)node{$1$};
}\quad
\tikzpic{-0.5}{[x=0.3cm,y=0.3cm]
\draw[very thick](0,0)--(1,1)--(2,0)--(3,1)--(4,0);
\draw(0,0)--(1,-1)--(2,0)--(3,-1)--(4,0);
\draw(1,0)node{$1$}(3,0)node{$2$};
}\quad
\tikzpic{-0.5}{[x=0.3cm,y=0.3cm]
\draw[very thick](0,0)--(1,1)--(2,0)--(3,1)--(4,0);
\draw(0,0)--(1,-1)--(2,0)--(3,-1)--(4,0)(3,1)--(4,2);
\draw[dotted](4,2)--(5,1)--(4,0);
\draw[red](3.5,0.5)--(4,1);
\draw(1,0)node{$1$}(3,0)node{$2$};
}\quad
\tikzpic{-0.5}{[x=0.3cm,y=0.3cm]
\draw[very thick](0,0)--(1,1)--(2,0)--(3,1);
\draw(0,0)--(1,-1)--(2,0)--(3,-1);
\draw[dotted](3,1)--(4,0)--(3,-1);
\draw(1,0)node{$1$}(3,0)node{$2$};
}\\[12pt]
\tikzpic{-0.5}{[x=0.3cm,y=0.3cm]
\draw[very thick](0,0)--(1,1)--(2,0)--(3,1)--(4,0);
\draw(0,0)--(1,-1)--(3,1)(2,0)--(3,-1)--(4,0)(1,1)--(2,2)--(3,1)--(4,2);
\draw[dotted](4,2)--(5,1)--(4,0);
\draw[red](1.5,0.5)--(2,1)--(2.5,0.5)(3.5,0.5)--(4,1);
\draw(1,0)node{$2$}(3,0)node{$1$};
}\quad
\tikzpic{-0.5}{[x=0.3cm,y=0.3cm]
\draw[very thick](0,0)--(1,1)--(2,0)--(3,1)--(4,0);
\draw(0,0)--(1,-1)--(2,0)--(3,-1)--(4,0)(1,1)--(4,4)(2,2)--(3,1)--(4,2)--(3,3);
\draw[dotted](4,4)--(5,3)--(4,2)--(5,1)--(4,0);
\draw[red](1.5,0.5)--(4,3)(3.5,0.5)--(4,1);
\draw(1,0)node{$2$}(3,0)node{$1$};
}\quad
\tikzpic{-0.5}{[x=0.3cm,y=0.3cm]
\draw[very thick](0,0)--(1,1)--(2,0)--(3,1)--(4,0);
\draw(0,0)--(1,-1)--(2,0)--(3,-1)--(4,0)(1,1)--(3,3)--(4,2)(2,2)--(3,1)--(4,2);
\draw[dotted](4,2)--(5,1)--(4,0);
\draw[red](1.5,0.5)--(2,1)--(3,0)--(4,1);
\draw(1,0)node{$1$}(3,2)node{$2$};
}\quad
\tikzpic{-0.5}{[x=0.3cm,y=0.3cm]
\draw[very thick](0,0)--(1,1)--(2,0)--(3,1)--(4,0);
\draw(0,0)--(1,-1)--(2,0)--(3,-1)--(4,0)(1,1)--(3,3)--(4,2)(2,2)--(3,1)--(4,2)
(3,3)--(4,4);
\draw[dotted](4,4)--(5,3)--(4,2)--(5,1)--(4,0);
\draw[red](1.5,0.5)--(2,1)--(3,0)--(4,1)(3.5,2.5)--(4,3);
\draw(1,0)node{$1$}(3,2)node{$2$};
}\quad
\tikzpic{-0.5}{[x=0.3cm,y=0.3cm]
\draw[very thick](0,0)--(2,2)--(4,0);
\draw(0,0)--(2,-2)--(4,0)(1,-1)--(4,2)(1,1)--(3,-1);
\draw[dotted](4,2)--(5,1)--(4,0);
\draw[red](2.5,-0.5)--(4,1);
\draw(2,-1)node{$1$}(2,1)node{$2$};
}\quad
\tikzpic{-0.5}{[x=0.3cm,y=0.3cm]
\draw[very thick](0,0)--(2,2)--(4,0);
\draw(0,0)--(2,-2)--(4,0)(1,-1)--(4,2)(1,1)--(3,-1)(2,2)--(4,4)(3,3)--(4,2);
\draw[dotted](4,4)--(5,3)--(4,2)--(5,1)--(4,0);
\draw[red](2.5,-0.5)--(4,1)(2.5,1.5)--(4,3);
\draw(2,-1)node{$1$}(2,1)node{$2$};
}\quad
\tikzpic{-0.5}{[x=0.3cm,y=0.3cm]
\draw[very thick](0,0)--(1,1)--(2,0)--(3,1);
\draw(0,0)--(1,-1)--(2,0)--(3,-1)(1,1)--(3,3)(2,2)--(3,1);
\draw[dotted](3,3)--(4,2)--(3,1)--(4,0)--(3,-1);
\draw[red](1.5,0.5)--(2,1)--(3,0);
\draw(1,0)node{$1$}(3,2)node{$2$};
}\\[12pt]
\tikzpic{-0.5}{[x=0.3cm,y=0.3cm]
\draw[very thick](0,0)--(1,1)--(2,0)--(3,1);
\draw(1,1)--(2,2)--(3,1)(0,0)--(1,-1)--(2,0)--(3,-1);
\draw[dotted](3,1)--(4,0)--(3,-1);
\draw[red](1.5,0.5)--(2,1)--(2.5,0.5);
\draw(1,0)node{$2$}(3,0)node{$1$};
}\quad
\tikzpic{-0.5}{[x=0.3cm,y=0.3cm]
\draw[very thick](0,0)--(1,1)--(2,0)--(3,1);
\draw(0,0)--(1,-1)--(2,0)--(3,-1)(1,1)--(3,3)(2,2)--(3,1);
\draw[dotted](3,3)--(4,2)--(3,1)--(4,0)--(3,-1);
\draw[red](1.5,0.5)--(3,2);
\draw(1,0)node{$2$}(3,0)node{$1$};
}\quad
\tikzpic{-0.5}{[x=0.3cm,y=0.3cm]
\draw[very thick](0,0)--(2,2)--(3,1);
\draw(0,0)--(2,-2)--(3,-1)(1,1)--(3,-1)(1,-1)--(3,1);
\draw[dotted](3,1)--(4,0)--(3,-1);
\draw(2,-1)node{$1$}(2,1)node{$2$};
}\quad
\tikzpic{-0.5}{[x=0.3cm,y=0.3cm]
\draw[very thick](0,0)--(2,2)--(3,1);
\draw(0,0)--(2,-2)--(3,-1)(1,1)--(3,-1)(1,-1)--(3,1)(2,2)--(3,3);
\draw[dotted](3,3)--(4,2)--(3,1)--(4,0)--(3,-1);
\draw[red](2.5,1.5)--(3,2);
\draw(2,-1)node{$1$}(2,1)node{$2$};
}\quad
\tikzpic{-0.5}{[x=0.3cm,y=0.3cm]
\draw[very thick](0,0)--(2,2);
\draw(0,0)--(2,-2)(1,1)--(2,0)--(1,-1);
\draw[dotted](2,2)--(3,1)--(2,0)--(3,-1)--(2,-2);
\draw[red];
\draw(2,-1)node{$1$}(2,1)node{$2$};
}
\caption{19 ballot tableaux of size $2$.}
\label{fig:BTn2}
\end{figure}

\subsection{Ballot tableaux and symDTR bijection}
In \cite{ABDH11}, a Dyck tableau of order $n$ is bijective to a cover-inclusive 
Dyck tiling above $\lambda=(UD)^{n}$.
In \cite{S19}, the notion of Dyck tableaux are generalized to Dyck tilings above 
general lower Dyck paths.
In this subsection, we generalize the above bijection to fundamental symmetric
Dyck tilings above a general ballot path $\lambda$.
Let $L$ be a natural label of the tree $T(\lambda)$, which may have circles on labels 
and dotted edges.
If we restricted $\lambda$ to Dyck paths and natural labels $L$ without circles and dots,
the bijection coincides with the one constructed in \cite{S19}.

Let $\lambda$ be a ballot path, $L$ a natural label of the tree $T(\lambda)$,
$\mathrm{BTab}(L)$ the ballot tableaux for the natural label $L$, and $\mu$ the 
top path of $\mathrm{BTab}(L)$.
We will construct a correspondence between $\mathrm{BTab}(L)$ and a cover-inclusive fundamental 
symmetric Dyck tiling whose boundaries are $\lambda$ and $\mu$.

Given a natural label $L$ of a tree $T$ (for a fundamental symmetric Dyck 
tiling, not for a ballot tiling), we will construct a map from a ballot tableau 
$\mathrm{BTab}(L)$ to a fundamental symmetric Dyck tiling $\mathcal{D}$.
Recall that an anchor box in the zeroth floor corresponds to a 
pair of an up step and a down step in a ballot path $\lambda$, or 
to an up step, which corresponds to a dotted edge in $L$, in a ballot 
path $\lambda$.
Suppose that an anchor box $a$ corresponds to an edge without a dot in $L$.
When the anchor box $a$ in $\mathrm{BTab}(L)$ is in the $p$-th floor with $p\ge1$,
we have $p$ non-trivial Dyck tiles above the corresponding 
pair of the up and down steps.
Similarly, suppose that a boundary anchor box $a^{\ast}$ corresponds to a dotted edge 
in $L$. 
Then, when the anchor box $a^{\ast}$ is in the $p$-th floor with $p\ge1$,
we have $p$ non-trivial ballot tiles above the corresponding up step.
\begin{figure}[ht]
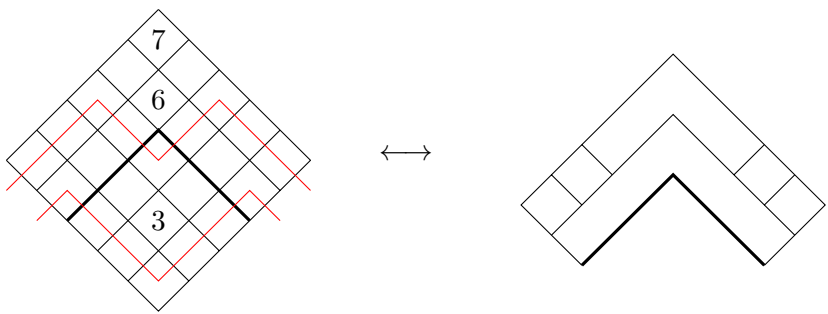

\tikzpic{-0.5}{[x=0.4cm,y=0.4cm]
\draw[very thick](0,0)--(3,3)--(6,0);
\draw(-2,2)--(3,-3)--(8,2)(-1,3)--(4,-2)(0,4)--(5,-1)(1,5)--(3,3)
(2,6)--(7,1)(3,7)--(8,2);
\draw(-2,2)--(3,7)(-1,1)--(4,6)(3,3)--(5,5)(1,-1)--(6,4)(2,-2)--(7,3);
\draw[red](-1,0)--(0,1)--(3,-2)--(6,1)--(7,0);
\draw[red](-2,1)--(1,4)--(2,3)--(3,2)--(5,4)--(6,3)--(8,1);
\draw(3,0)node{$3$}(3,4)node{$6$}(3,6)node{$7$};
}\qquad
$\longleftrightarrow$
\qquad
\tikzpic{-0.5}{[x=0.4cm,y=0.4cm]
\draw[very thick](0,0)--(3,3)--(6,0);
\draw(0,0)--(-2,2)--(3,7)--(8,2)--(6,0)(-1,3)--(0,2)(0,4)--(1,3);
\draw(-1,1)--(3,5)--(7,1)(5,3)--(6,4)(6,2)--(7,3);
}
\caption{A partial ballot tableau and a cover-inclusive Dyck tiling.}
\label{fig:BTtoDT}
\end{figure}
Figure \ref{fig:BTtoDT} is an example of partial correspondence between 
a ballot tableau and a cover-inclusive Dyck tiling.
The anchor box with label $3$ is in the first floor. So, we have one 
Dyck tile of size $4$. The anchor boxes with labels $6$ and $7$ are 
in the second floor. So we have two Dyck tiles of size $3$. 
However, one of the Dyck tiles of size $3$ is inside of the Dyck tile 
of size $4$. Therefore, in total, we have two Dyck tiles, whose sizes
are $3$ and $4$.

The lower boundary of $\mathcal{D}$ is $\lambda$ and the top boundary of 
$\mathcal{D}$ is the same as the one of $\mathrm{BTab}(L)$, that is, $\mu$.
The floor where an anchor box lies fixes the positions of 
non-trivial Dyck or ballot tiles in $\mathcal{D}$.
Note that once top and lower boundaries are fixed and the positions of 
non-trivial Dyck and ballot tiles are fixed, one can reconstruct a fundamental 
symmetric Dyck tiling by putting single boxes in the remaining region. 
In this way, we obtain a fundamental symmetric Dyck tiling $\mathcal{D}$ 
from a ballot tableau $\mathrm{BTab}(L)$.

Summarizing the above map, we have the following theorem (see also Theorem 4.18 in \cite{S19}
for Dyck tableaux and Dyck tilings).
\begin{theorem}
There is a bijection between the cover-inclusive fundamental symmetric Dyck tilings
whose lower path is $\lambda$ of length $(n,n')$ and the ballot tableaux associated 
with $\lambda$ with $n+n'$ labeled boxes.
\end{theorem}

We give an example of the correspondence between a ballot tableau and 
a cover-inclusive fundamental symmetric Dyck tiling in Figure \ref{fig:bijBTfsDT}.

\begin{figure}[ht]
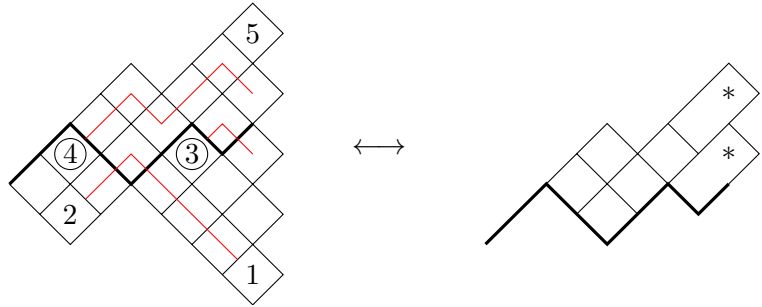

\tikzpic{-0.5}{[x=0.4cm,y=0.4cm]
\draw[very thick](0,0)--(2,2)--(4,0)--(6,2)--(7,1)--(8,2);
\draw(2,2)--(4,4)--(9,-1)--(7,-3)(1,-1)--(8,6)--(9,5)--(8,4)
     (3,3)--(9,-3)--(8,-4)(0,0)--(2,-2)--(4,0)--(8,-4)(1,1)--(3,-1)
     (6,4)--(9,1)--(6,-2)(7,5)--(9,3)--(5,-1)(6,2)--(8,4);
\draw(2,-1)node{$2$}(2,1)node{$\circnum{4}$}(8,-3)node{$1$}
     (6,1)node{$\circnum{3}$}(8,5)node{$5$};
\draw[red](2.5,-0.5)--(4,1)--(7.5,-2.5)(6.5,1.5)--(7,2)--(8,1)
          (2.5,1.5)--(4,3)--(5,2)--(7,4)--(8,3);
}\qquad $\longleftrightarrow$\qquad
\tikzpic{-0.5}{[x=0.4cm,y=0.4cm]
\draw[very thick](0,0)--(2,2)--(4,0)--(6,2)--(7,1)--(8,2);
\draw(2,2)--(4,4)--(6,2)(3,1)--(5,3)(3,3)--(5,1)
(5,3)--(8,6)--(9,5)--(8,4)--(9,3)--(8,2);
\draw(6,2)--(8,4)(7,3)--(6,4);
\draw(8,5)node{$\ast$}(8,3)node{$\ast$};
}
\caption{A bijection between a ballot tableau and a cover-inclusive 
fundamental symmetric Dyck tiling}
\label{fig:bijBTfsDT}
\end{figure}

\subsection{Generalized patterns and shadow and 
clear boxes of ballot tableaux}

In \cite{ABDH11}, they study the generalized patterns of a permutation and 
its relation a Dyck tableau whose lower boundary is a zig-zag path.
It was shown that a pattern is bijective to an object in a Dyck tableau
such as a ribbon, a shadow or clear box.
In \cite{S19}, the generalized patterns are defined in terms of 
trees which contain permutations as special cases. 
By introducing the notion of proper shadow and proper clear boxes in 
a Dyck tableau, a generalized pattern in a tree is shown to be bijective to 
an object in a generalized Dyck tableau.
In this subsection, we study a connection between a generalized pattern 
in a tree associated with a ballot path, and an object in a ballot 
tableau such as a ribbon, a proper shadow or a proper clear box. 
The case of Dyck tableaux associated with a general tree is included 
in the case of ballot tableaux since we have a generalized Dyck tableau
by restricting ourselves to the case of ballot tableaux with neither dotted 
edges nor circled edges.

\begin{defn}
Given a ballot tableau, let $a$ be a box which is either a parallel, turn or terminal box,
and $b$ be a labeled box.
We call $a$ a shadow (resp. clear) box if $a$ is above (resp. below) $b$ such that 
there is no labeled box between $a$ and $b$.
\end{defn}

Let $B=\mathrm{BTab}(L)$ be a ballot tableau for a natural label $L$.
Recall that $B$ consists of labeled boxes, paths connecting two labeled boxes
and empty boxes.
If we have a $\wedge$-turn box $t$ such that a box below $t$ is an empty 
boxes, we locally transform the path containing $t$ as Figure \ref{fig:lmove}.
\begin{figure}[ht]
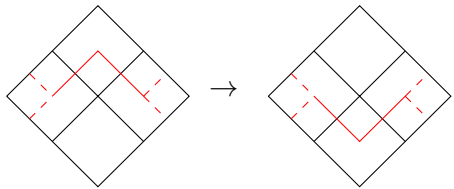

\tikzpic{-0.5}{[scale=0.6]
\draw(0,0)--(2,2)--(4,0)--(2,-2)--(0,0);
\draw(1,1)--(3,-1)(1,-1)--(3,1);
\draw[red](1,0)--(2,1)--(3,0);
\draw[red,dashed](0.5,0.5)--(1,0)(0.5,-0.5)--(1,0);
\draw[red,dashed](3,0)--(3.5,0.5)(3,0)--(3.5,-0.5);
}
$\rightarrow$
\tikzpic{-0.5}{[scale=0.6]
\draw(0,0)--(2,2)--(4,0)--(2,-2)--(0,0);
\draw(1,1)--(3,-1)(1,-1)--(3,1);
\draw[red](1,0)--(2,-1)--(3,0);
\draw[red,dashed](0.5,0.5)--(1,0)(0.5,-0.5)--(1,0);
\draw[red,dashed](3,0)--(3.5,0.5)(3,0)--(3.5,-0.5);
}
\caption{A local move for a path}
\label{fig:lmove}
\end{figure}
In the process of a local move, we never move the positions of 
labeled boxes.
We denote by $B'$ the tableau obtained from $B$ by local moves such that 
we can not apply a local move to any $\wedge$-turn box in $B'$.
This means that one cannot lower paths in $B'$.

Note that a shadow (resp. clear) box in $B$ is obviously bijective 
to a shadow (resp. clear) box in $B'$.

\begin{defn}
\label{defn:pshadow}
Let $B'$ be the tableau obtained from $B$ as above.
Let $s$ be a shadow (resp. clear) box above (resp. below) a labeled box 
$b$ in a ballot tableau $B$. 
We have two cases: 1) $s$ is not a terminal box, and 2) $s$ is a terminal box.
\begin{enumerate}
\item
We call $s$ a proper shadow (resp. clear) box if there is neither empty 
boxes nor a labeled box below (resp. above) $s$ and above (resp. below) 
$b$ in $B'$.
\item 
All terminal boxes which are shadow (resp. clear) boxes are 
proper shadow (resp. clear) boxes.  
\end{enumerate}
\end{defn}

\begin{remark}
In Definition \ref{defn:pshadow}, all shadow or clear terminal boxes are classified as 
proper shadow or proper clear boxes, or both. 
The number of proper shadow boxes associated with terminal boxes 
is less than or equal to the number of circles in the natural label.
\end{remark}

An example of shadow boxes and proper shadow boxes in a ballot tableau is 
shown in Figure \ref{fig:sbBT}.
\begin{figure}[ht]
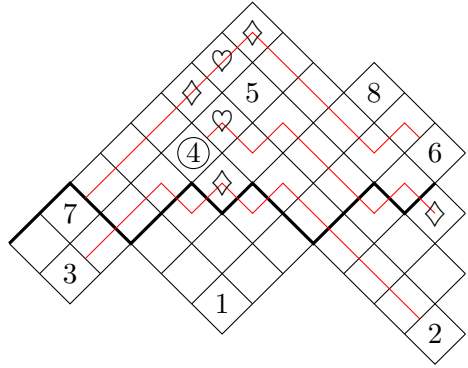

\tikzpic{-0.5}{[scale=0.4]
\draw[very thick](0,0)--(2,2)--(4,0)--(6,2)--(7,1)--(8,2)--(10,0)--(12,2)--(13,1)--(14,2);
\draw(0,0)--(2,-2)--(4,0)--(7,-3)--(10,0)--(14,-4)--(15,-3);
\draw(2,2)--(8,8)--(12,4)--(13,5)--(15,3);
\draw(1,1)--(3,-1)(3,3)--(8,-2)(4,4)--(9,-1)(5,5)--(10,0)(6,6)--(15,-3)(7,7)--(15,-1)
     (12,4)--(15,1);
\draw(1,-1)--(9,7)(4,0)--(10,6)(5,-1)--(11,5)(6,-2)--(12,4)(12,2)--(14,4);
\draw(11,-1)--(15,3)(12,-2)--(15,1)(13,-3)--(15,-1);
\draw(11,5)--(12,6)--(13,5);
\draw(2,1)node{$7$}(2,-1)node{$3$}(5.9,3)node{\circnum{4}}(7,-2)node{$1$}(8,5)node{$5$}
     (14,-3)node{$2$}(14,3)node{$6$}(12,5)node{$8$};
\draw[red](2.5,1.5)--(8,7)--(12,3)--(13,4)--(13.5,3.5)
          (6.5,3.5)--(7,4)--(8,3)--(9,4)--(12,1)--(13,2)--(14,1)
          (2.5,-0.5)--(5,2)--(6,1)--(7,2)--(8,1)--(9,2)--(13.5,-2.5);
\draw(6,5)node{$\diamondsuit$}(8,7)node{$\diamondsuit$}(7,2)node{$\diamondsuit$}(14,1)node{$\diamondsuit$};
\draw(7,4)node{$\heartsuit$}(7,6)node{$\heartsuit$};
}
\caption{Shadow boxes for a ballot tableau.
Boxes with $\diamondsuit$ and $\heartsuit$ are shadow boxes.
Boxes with $\diamondsuit$ are proper shadow boxes.}
\label{fig:sbBT}
\end{figure}

Let $\lambda$ be a ballot path.
To define the generalized patterns, we introduce the notion of a position tree.
Let $\mathcal{E}_{\bullet}$ be the set of edges with dots in the tree $\mathrm{Tree}(\lambda)$.
\begin{defn}[Position tree]
A tree $\mathrm{PosTree}(\lambda)$ is a tree such that its shape 
is $\mathrm{Tree}(\lambda)$ and its edge $E$ has the label $\mathrm{Pos}(E)$
given by 
\begin{align*}
\mathrm{Pos}(E):=&2\cdot\#\{E'|E'\leftarrow E\}+\#\{E'| E\uparrow E'\}+\#\{E'| E'\uparrow E\} \\
&+\#\{E'| E\in\mathcal{E}_{\bullet}, E'\not\in\mathcal{E}_{\bullet}, E'\uparrow E\}+1.
\end{align*}
We call $\mathrm{PosTree}(\lambda)$ the position tree for a ballot path $\lambda$.
\end{defn}

Figure \ref{fig:postree} is an example of a position tree. 
Compare this with the ballot tableau in Figure \ref{fig:bijBTfsDT}.
The position tree encodes the column number of the labeled boxes. 

\begin{figure}
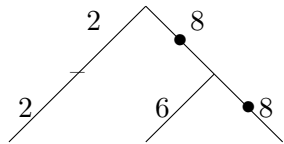

\tikzpic{0.5}{[scale=0.6]
\coordinate
	child{coordinate(c2)
		child{coordinate(c4)}
		child[missing]
		child[missing]	
	}
	child[missing]
	child{coordinate(c1)
		child{coordinate(c3)}
		child[missing]
		child{coordinate(c5)}
	};
\node[anchor=south east] at($(0,0)!0.5!(c2)$){$2$};
\node[anchor=east] at($(c2)!0.5!(c4)$){$2$};
\node[anchor=south west] at($(0,0)!0.5!(c1)$){$8$};
\node[anchor=east] at($(c1)!0.5!(c3)$){$6$};
\node[anchor=west] at($(c1)!0.5!(c5)$){$8$};
\node at($(0,0)!0.5!(c1)$){$\bullet$};
\node at($(c1)!0.5!(c5)$){$\bullet$};
\node at($(c2)$){--};
}
\caption{A position tree for a ballot path $UUDDUUDU$.}
\label{fig:postree}
\end{figure}

Recall that an edge $E$ without a dot in a tree $\mathrm{Tree}(\lambda)$ has a corresponding 
up and down steps in a ballot tableau. 
Similarly an edge $E$ with a dot has a corresponding up step in the ballot tableau.
Then, the label $\mathrm{Pos}(E)$ in $\mathrm{Tree}(\lambda)$ indicates the position of the label 
from left in a ballot tableau corresponding to the edge $E$.
In general, the labels of a position tree is neither increasing nor decreasing from the root 
to a leaf.

We extend the notions of patterns $2^{+}2$ and $1^{+}21$ for Dyck tableaux 
to the ones for ballot tableaux.
Recall that we can obtain a symmetric Dyck tableau by gluing a ballot tableau and 
its mirror image.
Then, when a label $l$ in a tree has a circle, we have a line from the north-east 
edge of the box labeled by $l$ ends at the terminal box. 
By gluing a ballot tableau with the mirror image, we have a $2^{+}2$ pattern 
in the symmetric Dyck tiling. 
This observation gives the following definitions of patterns $2^{+}2$, $2^{+}12$,
and $1^{+}21$ for the natural label $L$ for a fundamental symmetric Dyck tiling.
See \cite{ABN11} for the definitions of patterns used in the study of Dyck tableaux.
The definitions of patterns for ballot tableaux contain those for Dyck tableaux.

Let $E(i)$ be the edge labeled by $i$ in the natural label $L$, and $\mathrm{Pos}(E)$
be the label of an edge $E$ in a position tree associated with $L$.

A pattern $2^{+}2$ of a natural label $L$ is a relation of two edges $E(a)$ and $E(b)$
such that $b=a+1$ and $E(a)\rightarrow E(b)$, or a condition on a single edge $E(a)$ such that 
$E(a)$ has a circle.

A pattern $2^{+}12$ of a natural label $L$ is a relation among three edges $E(a)$, $E(b)$ 
and $E(c)$ such that 
\begin{enumerate}
\item $b<c$ and $a=c+1$,
\item $\mathrm{Pos}(E(a))<\mathrm{Pos}(E(b))<\mathrm{Pos}(E(c))$, and  
\item there is no $b'$ such that $\mathrm{Pos}(E(b))=\mathrm{Pos}(E(b'))$ and 
$b<b'<c$,
\end{enumerate}
or a relation between two edges $E(a)$ and $E(b)$ such that 
\begin{enumerate}
\item $E(a)$ has a circle and $b<a$,
\item $\mathrm{Pos}(E(a))<\mathrm{Pos}(E(b))$, and  
\item there is no $b'$ such that $\mathrm{Pos}(E(b))=\mathrm{Pos}(E(b'))$ 
and $b<b'<a$.
\end{enumerate}

A pattern $1^{+}21$ of a natural label $L$ is a relation among 
three edges $E(a)$, $E(b)$ and $E(c)$ such that 
\begin{enumerate}
\item $b>a$ and $a=c+1$,
\item $\mathrm{Pos}(E(a))<\mathrm{Pos}(E(b))<\mathrm{Pos}(E(c))$, and  
\item there is no $b'$ such that $\mathrm{Pos}(E(b))=\mathrm{Pos}(E(b'))$ and 
$b>b'>a$,
\end{enumerate}
or a relation between two edges $E(a)$ and $E(b)$ such that
\begin{enumerate}
\item $E(a)$ has a circle and $b>a$,
\item $\mathrm{Pos}(E(a))<\mathrm{Pos}(E(b))$, and  
\item there is no $b'$ such that $\mathrm{Pos}(E(b))=\mathrm{Pos}(E(b'))$ 
and $b>b'>a$.
\end{enumerate}

\begin{prop}
An added ribbon of a ballot tableau is bijective to the patterns $2^{+}2$ in $L$.
We have two case:
\begin{enumerate}
\item
Suppose that a ribbon connects two boxes with a label.
In a ballot tableau, if we read from left to right the labels of boxes that 
are connected by a ribbon, we get the pattern $2^{+}2$ in $L$.
\item
Suppose that a ribbon connects a labeled box and a terminal box.
If we read the label of the labeled box, we get the pattern $2^{+}2$ in $L$.
\end{enumerate}
\end{prop}
\begin{proof}
In case where two labeled edges $E(a)$ and $E(b)$ satisfy the pattern $2^{+}2$, 
the statement follows from the proof of Proposition 4.22 in \cite{S19}.

Below, we consider the case where $E(a)$ has a circle.
Since $E(a)$ has a circle, we add a ribbon from the box labeled by $a$ 
to the right-most column immediately after we put the box labeled by $a$
on a ballot tableau.
Conversely, suppose that the box labeled by $a$ is connected by a ribbon 
to the right-most column. 
By the inverse of the insertion algorithm, we remove the box labeled by 
$a$ immediately after the ribbon is removed.
This means that $E(a)$ has a circle.
\end{proof}

\begin{prop}
\label{prop:shadow}
A shadow box of $L$ is bijective to a pattern $2^{+}12$.
A clear box of $L$ is bijective to a pattern $1^{+}21$.
\end{prop}
\begin{proof}
We consider the case where three boxes labeled by $a,b$ and $c$ 
satisfy the pattern $2^{+}12$.
The statement follows from the proof of Proposition 4.23 in \cite{S19}. 

Suppose the edges $E(a)$ and $E(b)$ satisfy the pattern $2^{+}12$.
Recall that a ballot tableau $B$ is obtained from a symmetric Dyck tableau 
$D$ by cutting it in the middle.
The conditions for $E(a)$ and $E(b)$ in the ballot tableau $B$ correspond 
to a condition for $E(a), E(b)$ and $E(c)$ satisfying the generalized 
pattern $2^{+}12$ in the symmetric Dyck tableau $D$.
A ribbon in $B$ connected to the right-most column corresponds to 
a ribbon in $D$ connecting the box labeled by $a$ and the 
box labeled by $c=a+1$.
Therefore, this case is equivalent to the case of the pattern $2^{+}12$
in the symmetric Dyck tableau.
Again by the proof of Proposition 4.23 in \cite{S19}, the edges $E(a)$
and $E(b)$ satisfying the patterns $2^{+}12$ is bijective to 
a shadow box.

The proof is the same for the pattern $1^{+}21$.
\end{proof}

\subsection{The shape of a ballot tableau}
In this subsection, we study the top path of a ballot tableau.
By the explicit construction of the insertion process, the top
path is characterized by the tree structure of a natural label.
See \cite{ABDH11,S19} in case of Dyck tableaux.

Let $\lambda$ and $\mu$ be ballot paths of length $(n,n')$ 
satisfying $\lambda\le\mu$, and $T_{1}$ be a label of the tree $\mathrm{Tree}(\lambda)$.
Recall that given an anchor box at the zeroth floor, we have a corresponding 
pair of $U$ and $D$ steps, or a corresponding $U$ step in $\lambda$.
In former case, let $i_{U}$ and $i_{D}$ be the positions of these $U$ and $D$ steps in $\lambda$
from left. In the latter case, let $i_{U}$ be the position of the $U$ step in $\lambda$.
The pair of $U$ and $D$ steps corresponds to an edge $E$ without a dot in $\mathrm{Tree}(\lambda)$ and 
a single $U$ step corresponds to an edge $E$ with a dot in $\mathrm{Tree}(\lambda)$.
We denote by $T_{1}(E)$ the label of the edge $E$ in $T_{1}$ and by $E^{+}$ (resp. $E^{-}$) 
the edge whose label in $T_{1}$ is given by $T_{1}(E)+1$ (resp. $T_{1}(E)-1$).
We denote $\mathrm{lb}(E)$ (resp. $\mathrm{rb}(E)$) the step of the top path $\mu$ of the ballot 
tableau at the position $i_{U}$ (resp. $i_{D}$).
We call $\mathrm{lb}(E)$ (resp. $\mathrm{rb}(E)$) the left border (resp. right border) for the edge $E$.
\begin{prop}
The left border for the edge $E$ in $\mathrm{Tree}(\lambda)$ is give by 
\begin{align*}
\mathrm{lb}(E)=
\begin{cases}
U & \text{if } T_{1}(E)=n+n', \\
U & \text{if } E\rightarrow E^{+},\text{ or } E^{+}\uparrow E, \\
U & \text{if } E^{+}\rightarrow E \text{ and }  E^{+} \text{ is circled}, \\
D & \text{if } E^{+}\rightarrow E \text{ and } E^{+} \text{ is not circled}.
\end{cases}
\end{align*}
Suppose that there exists a right border for the edge $E$. Then,
the right border for the edge $E$ is given by 
\begin{align*}
\mathrm{rb}(E)=
\begin{cases}
D & \text{if } T_{1}(E)=1, \\
D & \text{if } E \text{ is not circled and either } E^{-}\rightarrow E \text{ or } E\uparrow E^{-},\\
U & \text{if } E\rightarrow E^{-} \text{ or } E \text{ is circled}. \\
\end{cases}
\end{align*}
\end{prop}
\begin{proof}
In the insertion procedure of a ballot tableau, we may add a ribbon between the box 
labeled by $j$ and the box labeled by $j+1$ when $j+1$ is not circled.
In this case, we change the left border for the edge with the label $j$ from $U$ to $D$
and the right border for the edge with the label $j+1$ from $D$ to $U$. 
When $j+1$ is circled in $\mathrm{Tree}(\lambda)$, we add a ribbon between the box 
labeled by $j+1$ and a box at $x=2n+n'$. 
This means that we do not change the left and right borders for the edge with the label $j$
and the position of the box labeled by $j$ is irrelevant. 
From these observations, it is enough to consider the three entries $E^{-}, E$ and $E^{+}$. 

If $T_{1}(E)=1$, there is no ribbon starting at the position $i_{D}$. 
We have $\mathrm{rb}(E)=D$ if the right border exists.

Suppose $T_{1}(E)\in[2,n+n']$ and there exists a right border for the edge $E$. 
The right border for the edge $E$ depends on 
whether there is a ribbon starting from $E$:
\begin{enumerate}
\item if $E\rightarrow E^{-}$ or $E$ is circled, 
we have a ribbon between the box labeled by $T_{1}(E)$ 
and the box labeled by $T_{1}(E)-1$ or a right-most box.
This means $\mathrm{rb}(E)=U$.
\item if $E$ is not circled and either $E^{-}\rightarrow E$ or $E\uparrow E^{-}$, 
there is no ribbon which starts from the box labeled by $T_{1}(E)$.
This means that $\mathrm{rb}(E)=D$.
\end{enumerate}

If $T_{1}(E)=n+n'$, there is no ribbon ending at $i_{U}$. 
This means that $\mathrm{lb}(E)=U$.

If $T_{1}(E)\in[1,n+n'-1]$, the left border for the edge $E$ depends on 
whether there is a ribbon ending at $E$:
\begin{enumerate}
\item if $E\rightarrow E^{+}$, $E^{+}\uparrow E$, 
we have no ribbon ending at $E$.
This means $\mathrm{lb}(E)=U$.
\item if $E^{+}\rightarrow E$ and $E^{+}$ is circled, we have a ribbon 
connecting the box labeled by $E^{+}$ and a boundary anchor box. 
This means $\mathrm{lb}(E)=U$.
\item if $E^{+}\rightarrow E$ and $E^{+}$ is not circled, we have a 
ribbon between the box labeled by $T_{1}(E)+1$ and the box labeled 
by $T_{1}(E)$.
This means $\mathrm{lb}(E)=D$.
\end{enumerate}
\end{proof}

\subsection{The (LR/RL)-(minima/maxima) of a fundamental symmetric Dyck tiling}
In this subsection, we introduce (LR/RL)-(minima/maxima) of a natural label $L$
for a fundamental symmetric Dyck tiling, and study its relation to a ballot 
tableau.
See \cite{ABDH11} in case of Dyck tableaux whose lower boundary paths
are zig-zag paths, and \cite{S19} for general lower boundary paths.

The notion of (LR/RL)-(minima/maxima) of $L$ is defined as follows.
\begin{defn}[Section 4.7 in \cite{S19}]
Let $L(E)$ be the label of the edge $E$ in $L$.   
\begin{enumerate}
\item 
$L(E)$ is a {\it right-to-left minimum} (RL-minima) 
if and only if $E$ is connected to the root and 
such that $E\rightarrow E'$ $\Rightarrow$ $L(E)<L(E')$,
\item 
$L(E)$ is a {\it right-to-left maximum} (RL-maxima) 
if and only if  $E$ is connected to a leaf and 
such that $E\rightarrow E'$ $\Rightarrow$ $L(E)>L(E')$,
\item 
$L(E)$ is a {\it left-to-right minimum} (LR-minima) 
if and only if $E$ is connected to the root and 
such that $E'\leftarrow E$ $\Rightarrow$ $L(E)<L(E')$,

\item
$L(E)$ is a {\it left-to-right maximum} (LR-maxima) 
if and only if $E$ is connected to a leaf and 
such that $E'\leftarrow E$ $\Rightarrow$ $L(E)>L(E')$,
\end{enumerate}
\end{defn}

\begin{remark}
If we restrict ourselves to Dyck tilings whose lower boundary 
is a zig-zag Dyck path, the definitions of (LR/RL)-(minima/maxima) 
are equivalent to the standard form considered in \cite{ABDH11}.
In this case, the condition that the edge $E$ is connected to the 
root is equivalent to the condition that $E$ is connected to a leaf.
\end{remark}

\begin{remark}
\label{remark:RL}
Given a natural label $L$ of a fundamental symmetric Dyck tiling, we have 
a unique path from the root to a leaf which consists of only dotted edges of $L$.
By definition, the label on a dotted edge connected to the root is a RL-minima,
and the label on a dotted edge connected to a leaf is a RL-maxima.
\end{remark}

In the following propositions, one can show the statement by the same 
arguments as the proofs from Propositions 4.25 to 4.29 in Section 4.7 in \cite{S19}.

We denote by $\mathrm{BTab}(L)$ the ballot tableau for a natural label $L$, 
and by $n$ the number of labeled boxes in $\mathrm{BTab}(L)$.

From Remark \ref{remark:RL}, we have an obvious RL-minima on a dotted edge in $L$.
We have the following proposition for other RL-minimas which are on edges 
without a dot in $L$.

We say that a labeled box $b$ in a ballot tableau is at the maximal (resp. minimal) 
height if there is no boxes above (resp. below) $b$.
We denote by $\mathrm{Lab}(b)$ the label of a labeled box $b$.
\begin{prop}
\label{prop:RLmin}
A RL-minima of $L$ is bijective to a labeled box $b$ in $\mathrm{BTab}(L)$ such that
\begin{enumerate}
\item it is at the minimal height 
\item if $\mathrm{Lab}(b)$ does not have a circle in $L$, the right border of $b$ is equal to $D$, 
\item if $\mathrm{Lab}(b)$ has a circle in $L$, the ribbon starting from $b$ does not contain shadow 
boxes. 
\end{enumerate}
\end{prop}
In Proposition \ref{prop:RLmin}, the first condition 
insures that the edge corresponding to the anchor box $b$ is connected to 
the root.
The second and third conditions ensure that if $E\rightarrow E'$ we have 
$L(E)<L(E')$.

\begin{prop}
\label{prop:LRmax}
A LR-maxima of $L$ is bijective to a labeled box $b$ in $\mathrm{BTab}(L)$ such that
\begin{enumerate}
\item it is at the maximal height,
\item  the left border of $b$ is equal to $U$.
\end{enumerate}
\end{prop}

In Proposition \ref{prop:LRmax}, the second condition 
insures that the box labeled by $\mathrm{Lab}(b)+1$ is right to 
the box $b$ if $\mathrm{Lab}(b)+1$ does not have a circle in $L$.
When $\mathrm{Lab}(b)+1$ has a circle in $L$, the first condition 
insures that the box labeled by $\mathrm{Lab}(b)+1$ is right to 
the box $b$.
In both cases, the box labeled by $\mathrm{Lab}(b)+1$ is right to 
the box $b$.
By similar reasoning, the box labeled by larger integer than $\mathrm{Lab}(b)$ 
is always right to the box $b$.

\begin{prop}
The box with the label $n$ in $\mathrm{BTab}(L)$ corresponds to 
the rightmost labeled box such that it is at the maximal height 
and its left border is equal to $U$. 
The box with the label $1$ in $\mathrm{BTab}(L)$ corresponds to the leftmost 
labeled box such that it is at the minimal height and its right border equal to $D$,
or the lowest boundary anchor box at the minimal height.
\end{prop}

From Remark \ref{remark:RL}, we have an obvious RL-maxima on a dotted edge in $L$.
We have the following proposition for the other RL-maximas which are on edges 
without a dot in $L$.
\begin{prop}
\label{prop:LRmaxima}
A RL-maxima $j<n$ of $L$ is bijective to a labeled box $b$ in $\mathrm{BTab}(L)$ 
such that 
\begin{enumerate}
\item it is at the maximal height, 
\item if $\mathrm{Lab}(b)+1$ does not have a circle in $L$, the left border of $b$ is equal to $D$,
\item $b$ is right to the box labeled by $n$.
\end{enumerate}
\end{prop}
In Proposition \ref{prop:LRmaxima}, the second condition implies that 
the box labeled by $\mathrm{Lab}(b)+1$ is left to the box $b$ if $\mathrm{Lab}(b)+1$ does 
not have a circle.
From the third condition, one can assume that there exists $i$ such that 
the box labeled by $i$ is right to the box $b$ and the box labeled by $i+1$
is left to the box $b$. 
However, the first condition insures that such $i$ does not exist since 
there is a ribbon from the box labeled by $i+1$.

A box labeled by $1$ in $\mathrm{BTab}(L)$ is obviously  a LR-minima.
We have the following proposition for the other LR-minimas which are on edges 
without a dot in $L$.
\begin{prop}
\label{prop:LRminima}
A LR-minima $j>1$ of $L$ is bijective to a labeled box $b$ in $\mathrm{BTab}(L)$ 
such that 
\begin{enumerate}
\item it is at the minimal height,
\item if $\mathrm{Lab}(b)$ does not have a circle in $L$, the right border of $b$ is equal to $U$,
\item $b$ is left to the box labeled by $1$.
\end{enumerate}
\end{prop}
In Proposition \ref{prop:LRminima}, the second condition is equivalent to that 
the box labeled by $\mathrm{Lab}(b)-1$ is right to the box $b$ if $\mathrm{Lab}(b)$ does not 
have a circle.
From the third condition, one can assume that there exist $i<\mathrm{Lab}(b)$ such that the 
box labeled by $i$ is left to the box $b$, and $i$ is circled in $L$, or 
$i-1$ is right to the box $b$. 
However, the first condition insures that there is no such $i$.

\section{Tree-like tableaux of shifted shapes}
\label{sec:TlTshifted}
\subsection{Insertion procedure to produce a tree-like tableau for a symmetric Dyck tiling}
A tree-like tableau is a Ferrers diagram with labels (or dots) and with some conditions on 
the positions of dots \cite{ABN11}.
Let $T$ be a tree-like tableau.
When the number of labels in $T$ is $n$, we say that $T$ is of size $n$.

In this subsection, we consider symmetric tree-like tableaux, which are symmetric with respect 
to the main diagonal line.
In \cite{ABN11}, symmetric tableaux are defined and the total number of them is studied. 
They realize symmetric tableaux as the set of symmetric tree-like tableaux of size $2n+1$. 
This class of symmetric tree-like tableaux corresponds to the symmetric Dyck tilings with 
a zig-zag path as the lower boundary path.
We generalize symmetric tree-like tableaux in \cite{ABN11} to symmetric tableaux for symmetric Dyck 
tilings whose lower boundary paths are general ballot paths.

Let $\lambda$ be a ballot path of length $(n,n')$. 
In this paper, we consider symmetric tableaux of size $2n+n'+1$ associated with $\lambda$.
As we will see in Section \ref{sec:tltss}, cutting a symmetric tree-like tableau along 
the main diagonal gives a tree-like tableau of a shifted shape and of size $n+n'+1$.

\paragraph{\bf Row, column and diagonal insertion}
Let $F$ be a Ferrers diagram such that it is symmetric along the diagonal line, and 
$F_{i}$ boxes in the $i$-th row.
The {\it half-perimeter} of $F$ is the sum of the numbers of rows and columns.
We introduce four types of insertions: symmetric row insertion, symmetric column insertion, 
symmetric diagonal insertion on a non-diagonal vertex and diagonal insertion on a diagonal vertex.
{\it Boundary edges} of $F$ are the edges that are on the south-east border 
of $F$.

We first define non-symmetric row, column and diagonal insertions following \cite{ABN11,S19}.
Let $e$ be a boundary edge of $F$ that is on the boundary box in the $p$-th row and the $q$-th column.
When $e$ is a horizontal edge, we define the insertion of a row at $e$ by 
$F'_{p+1}=q$, $F'_{i+1}=F_{i}$ for $i\ge p+1$ and $F'_{j}=F_{j}$ for $j\le p$.
We call this insertion a row insertion.
Similarly, when $e$ is a vertical edge, we define the insertion of a column at $e$
by $F'_{i}=F_{i}+1$ for $1\le i\le p$ and $F'_{j}=F_{j}$ for $j\ge p+1$.
We call this insertion a column insertion.
Let $v$ be a vertex on the boundary box at the coordinate $(p,q)$ with $p\neq0$ and $q\neq0$.
We define the insertions of a row and a column at $v$ by 
$F'_{p+1}=q+1$, $F'_{i}=F_{i}+1$ for $1\le i\le p$ and $F'_{j+1}=F_{j}$ for $j\ge p+1$. 
We call this insertion a diagonal insertion.
The definition of a diagonal insertion is independent of whether the vertex is 
on the main diagonal of $F$ or not.

Let $F$ be a Ferrers diagram which is symmetric along the main diagonal.
We enumerate the boundary edges by $1,2,\ldots,N$ from the south-west edge to the north-east 
edge. Here, $N$ is the half-perimeter of the diagram $F$.
Let $e$ be the $i$-th boundary edge and $e'$ be the $N+1-i$-th edge.
We define the symmetric row insertion to be the row insertion at $e$ and the column insertion 
at $e'$.
Similarly, we define the symmetric column insertion to be the column insertion at $e$ and 
the row insertion at $e'$.
Let $v$ be a vertex on the boundary box at $(p,q)$ with $p\neq0$, $q\neq0$ and 
$v'$ be a vertex at $(q,p)$. 
We have two cases: $p\neq q$ and $p=q$.
For the first case, we define the symmetric diagonal insertion on a non-diagonal 
vertex to be the diagonal insertions at $v$ and $v'$.
For the second case we define the symmetric diagonal insertion on the diagonal 
vertex to be the diagonal insertions at $v$.
Note that the diagonal vertex on the main diagonal is unique in $F$.

We introduce the notion of {\it special point} following \cite{ABN11}.
\begin{defn}[Special point]
Let $T$ be a symmetric tree-like tableau.
The special point of $T$ is the northeast-most labeled box that is placed at the bottom
of a column and below the main diagonal of $T$.
\end{defn}

We recursively define the insertion procedure for symmetric tree-like tableaux
starting from a single box with the label $0$ associated with the insertion history 
$\mathbf{h}=\emptyset$.
Let $\mathbf{h}':=(h_1,\ldots,h_{N-1})$ and $\mathbf{h}:=(h_1,\ldots,h_{N})$
with $N=n+n'$ be insertion histories. 
Recall that $h_{i}$ corresponds to an edge $e$ of the tree $T(\lambda)$ and the edge $e$ of $T(\lambda)$
corresponds to a pair of an up step and a down step, or to an up step. 
We say $h_{i}$ is associated with $\lambda_{k}$ if $e$ is associated with the step $\lambda_{k}$.

Before proceeding to the definition of symmetric tree-like tableaux,
we introduce some definitions and a lemma with respect to a ballot
path $\lambda$.
\begin{defn}
\label{defn:edashed}
We define the integer sequence $\mathbf{e}':=(e'_{1},\ldots,e'_{m})$ for a ballot
path $\lambda$ of length $(n,n')$, where $m=2n+n'$ if the tree $T(\lambda)$ contains an edge without a dot and 
connected to the root, and $m=2n+n'+1$ otherwise.

We construct $\mathbf{e}'$ starting from $e'_1$ by the following algorithm.
\begin{enumerate}
\item Set $e'_{1}:=0$.
\item Set $e'_{p+1}:=e'_{p}+1$, 
\begin{enumerate}
\item if $\lambda_{p+1}=\lambda_{p}$, and both or neither of $h_{k}$ 
and $h_{k'}$ associated with $\lambda_{p}$ and $\lambda_{p+1}$ respectively are boxed for $1\le p\le n-1$,
or 
\item
if $(\lambda_{p},\lambda_{p+1})=(D,U)$, $h_{k}$ associated with $\lambda_{p}$ is not boxed, 
and $h_{k'}$ associated with $\lambda_{p+1}$ is boxed.
\end{enumerate}
\item 
We define $e'_{p+1}:=e'_{p}$ if $p$ satisfies the following three conditions:
\begin{enumerate}
\item 
Let $q<p$ be a maximum integer such that $\lambda_{q}=D$. 
If such down step $\lambda_{q}=D$ does not exist, then we set $q=0$. 
\item $\lambda_{i}=U$ and $h_{k}$ associated with $\lambda_{i}$ is boxed for $q+1\le i\le p$.
\item $\lambda_{p+1}=U$ and $h_{k}$ associated with $\lambda_{p+1}$ is not boxed.
\end{enumerate}
\item We define $e'_{p+1}=e'_{p}+2$ if $\lambda_{p}=U$ and $\lambda_{p+1}=D$ for $1\le p\le n-1$.
\end{enumerate}
When the tree $T(\lambda)$ does not contain an edge without a dot and connected to the root,
we increase each element of $\mathbf{e}'$ by one and redefine $\mathbf{e}'$ as a concatenation
of $(0)$ and increased $\mathbf{e}'$.
\end{defn}

\begin{lemma}
\label{lemma:edashed}
In $\mathbf{e}'$, there is no $p\le m-2$ such that 
$e'_{p}=e'_{p+1}=e'_{p+2}$.
\end{lemma}
\begin{proof}
Let $\lambda:=(\lambda_p,\lambda_{p+1},\lambda_{p+2})$ be a partial steps in $\lambda$.
Since $\lambda\in\{U,D\}^{3}$, we have eight partial paths.
Suppose that $e'_{p}=e'_{p+1}=e'_{p+2}$.
From (4) in Definition \ref{defn:edashed}, $\lambda$ does not contain the partial 
path $UD$. From (2a) in Definition \ref{defn:edashed}, $\lambda$ does not contain 
the partial path $DD$.
The remaining cases for $\lambda$ are 1) $DUU$ and 2) $UUU$.

Let $h_{k}$, $h_{k'}$ and $h_{k''}$ be an element of the insertion history
associated with $\lambda_{p}$, $\lambda_{p+1}$ and $\lambda_{p+2}$ respectively.

In case of $DUU$. From (2b) in Definition \ref{defn:edashed}, $h_{k'}$ associated with $\lambda_{p+1}$ is not boxed.
Since $\lambda_{p+2}$ is placed between $\lambda_{p+1}=U$ and $D$, $h_{k''}$ associated 
with $\lambda_{p+2}$ is not boxed.
Since $\lambda_{p+1}=\lambda_{p+2}$, and $h_{k'}$ and $h_{k''}$ are not boxed,
we have $e'_{p+2}=e'_{p+1}+1$ by (2a) in Definition \ref{defn:edashed}. This contradicts the assumption.
 
In case of $UUU$.
We have  three cases: 1) all $h_{i}$, $i\in\{k,k',k''\}$ are simultaneously either boxed or unboxed,
2) $h_{k}$ boxed and $h_{k'}$ and $h_{k''}$ unboxed, and 
3) $h_{k}$ and $h_{k'}$ boxed and $h_{k''}$ unboxed.
In all cases, we have $e'_{q+1}=e'_{q}+1$ for $q=p$ or $q=p+1$ from (2a) 
in Definition \ref{defn:edashed}. This contradicts the assumption.

Therefore, there is no $p$ satisfying $e'_{p}=e'_{p+1}=e'_{p+2}$.
\end{proof}

From Lemma \ref{lemma:edashed}, an integer appears in $\mathbf{e}'$ at most twice.
Then, we define an integer sequence $\mathbf{e}$ from $\mathbf{e}'$ 
as follows. 
Especially, $\mathbf{e}$ is an increasing sequence.

\begin{defn}
\label{defn:e}
The integer sequence $\mathbf{e}$ for $\lambda$ is defined from the integer sequence 
$\mathbf{e'}$ by deleting one of duplicated integers.
\end{defn}

\begin{remark}
Some remarks are in order.
\begin{enumerate}
\item 
Entries in $\mathbf{e}$ are all distinct by construction.
\item
When a tree consists of only $n'$ edges with dots, the length of the 
integer sequence $\mathbf{e}$ is $n'+1$.
\item 
If a tree $T$ contains only $n$ edges without a dot and $n$ leaves, that is,
the underlying lower boundary path for $T$ is a zig-zag path, then 
the length of $\mathbf{e}$ is $n+1$. 
\end{enumerate}
\end{remark}

\begin{example}
We consider the same natural label as in Example \ref{ex:symDTR}.
The ballot path for this natural label is 
\begin{align*}
UUDUDDUUDU,\qquad \leftrightarrow \qquad 
\tikzpic{-0.5}{[x=0.4cm,y=0.4cm]
\draw[very thick](0,0)--(2,2)--(3,1)--(4,2)--(6,0)--(8,2)--(9,1)--(10,2);
\draw(1,1)node{$-$}(5,1)node{$-$}(7,1)node{$-$};
}.
\end{align*}
We have $\mathbf{e}'=(0,1,3,3,5,6,7,7,9,10)$ and obtain 
$\mathbf{e}=(0,1,3,5,6,7,9,10)$.
\end{example}

Let $\mathbf{e}:=(e_1,\ldots,e_{m})$ be an integer sequence defined in Definition \ref{defn:e}.
\begin{defn}[Valid vertices]
\label{defn:vv}
Let $F$ be a symmetric Ferrers diagram with the half-perimeter $2m$,
where $m$ is the length of $\mathbf{e}$.
We put a label $e_{i}$ on the $i$-th and $(2m+1-i)$-the boundary edges.
When the difference of the labels on the $i$-th  (resp. $2m+1-i$-th) and $(i+1)$-the 
(resp. $2m-i$-th) boundary edges is two,
we put a circle on the vertex between the two edges and put the label $e_{i}+1$ on the 
vertex if it is placed below (resp. above) the diagonal line.
We call these circled vertices valid vertices.
\end{defn}

We are ready to introduce symmetric tree-like tableaux from insertion histories.
Recall that an insertion history is an integer sequence possibly with circles and 
with boxes.
\begin{defn}[Insertion procedure for symmetric tree-like tableaux]
\label{defn:IPTTab}
Let $T'$ be a symmetric tree-like tableau of size $N-1$ associated with $\mathbf{h}'$.
We construct a symmetric tree-like tableau $T$ of size $N$ associated with $\mathbf{h}$ by 
the following operations.
\begin{enumerate}
\item When $\mathbf{h}=\emptyset$, the symmetric tree-like tableau is defined as 
a single box labeled by zero.
\item We perform an insertion of two boxes with label $N$ on $T'$: 
\begin{enumerate}
\item If $h_{N}$ does not have a box, take the boundary edge $e$ or the valid vertex $v$ 
labeled by $h_{N}$.
\begin{enumerate}
\item If we have $e$ which is horizontal (resp. vertical), we perform a symmetric 
row (resp. column) insertion at $e$.
\item If we have $v$, we perform a symmetric diagonal insertion at $v$.
\end{enumerate}
We put a label $N$ at the two boxes which are added in the insertion process and 
southeast-most boxes below and above the diagonal line.
\item If $h_{N}$ is boxed, take the boundary vertex $v$ on the main diagonal. 
We perform a diagonal insertion at $v$.
We put a label $N$ on the box which is added as a boundary box on the main diagonal line.
\end{enumerate}
We denote by $\widetilde{T}$ the new symmetric tree-like tableau. 
\item We add a ribbon on $\widetilde{T}$ if $h_{N}$ does not have a box, and $h_{N-1}>h_{N}$
or $h_{N}$ has a circle.
\begin{enumerate}
\item In case of $h_{N}$ without a circle.
If there is the special point $s$ right to the box $b$ labeled by $N$ below the diagonal line,
we add a ribbon starting from the east edge of $b$ and ending at the south edge of $s$.
Similarly, we add a ribbon starting from the south edge of the box labeled by $N$ and ending at 
the east edge of the box labeled by $N-1$ above the diagonal line.
\item In case  of $h_{N}$ with a circle.
We add a ribbon starting from the east edge of the box labeled by $N$ below the main diagonal and 
ending at the south edge of the box labeled by $N$ above the main diagonal. 
\end{enumerate}
\item We denote by $T$ the new symmetric tree-like tableau obtained from $\widetilde{T}$.
\end{enumerate}
\end{defn}

\begin{prop}
\label{prop:BTperi}
Let $\mathbf{e}$ be an integer sequence constructed from $\mathbf{h}:=(h_1,\ldots,h_{n+n'})$ 
with $n+n'\ge1$, and $T$ be the symmetric tableau for $\mathbf{h}$.
Then, Definition \ref{defn:IPTTab} is well-defined.
In other words, the length of $\mathbf{e}$ coincides with the half of the half-perimeter of 
$T$.
\end{prop}
\begin{proof}
We prove the proposition by induction.
When $\mathbf{h}=\emptyset$, $T$ is a single box labeled by zero.
We have three cases when the length of $\mathbf{h}$ is one:
$\mathbf{h}=(0)$, $\mathbf{h}=(\ \circnum{0}\ )$, and 
$\mathbf{h}=(\boxed{0})$.
From Definition \ref{defn:IPTTab}, the symmetric tree-like tableaux 
for these three cases are
\begin{align*}
\begin{matrix}
\tikzpic{-0.5}{[x=0.4cm,y=0.4cm]
\draw(0,0)--(2,0)(0,-1)--(2,-1)(0,-2)--(1,-2)(0,0)--(0,-2)(1,0)--(1,-2)(2,0)--(2,-1);
\draw(0.5,-0.5)node{$0$}(0.5,-1.5)node{$1$}(1.5,-0.5)node{$1$};
} & 
\tikzpic{-0.5}{[x=0.4cm,y=0.4cm]
\draw(0,0)--(2,0)(0,-1)--(2,-1)(0,-2)--(2,-2)(0,0)--(0,-2)(1,0)--(1,-2)(2,0)--(2,-2);
\draw(0.5,-0.5)node{$0$}(0.5,-1.5)node{$1$}(1.5,-0.5)node{$1$};
} & 
\tikzpic{-0.5}{[x=0.4cm,y=0.4cm]
\draw(0,0)--(2,0)(0,-1)--(2,-1)(0,-2)--(2,-2)(0,0)--(0,-2)(1,0)--(1,-2)(2,0)--(2,-2);
\draw(0.5,-0.5)node{$0$}(1.5,-1.5)node{$1$};
}
\\[20pt]
\mathbf{h}=(0) & \mathbf{h}=(\ \circnum{0}\ ) & \mathbf{h}=(\boxed{0}) \\
\mathbf{e}=(0,2) & \mathbf{e}=(0,2) & \mathbf{e}=(0,1)
\end{matrix}
\end{align*}
The theorem is true for $n+n'=1$.
Suppose that the theorem is true up to $n+n'=N-1$.
We have three cases: 1) a symmetric row or column insertion, 2) a symmetric diagonal insertion,
and 3) a diagonal insertion on the main diagonal.
In the first and third cases, the half-perimeter of $T$ is increased by two and the 
length of $\mathbf{e}$ is increased by one.
In the second case, the half-perimeter of $T$ is increased by four and the length of 
$\mathbf{e}$ is increased by two.
In all cases, the half of the half-perimeter of $T$ is equal to the length of $\mathbf{e}$.
\end{proof}

From Proposition \ref{prop:BTperi}, the half of the half-perimeter of a symmetric 
tree-like tableau $T$ is equal to the length of the integer sequence $\mathbf{e}$.
We enumerate the boundary edges below the main diagonal in $T$ 
from southwest to northeast by $1,2\ldots,m$, where $m$ is the length of $\mathbf{e}$.
Then, we put a circle on the valid vertices along the boundary edges in $T$ 
as in Definition \ref{defn:vv}.
Similarly, we put circles on the valid vertices above the main diagonal.

\begin{example}
We consider the same natural label as Example \ref{ex:symDTR}.
The insertion history for this label is $\mathbf{h}=(\boxed{0},0,\ \circnum{3}\ ,\ \circnum{1}\ ,\boxed{7},3)$.
The tree-like tableaux for $\mathbf{h}$ are shown in Figure \ref{fig:TTab}.
\begin{figure}[ht]
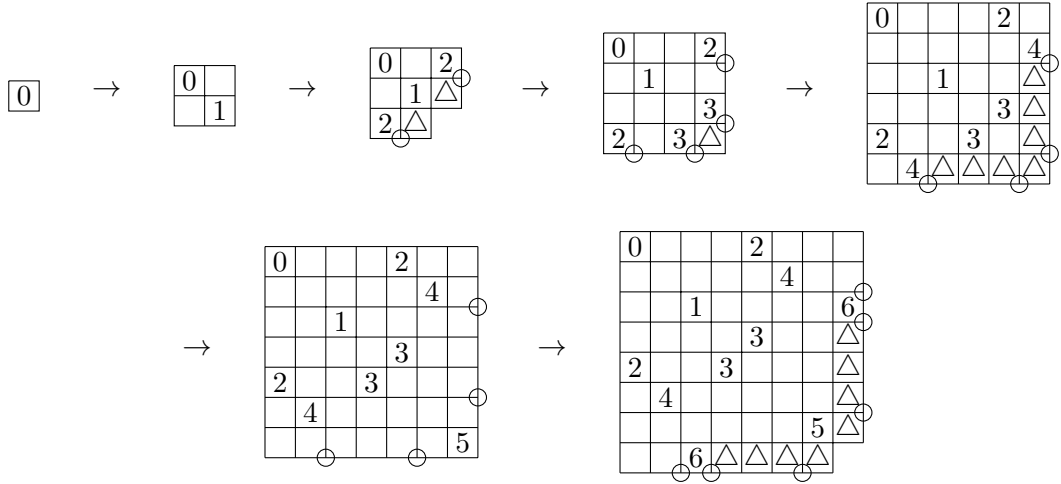

\tikzpic{-0.5}{[x=0.4cm,y=0.4cm]
\draw(0,0)--(1,0)--(1,-1)--(0,-1)--(0,0);
\draw(0.5,-0.5)node{$0$};
}
\quad$\rightarrow$\quad
\tikzpic{-0.5}{[x=0.4cm,y=0.4cm]
\draw(0,0)--(2,0)--(2,-2)--(0,-2)--(0,0);
\draw(0,-1)--(2,-1)(1,0)--(1,-2);
\draw(0.5,-0.5)node{$0$}(1.5,-1.5)node{$1$};
}
\quad$\rightarrow$\quad
\tikzpic{-0.5}{[x=0.4cm,y=0.4cm]
\draw(0,0)--(3,0)--(3,-2)--(2,-2)--(2,-3)--(0,-3)--(0,0);
\draw(0,-1)--(3,-1)(0,-2)--(2,-2)(1,0)--(1,-3)(2,0)--(2,-2);
\draw(0.5,-0.5)node{$0$}(1.5,-1.5)node{$1$}(0.5,-2.5)node{$2$}(2.5,-0.5)node{$2$};
\draw(1,-3)node[circle,inner sep=0.8mm,draw]{}(3,-1)node[circle,inner sep=0.8mm,draw]{};
\draw(1.5,-2.5)node{$\triangle$}(2.5,-1.5)node{$\triangle$};
}
\quad$\rightarrow$\quad
\tikzpic{-0.5}{[x=0.4cm,y=0.4cm]
\draw(0,0)--(4,0)--(4,-4)--(0,-4)--(0,0);
\draw(0,-1)--(4,-1)(0,-2)--(4,-2)(0,-3)--(4,-3);
\draw(1,0)--(1,-4)(2,0)--(2,-4)(3,0)--(3,-4);
\draw(0.5,-0.5)node{$0$}(1.5,-1.5)node{$1$}(0.5,-3.5)node{$2$}(3.5,-0.5)node{$2$}
	(2.5,-3.5)node{$3$}(3.5,-2.5)node{$3$};
\draw(1,-4)node[circle,inner sep=0.8mm,draw]{}(4,-1)node[circle,inner sep=0.8mm,draw]{};
\draw(3,-4)node[circle,inner sep=0.8mm,draw]{}(4,-3)node[circle,inner sep=0.8mm,draw]{};
\draw(3.5,-3.5)node{$\triangle$};
}
\quad$\rightarrow$\quad
\tikzpic{-0.5}{[x=0.4cm,y=0.4cm]
\foreach \x in{0,1,2,3,4,5,6}{\draw(0,-\x)--(6,-\x)(\x,0)--(\x,-6);}
\draw(0.5,-0.5)node{$0$}(2.5,-2.5)node{$1$}(0.5,-4.5)node{$2$}(4.5,-0.5)node{$2$}
	(3.5,-4.5)node{$3$}(4.5,-3.5)node{$3$}(1.5,-5.5)node{$4$}(5.5,-1.5)node{$4$};
\draw(2,-6)node[circle,inner sep=0.8mm,draw]{}(6,-2)node[circle,inner sep=0.8mm,draw]{};
\draw(5,-6)node[circle,inner sep=0.8mm,draw]{}(6,-5)node[circle,inner sep=0.8mm,draw]{};
\draw(2.5,-5.5)node{$\triangle$}(3.5,-5.5)node{$\triangle$}(5.5,-2.5)node{$\triangle$}
     (5.5,-3.5)node{$\triangle$}(4.5,-5.5)node{$\triangle$}(5.5,-4.5)node{$\triangle$}
     (5.5,-5.5)node{$\triangle$};
}\\[12pt]
$\rightarrow$\quad
\tikzpic{-0.5}{[x=0.4cm,y=0.4cm]
\foreach \x in {0,-1,-2,-3,-4,-5,-6,-7} {\draw(0,\x)--(7,\x)(-\x,0)--(-\x,-7);}
\draw(0.5,-0.5)node{$0$}(2.5,-2.5)node{$1$}(0.5,-4.5)node{$2$}(4.5,-0.5)node{$2$}
	(3.5,-4.5)node{$3$}(4.5,-3.5)node{$3$}(1.5,-5.5)node{$4$}(5.5,-1.5)node{$4$}
	(6.5,-6.5)node{$5$};
\draw(2,-7)node[circle,inner sep=0.8mm,draw]{}(7,-2)node[circle,inner sep=0.8mm,draw]{};
\draw(5,-7)node[circle,inner sep=0.8mm,draw]{}(7,-5)node[circle,inner sep=0.8mm,draw]{};
}
\quad$\rightarrow$\quad
\tikzpic{-0.5}{[x=0.4cm,y=0.4cm]
\foreach \x in {0,-1,-2,-3,-4,-5,-6,-7} {\draw(0,\x)--(8,\x)(-\x,0)--(-\x,-8);}
\draw(8,0)--(8,-7)(0,-8)--(7,-8);
\draw(0.5,-0.5)node{$0$}(2.5,-2.5)node{$1$}(0.5,-4.5)node{$2$}(4.5,-0.5)node{$2$}
	(3.5,-4.5)node{$3$}(4.5,-3.5)node{$3$}(1.5,-5.5)node{$4$}(5.5,-1.5)node{$4$}
	(6.5,-6.5)node{$5$}(2.5,-7.5)node{$6$}(7.5,-2.5)node{$6$};
\draw(2,-8)node[circle,inner sep=0.8mm,draw]{}(8,-2)node[circle,inner sep=0.8mm,draw]{};
\draw(3,-8)node[circle,inner sep=0.8mm,draw]{}(8,-3)node[circle,inner sep=0.8mm,draw]{};
\draw(6,-8)node[circle,inner sep=0.8mm,draw]{}(8,-6)node[circle,inner sep=0.8mm,draw]{};
\draw(3.5,-7.5)node{$\triangle$}(4.5,-7.5)node{$\triangle$}(5.5,-7.5)node{$\triangle$}(6.5,-7.5)node{$\triangle$}
     (7.5,-3.5)node{$\triangle$}(7.5,-4.5)node{$\triangle$}(7.5,-5.5)node{$\triangle$}(7.5,-6.5)node{$\triangle$};
}
\caption{An example of growth of tree-like tableaux.}
\label{fig:TTab}
\end{figure}
The boxes with $\triangle$ are the ribbon 
added during the insertion procedure.
We put a circle on valid vertices on the boundary edges of a symmetric tree-like 
tableau.
\end{example}

The unique tree-like tableau of size one is a single box labeled by zero.
We have three tree-like tableaux with labels in $[0,1]$ (see Proof of Proposition \ref{prop:BTperi}).
We list up all 19 symmetric tree-like tableaux with labels in $[0,2]$ in Figure \ref{fig:BT2}.
Each tree-like tableau in Figure \ref{fig:BT2} corresponds to the ballot tableau
in Figure \ref{fig:BTn2}.
\begin{figure}[ht]
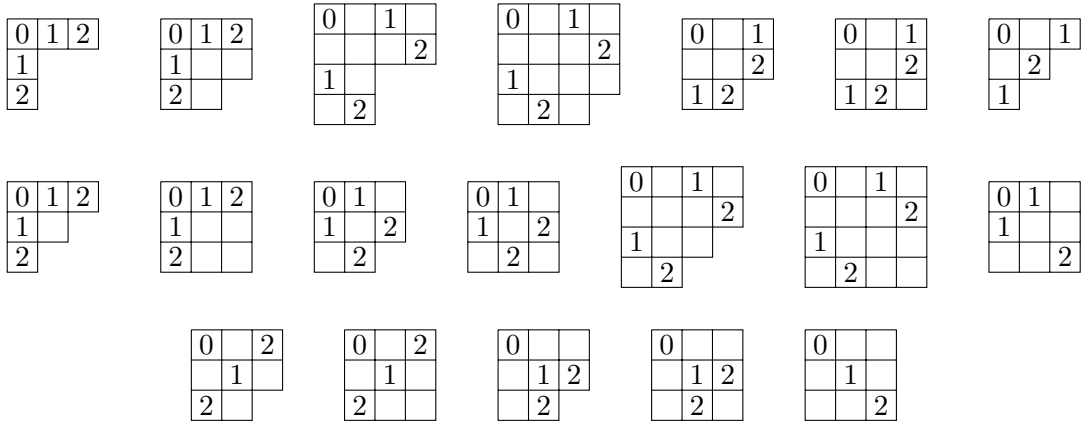

\tikzpic{-0.5}{[x=0.4cm,y=0.4cm]
\draw(0,0)--(3,0)--(3,-1)--(0,-1)(0,0)--(0,-3)--(1,-3)--(1,0)(0,-2)--(1,-2)(2,0)--(2,-1);
\draw(0.5,-0.5)node{$0$}(0.5,-1.5)node{$1$}(0.5,-2.5)node{$2$};
\draw(1.5,-0.5)node{$1$}(2.5,-0.5)node{$2$};
}\quad
\tikzpic{-0.5}{[x=0.4cm,y=0.4cm]
\draw(0,0)--(3,0)--(3,-2)--(0,-2)(0,0)--(0,-3)--(2,-3)--(2,0)(0,-1)--(3,-1)(1,0)--(1,-3);
\draw(0.5,-0.5)node{$0$}(0.5,-1.5)node{$1$}(0.5,-2.5)node{$2$};
\draw(1.5,-0.5)node{$1$}(2.5,-0.5)node{$2$};
}\quad
\tikzpic{-0.5}{[x=0.4cm,y=0.4cm]
\draw(0,0)--(4,0)--(4,-2)--(0,-2)(0,0)--(0,-4)--(2,-4)--(2,0)(3,0)--(3,-2)(1,0)--(1,-4)
(0,-1)--(4,-1)(0,-3)--(2,-3);
\draw(0.5,-0.5)node{$0$}(2.5,-0.5)node{$1$}(0.5,-2.5)node{$1$}(3.5,-1.5)node{$2$}(1.5,-3.5)node{$2$};
}\quad
\tikzpic{-0.5}{[x=0.4cm,y=0.4cm]
\draw(0,0)--(4,0)--(4,-3)--(0,-3)(0,0)--(0,-4)--(3,-4)--(3,0)(0,-1)--(4,-1)(0,-2)--(4,-2)
(1,0)--(1,-4)(2,0)--(2,-4);
\draw(0.5,-0.5)node{$0$}(2.5,-0.5)node{$1$}(0.5,-2.5)node{$1$}(3.5,-1.5)node{$2$}(1.5,-3.5)node{$2$};
}\quad
\tikzpic{-0.5}{[x=0.4cm,y=0.4cm]
\draw(0,0)--(3,0)--(3,-2)--(0,-2)(0,0)--(0,-3)--(2,-3)--(2,0)(0,-1)--(3,-1)(1,0)--(1,-3);
\draw(0.5,-0.5)node{$0$}(2.5,-0.5)node{$1$}(0.5,-2.5)node{$1$}(1.5,-2.5)node{$2$}(2.5,-1.5)node{$2$};
}\quad
\tikzpic{-0.5}{[x=0.4cm,y=0.4cm]
\draw(0,0)--(3,0)--(3,-3)--(0,-3)--(0,0)(0,-2)--(3,-2)(2,0)--(2,-3)(0,-1)--(3,-1)(1,0)--(1,-3);
\draw(0.5,-0.5)node{$0$}(2.5,-0.5)node{$1$}(0.5,-2.5)node{$1$}(1.5,-2.5)node{$2$}(2.5,-1.5)node{$2$};
}\quad
\tikzpic{-0.5}{[x=0.4cm,y=0.4cm]
\draw(0,0)--(3,0)--(3,-1)--(0,-1)(2,0)--(2,-2)--(0,-2)(1,0)--(1,-3)--(0,-3)--(0,0);
\draw(0.5,-0.5)node{$0$}(2.5,-0.5)node{$1$}(0.5,-2.5)node{$1$}(1.5,-1.5)node{$2$};
}\\[12pt]
\tikzpic{-0.5}{[x=0.4cm,y=0.4cm]
\draw(0,0)--(3,0)--(3,-1)--(0,-1)(2,0)--(2,-2)--(0,-2)(1,0)--(1,-3)--(0,-3)--(0,0);
\draw(0.5,-0.5)node{$0$}(1.5,-0.5)node{$1$}(0.5,-1.5)node{$1$}(0.5,-2.5)node{$2$}(2.5,-0.5)node{$2$};
}\quad
\tikzpic{-0.5}{[x=0.4cm,y=0.4cm]
\draw(0,0)--(3,0)--(3,-3)--(0,-3)--(0,0)(1,0)--(1,-3)(2,0)--(2,-3)(0,-1)--(3,-1)(0,-2)--(3,-2);
\draw(0.5,-0.5)node{$0$}(1.5,-0.5)node{$1$}(0.5,-1.5)node{$1$}(0.5,-2.5)node{$2$}(2.5,-0.5)node{$2$};
}\quad
\tikzpic{-0.5}{[x=0.4cm,y=0.4cm]
\draw(0,0)--(3,0)--(3,-2)--(0,-2)(0,0)--(0,-3)--(2,-3)--(2,0)(1,0)--(1,-3)(0,-1)--(3,-1);
\draw(0.5,-0.5)node{$0$}(0.5,-1.5)node{$1$}(1.5,-0.5)node{$1$}(2.5,-1.5)node{$2$}(1.5,-2.5)node{$2$};
}\quad
\tikzpic{-0.5}{[x=0.4cm,y=0.4cm]
\draw(0,0)--(3,0)--(3,-3)--(0,-3)--(0,0)(1,0)--(1,-3)(2,0)--(2,-3)(0,-1)--(3,-1)(0,-2)--(3,-2);
\draw(0.5,-0.5)node{$0$}(0.5,-1.5)node{$1$}(1.5,-0.5)node{$1$}(2.5,-1.5)node{$2$}(1.5,-2.5)node{$2$};
}\quad
\tikzpic{-0.5}{[x=0.4cm,y=0.4cm]
\draw(0,0)--(4,0)--(4,-2)--(0,-2)(0,-1)--(4,-1)(0,0)--(0,-4)--(2,-4)--(2,0)(1,0)--(1,-4)
(3,0)--(3,-2)(0,-3)--(2,-3)(3,-2)--(3,-3)--(2,-3);
\draw(0.5,-0.5)node{$0$}(2.5,-0.5)node{$1$}(0.5,-2.5)node{$1$}(3.5,-1.5)node{$2$}(1.5,-3.5)node{$2$};
}\quad
\tikzpic{-0.5}{[x=0.4cm,y=0.4cm]
\draw(0,0)--(4,0)--(4,-4)--(0,-4)--(0,0)(0,-1)--(4,-1)(0,-2)--(4,-2)(0,-3)--(4,-3)
(1,0)--(1,-4)(2,0)--(2,-4)(3,0)--(3,-4);
\draw(0.5,-0.5)node{$0$}(2.5,-0.5)node{$1$}(0.5,-2.5)node{$1$}(3.5,-1.5)node{$2$}(1.5,-3.5)node{$2$};
}\quad
\tikzpic{-0.5}{[x=0.4cm,y=0.4cm]
\draw(0,0)--(3,0)--(3,-3)--(0,-3)--(0,0)(1,0)--(1,-3)(2,0)--(2,-3)(0,-1)--(3,-1)(0,-2)--(3,-2);
\draw(0.5,-0.5)node{$0$}(1.5,-0.5)node{$1$}(0.5,-1.5)node{$1$}(2.5,-2.5)node{$2$};
}\\[12pt]
\tikzpic{-0.5}{[x=0.4cm,y=0.4cm]
\draw(0,0)--(3,0)--(3,-2)--(0,-2)(0,0)--(0,-3)--(2,-3)--(2,0)(1,0)--(1,-3)(0,-1)--(3,-1);
\draw(0.5,-0.5)node{$0$}(1.5,-1.5)node{$1$}(0.5,-2.5)node{$2$}(2.5,-0.5)node{$2$};
}\quad
\tikzpic{-0.5}{[x=0.4cm,y=0.4cm]
\draw(0,0)--(3,0)--(3,-3)--(0,-3)--(0,0)(1,0)--(1,-3)(2,0)--(2,-3)(0,-1)--(3,-1)(0,-2)--(3,-2);
\draw(0.5,-0.5)node{$0$}(1.5,-1.5)node{$1$}(0.5,-2.5)node{$2$}(2.5,-0.5)node{$2$};
}\quad
\tikzpic{-0.5}{[x=0.4cm,y=0.4cm]
\draw(0,0)--(3,0)--(3,-2)--(0,-2)(0,0)--(0,-3)--(2,-3)--(2,0)(1,0)--(1,-3)(0,-1)--(3,-1);
\draw(0.5,-0.5)node{$0$}(1.5,-1.5)node{$1$}(1.5,-2.5)node{$2$}(2.5,-1.5)node{$2$};
}\quad
\tikzpic{-0.5}{[x=0.4cm,y=0.4cm]
\draw(0,0)--(3,0)--(3,-3)--(0,-3)--(0,0)(1,0)--(1,-3)(2,0)--(2,-3)(0,-1)--(3,-1)(0,-2)--(3,-2);
\draw(0.5,-0.5)node{$0$}(1.5,-1.5)node{$1$}(1.5,-2.5)node{$2$}(2.5,-1.5)node{$2$};
}\quad
\tikzpic{-0.5}{[x=0.4cm,y=0.4cm]
\draw(0,0)--(3,0)--(3,-3)--(0,-3)--(0,0)(1,0)--(1,-3)(2,0)--(2,-3)(0,-1)--(3,-1)(0,-2)--(3,-2);
\draw(0.5,-0.5)node{$0$}(1.5,-1.5)node{$1$}(2.5,-2.5)node{$2$};
}
\caption{19 symmetric tree-like tableaux with labels in $[0,2]$ for symmetric Dyck tilings}
\label{fig:BT2}
\end{figure}

Symmetric tree-like tableaux of size $2n+1$ studied in \cite{ABN11} form a subset 
in our symmetric tree-like tableaux with general lower ballot path. 
They correspond to the trees which have no dots, have $n$ leaves and possibly 
have circles on labels.
The condition that a tree has $n$ leaves means that the lower path of 
a symmetric Dyck tiling is a zig-zag path, {\it i.e.,} $(UD)^{n}$.

In terms of tree-like tableaux, the above conditions are equivalent to 
the following conditions on symmetric tree-like tableaux:
\begin{enumerate}
\item we have two boxes labeled by $i$ for $i\in[1,n]$.
\item the half-perimeter of a tableau is $2(n+1)$.
\end{enumerate}

Let $B$ be a symmetric tree-like tableau of half-perimeter $2p$ with $p\ge n+1$.
Then, the number of diagonal insertions, which are not diagonal insertions on the 
main diagonal in $B$, is $p-(n+1)$.
The number of the dotted edges in a tree corresponding to $B$ 
is the number of labels which appear only once in $B$.
Further, if and only if a label appears once in $B$, this label is on a box 
on the main diagonal.

\subsection{Tree-like tableaux of shifted shapes}
\label{sec:tltss}
Since a tree-like tableau $T$ for a symmetric Dyck tiling is symmetric
along the diagonal line, we obtain a tree-like tableau of shifted 
shape by cutting $T$ along the diagonal line.

Each non-negative integers appears exactly once in a tree-like tableau 
of a shifted shape. 
When the labels of a tree-like tableau of a shifted shape are in $[0,n]$, 
we say that the tree-like tableau of a shifted shape is of size $n$.
\begin{example}
We consider the same example as Example \ref{fig:TTab}.
Figure \ref{fig:TTabshifted} shows a tree-like tableau of a shifted shape.
\begin{figure}[ht]
\tikzpic{-0.5}{[x=0.4cm,y=0.4cm]
\foreach \x in{0,-1,-2,-3,-4,-5,-6} {\draw(0,\x)--(-\x+1,\x);}
\draw(0,0)--(0,-8)(0,-7)--(7,-7);
\foreach \x in{1,2,3,4,5,6}{\draw(\x,-\x+1)--(\x,-8);}
\draw(0,-8)--(7,-8)(7,-6)--(7,-8);
\draw(0.5,-0.5)node{$0$}(2.5,-2.5)node{$1$}(0.5,-4.5)node{$2$}
	(3.5,-4.5)node{$3$}(1.5,-5.5)node{$4$}
	(6.5,-6.5)node{$5$}(2.5,-7.5)node{$6$};
}
\caption{A tree-like tableau of a shifted shape}
\label{fig:TTabshifted}
\end{figure}
\end{example}

\begin{remark}
Let $L$ be a natural label such that any edge is connected to both the root and a leaf and there 
is no edges with dots, {\it i.e.}, the lower boundary path is a zig-zag Dyck path.
In this case, the tree-like tableau of a shifted shape for $L$ constructed in this paper 
coincides with the symmetric tree-like tableau for $L$ studied in \cite{ABN11}.
\end{remark}

\subsection{Enumeration of tree-like tableaux of shifted shapes}
In this subsection, we study the enumeration of tree-like tableaux of 
shifted shapes.
Since there is a one-to-one correspondence between insertion histories
and tree-like tableaux of shifted shapes, this enumeration gives also 
the total number of symmetric Dyck tilings with a general lower boundary
ballot path and equivalently the total number of ballot tableaux (see 
Section \ref{sec:bttts} for a bijection between ballot tableaux and 
tree-like tableaux of shifted shapes).

Let $a_{n}$ be a sequence of integers (A001517 in \cite{Slo}) satisfying
\begin{align*}
A_{n}=(4n-6)A_{n-1}+A_{n-2},
\end{align*}
with initial conditions $A_{1}=1$ and $A_{2}=3$.
First few values of $A_{n\ge1}$ is $1, 3, 19, 193, 2721,\ldots$.
The values $\{A_{n}\}$ are the Bessel polynomials evaluated at $2$.

\begin{theorem}
\label{thrm:tlt}
The number of tree-like tableaux of shifted shapes of size $n$ is 
given by $A_{n+1}$.
\end{theorem}

Before proceeding to the proof of Theorem \ref{thrm:tlt}, 
we introduce several lemmas and propositions.

Let $a_{n,l,m}$ be the total number of tree-like tableaux of 
shifted shape of size $n-1$ which have $l+m$ diagonal points 
and especially $m$ diagonal points on the main diagonal.
Thus the value $l$ counts the number of diagonal points which
are not on the main diagonal.

\begin{lemma}
The value $a_{n,l,m}$ satisfies 
\begin{align}
\label{eqn:rec}
a_{n,l,m}=2(n-1+l)a_{n-1,l,m}+2(n-1-l-m)a_{n-1,l-1,m}+a_{n-1,l,m-1},
\end{align}
with the initial conditions 
\begin{align*}
a_{1,l,m}=\begin{cases}
1, & (l,m)=(0,0), \\
0, & (l,m)\neq (0,0).
\end{cases}
\end{align*}
\end{lemma}
\begin{proof}
Let $T_{n,l,m}$ be a tree-like tableaux of shifted shapes such 
that the size is $n-1$, the number of diagonal points is $l+m$,
and the number of diagonal points on the main diagonal is $m$.

First, we count the number of ways to produce $T_{n,l,m}$'s  
from $T_{n-1,l,m}$.
Since the number of diagonal points are the same, we have 
$2(n-1+l)$ ways to produce a tree-like tableau $T_{n,l,m}$.
Recall that a tree-like tableau of a shifted shape is one-to-one 
to a tree-like tableau for a symmetric Dyck tiling.
In the insertion algorithm for a tree-like tableau for a symmetric 
Dyck tiling, we have two choices when two points labeled by $n$ are 
inserted in a tableau. The two choices are whether we add a ribbon 
connecting the two points labeled by $n$ 
(see Definition \ref{defn:IPTTab}). 
Thus, the factor $2$ in $2(n-1+l)$ comes from these two choices.
The factor $(n-1+l)$ comes from the number of insertion points 
which does not correspond to a diagonal insertion.
The insertion of a diagonal insertion increases the number of 
insertion points by one. Since $T_{n-1,l,m}$ has $l$ diagonal 
points (which are not on the main diagonal), we have
a factor $(n-1+l)$.

Secondly, we have only one way to produce $T_{n,l,m}$ from 
$T_{n-1,l,m-1}$. 
This is obvious since we have a unique way to 
perform a diagonal insertion on the main diagonal in $T_{n-1,l,m}$.

Thirdly, we count the number of ways to produce $T_{n,l,m}$ 
from $T_{n-1,l-1,m}$.
When $(l,m)=(1,0)$, we have $n-2$ ways to perform a diagonal 
insertion since the size of $T_{n-1,l-1,m}$ is $n-2$.
The increment of $l$ or $m$ by one decrease the number of ways 
of diagonal insertions by one.
Thus, we have $n-2-(l-1)-m=n-1-l-m$ ways to perform a diagonal 
insertion. 
The factor $2$ comes from the same fact as the first case.

Combining these observations together, we obtain Eqn. (\ref{eqn:rec}).
\end{proof}

Let $c_{n}(x,y)$ be the formal power series given by
\begin{align*}
c_{n}(x,y):=\sum_{l,m\ge0}x^{l}y^{m}a_{n,l,m}.
\end{align*}
The formal power series $c_{n}(x,y)$ satisfies the following partial differential 
equation.
\begin{lemma}
$c_{n}(x,y)$ satisfies 
\begin{align*}
c_{n}(x,y)=2((n-1)(1+x)-x+x\partial_{x}-x^2\partial_{x})c_{n-1}(x,y)
+y(1-2x\partial_{y})c_{n-1}(x,y)
\end{align*}
\end{lemma}
\begin{proof}
We take a sum (with respect to $l$ and $m$) of Eqn. (\ref{eqn:rec}) 
multiplied by $x^{l}y^{m}$.
For example, the term $\sum_{l,m}2(n-1+l)a_{n-1,l,m}$
comes from $2(n-1+x\partial_{x})c_{n-1}(x,y)$.
One can get other terms by a similar straightforward computation.
\end{proof}

By setting $x=1$ in the above equation, we have 
\begin{align}
\label{eqn:pdey}
c_{n}(1,y)=(4n-6+y)c_{n-1}(1,y)-2y\partial_{y}c_{n-1}(1,y).
\end{align}
Note that we have no terms involving a partial derivative 
with respect to $x$.

Let $T(n,k)$ be the integer sequence  (A113025 in \cite{Slo}) 
given by 
\begin{align*}
T(n,k):=\genfrac{}{}{0.8pt}{}{(n+k)!}{(n-k)!k!}.
\end{align*}
We first define a formal power series from $T(n,k)$, and 
show this formal power series solves the partial differential equation (\ref{eqn:pdey}).
\begin{lemma}
Define 
\begin{align}
\label{eqn:formulay}
C_{n}(y):=\sum_{0\le k\le n}y^{n-k-1}T(n-1,k).
\end{align}
The formal power series $C_{n}(y)$ is the solution of the equation (\ref{eqn:pdey}).
\end{lemma}
\begin{proof}
One can easily show that (\ref{eqn:formulay}) is a solution of 
the partial differential equation (\ref{eqn:pdey}) by substituting 
Eqn. (\ref{eqn:formulay}) into (\ref{eqn:pdey}) and 
comparing the initial conditions.
\end{proof}

\begin{lemma}
The formal power series $C_{n}(y)$ satisfies the following 
recurrence relation:
\label{lemma:py}
\begin{align*}
C_{n}(y)-2\partial_{y}C_{n}(y)=y C_{n-1}(y)
\end{align*}
\end{lemma}
\begin{proof}
One can easily show the statement by a straightforward computation. 
\end{proof}

Combining the partial differential equation (\ref{eqn:pdey}) with 
Lemma \ref{lemma:py},
we obtain
\begin{prop}
\label{prop:inhorec}
The formal power series $C_{n}(y)$ satisfies the following 
recurrence relation:
\begin{align}
\label{eqn:inhomrec}
C_{n}(y)=(4n-6)C_{n-1}(y)+y^{2}C_{n-2}(y).
\end{align}
\end{prop}

\begin{proof}[Proof of Theorem \ref{thrm:tlt}]
The total number of tree-like tableaux of shifted shapes  
of size $n-1$ is given by $\sum_{l,m}a_{n,l,m}$.
In terms of $c_{n}(x,y)$, this number is equal to the 
specialization $c_{n}(1,1)$, which is equal to $C_{n}(1)$.
From Proposition \ref{prop:inhorec}, we have, by specializing 
$y$ to $1$, the recurrence relation
\begin{align*}
c_{n}(1,1)=(4n-6)c_{n-1}(1,1)+c_{n-2}(1,1),
\end{align*}
which is the defining relation of $A_{n}$. 
The sequence $A_{n}$ and $c_{n}(1,1)$ have the same initial 
conditions, therefore, we conclude $A_{n}=c_{n}(1,1)$.
\end{proof}

\subsection{Ribbons in a tree-like tableau of a shifted shape}
In this subsection, we enumerate the number of boxes added as a ribbon
in the insertion process of a tree-like tableau of a shifted shape.
We make use of the structure of a tree associated with a lower boundary 
path $\lambda$.

Let $L$ be a natural label of the tree $\mathrm{Tree}(\lambda)$,
and $\mathbf{h}:=(h_1,\ldots,h_{n})$ be its insertion history.
Recall that we have three types of integer values $h_{i}$ in $\mathrm{h}$:
1) $h_{i}$ is boxed, 2) $h_{i}$ is circled, and 3) $h_i$ is neither 
boxed nor circled. 
From the definition of the insertion procedure of tree-like tableaux of shifted shapes,
we add a ribbon to a tree-like tableau of a shifted shape in the following 
two cases: 1) when $h_{n-1}>h_{n}$ if $h_{n}$ is not circled, or 
2) when $h_{n}$ is circled.

Let $R_{n}$ be the number of boxes in a ribbon added to a tableau of 
a shifted shape at the $n$-th step.
We give the value $R_{n}$ in terms of an insertion history and a natural label of 
a tree for a ballot path.

Let $T$ be a tree constructed from the insertion history $\mathbf{h}$ 
and $E(i)$ with $i\in[1,n]$ be the edge labeled by $i$ in $L$. 

\begin{prop}
\label{prop:Rn}
The value $R_{n}$ satisfies the following properties.
\begin{enumerate}
\item If $h_{n}$ is boxed, then $R_{n}=0$. 
\item If $h_{n-1}=h_{n}$ with $h_{n-1}$ boxed, then 
\begin{align*}
R_{n}=
\begin{cases}
1, & h_{n} \text{ does not have a circle}, \\
2, & h_{n} \text{ has a circle}. 
\end{cases}
\end{align*}

\item If $h_{n}$ is not circled with $h_{n-1}>h_{n}$, then 
\begin{align*}
R_{n}=h_{n-1}-h_{n}
-\#\{ k<n | E(n)\rightarrow E(k) \rightarrow E(n-1), E(k)\in \mathcal{E}_{l}\}
+
\begin{cases}
0, & h_{n-1} \text{ unboxed}, \\
1, & h_{n-1} \text{ boxed}, \\
\end{cases}
\end{align*}
where $\mathcal{E}_{l}$ is the set of edges connected to leaves in $L$.

\item Suppose that $h_{n}$ is circled. 
Then, we have 
\begin{align*}
R_{n}=&1-h_{n}+2\cdot\#\{k<n| E(k)\notin\mathcal{E}_{\bullet}\}
+\#\{k<n|E(k)\in\mathcal{E}_{\bullet}\} \\
&-\#\{k<n | E(n)\rightarrow E(k), E(k)\notin\mathcal{E}_{\bullet}\},
\end{align*}
where $\mathcal{E}_{\bullet}$ is the set of dotted edges in the tree.
\item Otherwise, $R_{n}=0$.
\end{enumerate}
\end{prop}
\begin{proof}
(1) and (2) are clear from the definition of insertion procedure of tree-like  
tableaux of shifted shapes.

(3)
We define the value $m$ as $m:=h_{n-1}-h_{n}-R_{n}$. 

The value $m$ may be a negative integer since $h_{n}$ is not circled, 
and $h_{n-1}-h_{n}-R_{n}$ can be a negative integer (see (2)). 
When $h_{n}$ is boxed, the value $m$ is equal to the number of valid vertices 
between the $h_{n-1}$-th and the $h_{n}$-th positions in the tree-like 
tableau of a shifted shape minus one.
Similarly, when $h_{n}$ is not boxed, $m$ is equal to the number of valid vertices 
between the $h_{n-1}$-th and the $h_{n}$-th positions in the tree-like 
tableau.
From the definition of $\mathbf{e}$ for a ballot path $\lambda$, we have a valid 
vertex if we have a partial $UD$ path in $\lambda$. 
In terms of trees, a partial path $UD$ corresponds to an edge connected to 
a leaf of the tree. 
Summarizing the above observations, the value $m$ is expressed 
in terms of the number of valid vertices.
We have 
\begin{align*}
m=
\#\{ k<n | E(n)\rightarrow E(k) \rightarrow E(n-1), E(k)\in \mathcal{E}_{l}\}
-
\begin{cases}
0, & h_{n-1} \text{ unboxed}, \\
1, & h_{n-1} \text{ boxed}. \\
\end{cases}
\end{align*}

(4) Since $h_{n}$ is circled, we have a ribbon connecting two same labels 
$n$ in the tree-like tableau of a non-shifted shape.
By keeping the value $R_{n}$ the same, we transform the tree-like tableau of a shifted shape
of size $n$ into another tree-like tableau of a shifted shape of size $n+1$ and apply 
the formula in (3).
This is realized in terms of the tree-like tableau of a shifted shape as follows.

Let $T'$ be a tree-like tableau of $\mathbf{h}':=(h_1,\ldots,h_{n-1})$.
We insert the box labeled by $n$ on the main diagonal and then 
insert the box labeled by $n+1$ such that its insertion history is equal to $h_{n}$. 
The new insertion history satisfies $h_{n}>h_{n+1}$ and we connect 
the box labeled by $n+1$ and the box labeled by $n$ by a ribbon.
The entry of the insertion history $h_{n+1}$ has no longer a circle.
Since the box labeled by $n$ is a diagonal box, $h_{n}$ is equal to 
$2\#\{k<n | E(k)\notin\mathcal{E}_{\bullet} \}+\#\{k<n | E(k)\in\mathcal{E}_{\bullet} \}$.
Then, the number of boxes in the ribbon added to the new tree-like tableau of a shifted shape 
is the same as the one added to the old tree-like tableau.
We apply the case (3) to the new tree-like tableau. 

(5) The remaining case is $h_{n-1}\le h_{n}$ and $h_{n}$ is not circled.
We have no ribbon starting from the box labeled by $n$, which implies 
$R_{n}=0$.
\end{proof}

\subsection{Ballot tableaux and tree-like tableaux of shifted shapes}
\label{sec:bttts}
In this subsection, we study a bijection between ballot tableaux and 
tree-like tableaux of shifted shapes.
Obviously, both objects are bijective to the natural labels of a tree. 
We show that the bijection connect the properties of these two objects such as 
shadow boxes in ballot tableaux and the number of valid vertices in
tree-like tableaux.

Ballot tableaux introduced in Section \ref{sec:BTab} and tree-like tableaux
of shifted shapes have similar recursive structure as in the case 
of Dyck tableaux and tree-like tableaux \cite{ABDH11,ABN11,S19}.
The following theorem implies that the correspondence between Dyck tableaux
and tree-like tableaux can be generalized in case of ballot tableaux and 
tree-like tableaux of shifted shapes.
\begin{theorem}
\label{thrm:BTTT}
We have a bijection between ballot tableaux $B$ of shifted shapes and 
tree-like tableaux $T$ of shifted shapes satisfying the following properties.
\begin{enumerate}
\item The  number of labeled box in $T$ is one plus the number of labeled box in $B$. 
\item There is a ribbon between the boxes labeled by $i$ and $i+1$ with $1\le i\le n-1$ 
in $B$ if and only if there is a ribbon between the boxes labeled by $i$ and $i+1$ or 
if the box labeled by $i+1$ is below the box labeled by $i$ in $T$.
\item There is a ribbon starting from the box labeled by $i$ and ending at a terminal 
box in $B$ if and only if there is a ribbon starting from the box labeled by $i$
and ending at a box on the main diagonal in $T$. 
\item Let $M_{p}$ be the total number of proper shadow boxes in $B$.
Suppose that $\mathbf{h}=(h_1,\ldots,h_{n})$ be the insertion history for $T$.
When $h_{i}<h_{i-1}$ or $h_{i}$ is circled, we add a ribbon in $T$.
We denote by $R_{i}$ the number of added boxes in the ribbon in $T$ defined 
in Proposition \ref{prop:Rn}.
We define a positive integer $m_{i}$ when $h_{i}<h_{i-1}$ or $h_i$ is circled, 
and $m_{i}=0$ otherwise. 
We have three cases:
\begin{enumerate}
\item $h_{i-1}$ is boxed and $h_{i}$ do not have a circle.
We define $m_{i}:=h_{i-1}-h_{i}+1-R_{i}$.
\item $h_{i}$ is circled.
Let $n$ be the number of entries without a box and $n'$ be the number of 
entries with a box in the subsequence $(h_1,\ldots,h_{i})$.
We define $m_{i}:=2n-1-h_{i}-R_{i}$ for $n'=0$ and 
$m_{i}:=2n+n'-h_{i}-R_{i}$ for $n'\ge1$.
\item Both $h_{i}$ and $h_{i-1}$ do not have a circle.
We define $m_{i}:=h_{i-1}-h_{i}-R_{i}$.
\end{enumerate}
Then, we have 
\begin{align*}
M_{p}=\sum_{i}m_{i}.
\end{align*}
\end{enumerate}
\end{theorem}
\begin{proof}
Both a ballot tableau and a tree-like tableau of a shifted shape are constructed 
from a natural label of size $n$ and bijective to this natural label.
This induces a natural bijection between ballot tableaux and tree-like tableaux
of shifted shapes through natural labels of a tree.
We show that this bijection satisfies the four properties.

(1) Note that ballot tableaux have a box labeled by $0$. So, the number of labeled 
boxes in $B$ is one plus that of labeled boxes in $T$. 

(2) 
Let $\mathbf{h}:=(h_1,\ldots,h_n)$ be the insertion history of a tree-like tableau $T$, 
which is also the insertion history of the ballot tableau $B$.
We add a ribbon starting from the box labeled by $i$ and ending at the box 
labeled by $i-1$ in $T$ if and only if $h_{i}<h_{i-1}$.
The box labeled by $i$ is left to the box labeled by $i-1$ and there is no boxes
labeled by $j<i-1$ in-between the boxes labeled by $i-1$ and $i$ if and only if 
$h_{i}=h_{i-1}$. In this case, we have a ribbon between the boxes labeled by 
$i-1$ and $i$.

(3) We have a ribbon between the box labeled by $i$ and a terminal box in $B$
if and only if $h_{i}$ has a circle.
Similarly, we have a ribbon between the box labeled by $i$ and a box on the main 
diagonal if and only if $h_{i}$ has a circle.

(4) 
Let $L$ be a natural label and $L(m)$ be a natural label consisting of 
labels in $[1,m]$.
From Proposition \ref{prop:shadow}, a shadow box 
is characterized by a pattern $2^{+}12$. 
Suppose that $h_{i}$ is not circled.
Let $i$ and $j$ be integers such that 
$E(i)\rightarrow E(j)\rightarrow E(i-1)$ and $j<i$, which gives a pattern 
$2^{+}12$, where $E(p)$ is an edge labeled by $p$ in $L$.
Suppose $h_{i}$ is circled.
Let $i$ and $j$ be integers such that $E(i)\rightarrow E(j)$ with $j<i$.
Among such $i$ and $j$, a proper shadow box is above a box labeled by $j$ and 
the edge labeled by $j$ is connected to a leaf in $L(i)$.
The number of shadow boxes on the line connecting the box labeled by $i$
and the box labeled by $i-1$ in the ballot tableau for $L$ is 
equal to the number of valid vertices between the box labeled by $i$ 
and the box labeled by $i-1$ in the tree-like tableau of a shifted shape 
for $L$.
In the cases of (a), (b) and (c), the number $m_{i}$ gives the number of such 
valid vertices in the tree-like tableau.
It is obvious that the number of valid vertices is equal to the number of 
a partial tree $UD$ in the lower ballot path $\lambda$, which implies 
$M_p=\sum_{i}m_{i}$.
\end{proof}

\begin{remark}
From Proposition \ref{prop:Rn}, we have an expression of $R_{n}$ 
in terms of the structure of a natural label.
Thus, the value $m_{i}$ in Theorem \ref{thrm:BTTT}, 
which counts the number of proper shadow boxes,
can be expressed in terms of a natural label and its insertion 
history.

When a natural label has neither dotted edges and circled labels, 
the number $m_{i}$ defined in Theorem \ref{thrm:BTTT} is 
equal to the number of boxes added as a ribbon in $T$.
That is, we have $m_{i}=b_{i}$ as observed in \cite{ABDH11}.
This is because the insertion history $\mathbf{h}=(h_{1},\ldots,h_{n})$
and its symmetric tree-like tableau satisfy the following two properties: 
1) $h_{i}$ is an even integer, 
and 2) the number of valid vertices between the boxes 
labeled by $i$ and $i+1$ in the tree-like tableau is equal to the number of edges 
between the two boxes.
\end{remark}

\bibliographystyle{amsplainhyper} 
\bibliography{biblio}

\end{document}